\documentclass[reqno,11pt]{amsart}
\usepackage[utf8]{inputenc}
\usepackage{amsmath, latexsym, amsfonts, amssymb, amsthm, amscd, bm}
 \usepackage{graphics,epsf,psfrag}
\usepackage[dvipsnames]{xcolor}
\usepackage{mathrsfs,stmaryrd}
\usepackage{url}

\numberwithin{equation}{section}

\newtheorem{theorem}{Theorem}[section]
\newtheorem{lemma}[theorem]{Lemma}
\newtheorem{proposition}[theorem]{Proposition}

\newtheorem{rem}[theorem]{Remark}

\renewcommand{\tilde}{\widetilde}
\xdefinecolor{red}{named}{Red}

\newcommand{\cA}{{\ensuremath{\mathcal A}} }
\newcommand{\cB}{{\ensuremath{\mathcal B}} }

\newcommand{\cC}{{\ensuremath{\mathcal C}} }

\newcommand{\cL}{{\ensuremath{\mathcal L}} }

\newcommand{\cU}{{\ensuremath{\mathcal U}} }

\newcommand{\cM}{{\ensuremath{\mathcal M}} }
\newcommand{\cW}{{\ensuremath{\mathcal W}} }

\DeclareMathSymbol{\leqslant}{\mathalpha}{AMSa}{"36} 
\DeclareMathSymbol{\geqslant}{\mathalpha}{AMSa}{"3E} 
\DeclareMathSymbol{\eset}{\mathalpha}{AMSb}{"3F}     
\renewcommand{\leq}{\;\leqslant\;}                   
\renewcommand{\geq}{\;\geqslant\;}                   
\newcommand{\dd}{\,\text{\rm d}}             

\newcommand{\bbN}{{\ensuremath{\mathbb N}} }

\newcommand{\bbP}{{\ensuremath{\mathbb P}} }

\newcommand{\bbR}{{\ensuremath{\mathbb R}} }

\newcommand{\bbZ}{{\ensuremath{\mathbb Z}} }

\newcommand{\ga}{\alpha}

\newcommand{\gd}{\delta}
\newcommand{\gep}{\varepsilon}       

\newcommand{\gs}{\sigma}

\makeatletter
\def\captionfont@{\footnotesize}
\def\captionheadfont@{\scshape}

\long\def\@makecaption#1#2{%
  \vspace{2mm}
  \setbox\@tempboxa\vbox{\color@setgroup
    \advance\hsize-6pc\noindent
    \captionfont@\captionheadfont@#1\@xp\@ifnotempty\@xp
        {\@cdr#2\@nil}{.\captionfont@\upshape\enspace#2}%
    \unskip\kern-6pc\par
    \global\setbox\@ne\lastbox\color@endgroup}%
  \ifhbox\@ne 
    \setbox\@ne\hbox{\unhbox\@ne\unskip\unskip\unpenalty\unkern}%
  \fi
  \ifdim\wd\@tempboxa=\z@ 
    \setbox\@ne\hbox to\columnwidth{\hss\kern-6pc\box\@ne\hss}%
  \else 
    \setbox\@ne\vbox{\unvbox\@tempboxa\parskip\z@skip
        \noindent\unhbox\@ne\advance\hsize-6pc\par}%
\fi
  \ifnum\@tempcnta<64 
    \addvspace\abovecaptionskip
    \moveright 3pc\box\@ne
  \else 
    \moveright 3pc\box\@ne
    \nobreak
    \vskip\belowcaptionskip
  \fi
\relax
}
\makeatother
\def\writefig#1 #2 #3 {\rlap{\kern #1 truecm
\raise #2 truecm \hbox{#3}}}

\makeatletter
\newsavebox{\@brx}
\newcommand{\llangle}[1][]{\savebox{\@brx}{\(\m@th{#1\langle}\)}%
  \mathopen{\copy\@brx\kern-0.5\wd\@brx\usebox{\@brx}}}
\newcommand{\rrangle}[1][]{\savebox{\@brx}{\(\m@th{#1\rangle}\)}%
  \mathclose{\copy\@brx\kern-0.5\wd\@brx\usebox{\@brx}}}
\makeatother

\newcommand{\dist}{\text{dist}}

\title{Periodicity and longtime diffusion for mean field systems in $\bbR^d$}

\author{Eric Lu\c{c}on}
\address{Universit\'e de Paris, Laboratoire MAP5 (UMR CNRS 8145), 75270 Paris, France \& FP2M, CNRS FR 2036, \url{eric.lucon@u-paris.fr}.
}

\author{Christophe Poquet}
\address{Université Claude Bernard Lyon 1, CNRS UMR 5208, Institut Camille Jordan, F-69622 Villeurbanne, France, \url{poquet@math.univ-lyon1.fr}}

\keywords{Mean-field systems, McKean-Vlasov process, longtime behavior of Markov processes, nonlinear Fokker-Planck equation}
\subjclass[2010]{60K35, 35K55, 35Q84, 37N25, 82C31, 92B20}

\date{\today}

\begin{document}

\maketitle

\begin{abstract}
We study in this paper the longtime behavior of some large but finite populations of interacting stochastic differential equations whose (infinite population) limit Fokker-Planck PDE admits a stable periodic solution. We show that the empirical measure for the population of size $N$ stays close to the periodic solution, but with a random dephasing at the timescale $Nt$ that converges weakly to a Brownian motion with constant drift.
\end{abstract}

\section{Introduction}
\subsection{The particle system}
We are interested in this work in the large time behavior of the following systems of coupled stochastic differential equations on $ \mathbb{ R}^{ d}$ ($d\geq1$)
\begin{equation}
\label{eq:EDS_N}
d X_{i,t}= \left(\gd F(X_{i,t})-K\left(X_{i,t}-\frac{1}{N}\sum_{j=1}^NX_{j,t}\right)\right)\dd t +\sqrt{2} \gs  \dd B_{i,t},\ i=1,\ldots, N,\ t\geq0.
\end{equation}
Here, $N$ is the size of the population and $X_{ i, t}$ represents the state of the individual with index $i$ at time $t$. In \eqref{eq:EDS_N}, the dynamics of $X_{ i}$ is decomposed into some local dynamics part $ \delta F(X_{ i, t}) {\rm d}t$, where $F: \mathbb{ R}^{ d} \to \mathbb{ R}^{ d}$ is a smooth function (we will suppose that $F$ and all its derivatives are bounded) and $ \delta>0$ a scaling parameter, a linear interaction with the rest of the population, governed by a constant interaction matrix $K$, and an additive noise with a constant diffusion matrix $ \sigma$. Here $B_{ i}$, $i=1, \ldots, N$ is a collection of i.i.d. standard Brownian motions on $ \mathbb{ R}^{ d}$. A simplifying hypothesis that we adopt here is to suppose that $K$ and $ \sigma$ are diagonal with positive coefficients (that is both interaction and noise are nondegenerate w.r.t. all components). The system \eqref{eq:EDS_N} is of mean-field type, as the interactions occur on the complete graph: the interaction term in \eqref{eq:EDS_N} forces each $X_{ i}$ to align itself to the empirical mean value of the whole system $\frac{1}{N}\sum_{j=1}^NX_{j,t}$. 

\subsection{The mean-field limit}
As the size of the population goes to infinity, standard propagation of chaos results \cite{sznitman1991topics} shows that the natural limit of the particle system \eqref{eq:EDS_N} may be given in terms of the following nonlinear McKean-Vlasov process
\begin{equation}\label{eq:McKean}
d \bar X_t =\gd F(\bar X_t)\dd t - K\left(\bar X_t- \mathbf{ E}[\bar X_t]\right)\dd t+\sqrt{2}\gs \dd B_t.
\end{equation} 
The nonlinear character of \eqref{eq:McKean} lies in the fact that $\bar X_{ t}$ interacts with its own expectation $ \mathbb{ E} \left[\bar X_{ t}\right]$.
The proximity between \eqref{eq:EDS_N} and \eqref{eq:McKean} as $N\to \infty$ can be easily understood by a standard coupling argument: under mild assumptions on $F$, choosing the same Brownian motions for both \eqref{eq:EDS_N} and $(\bar X_{1}, \ldots, \bar X_{N})$ solving \eqref{eq:McKean},  Gr\"onwall type arguments lead to the following standard propagation of chaos estimate (for any fixed $T>0$):
\begin{equation}
\sup_{ i=1, \ldots, N}\mathbb{ E} \left[ \sup_{ t\in[0, T]}\left\vert X_{ i, t} -\bar X_{ i, t}\right\vert\right] \leq \frac{ C(T)}{ \sqrt{ N}}.
\end{equation}
Note that without further hypotheses on the system, the above constant $C(T)$ is exponential in $T$, so that the above estimate is only relevant for horizon time $T$ that are logarithmic in $N$. 

\subsection{Periodicity and large-time behavior}
At the level of the whole particle system, the previous propagation of chaos estimate boils down to the convergence as $N\to \infty$ of the empirical measure 
\begin{equation}
\label{eq:muNt}
\mu_{ N, t}:= \frac{ 1}{ N}\sum_{j=1}^{ N} \delta_{ X_{ j,t}}
\end{equation}
 to $ \mu_{ t}$, weak solution to the following nonlinear PDE
\begin{equation}\label{eq:PDE non centered}
\partial_t \mu_t =  \nabla\cdot \left(\gs^2 \mu_t\right) +\nabla \cdot \left(K\left(x-\int_{\bbR^d}y \mu_t({\rm d}y)\right) \mu_t\right)-\gd \nabla\cdot \left( F(x) \mu_t\right).
\end{equation}
Here, $ \mu_{ t}$ is the law of the nonlinear process $\bar X_{ t}$ given by \eqref{eq:McKean}. The point of the present paper is to analyse the proximity of \eqref{eq:muNt} w.r.t. its mean-field limit \eqref{eq:PDE non centered} on a time horizon that goes beyond the time scale of order $\log N$. A first step in this direction is of course to understand the behavior as $t\to\infty$ of the mean-field system \eqref{eq:PDE non centered} itself.  The long-time dynamics of \eqref{eq:McKean}-\eqref{eq:PDE non centered} is a longstanding issue in the literature. In particular, the existence of stable equilibria for \eqref{eq:PDE non centered} (that is invariant measures for \eqref{eq:McKean}) has been studied for various choices of dynamics, interaction and regimes of parameters $ \delta, K, \sigma$. A prominent example in this direction concerns the case of granular media equation where $F= - \nabla V$ for some some confining potential $V$ and the interaction is given by a more general potential $W$ (we consider here the case $W$ quadratic). Uniform in time propagation of chaos results have been studied in the literature for this situation, see e.g. \cite{10.1214/12-AOP749,MR1632197,Cattiaux2008,DEGZ2020} for further details and references. Note that most of the mentioned results concern situations where the corresponding particle dynamics \eqref{eq:EDS_N} is reversible. We present in this paper a non-reversible situation where uniform propagation of chaos does not occur.

\medskip

We are concerned in this work with the (intrinsically non-reversible) situation when \eqref{eq:PDE non centered} has a periodic orbit. The emergence of periodicity for mean-field systems is a universal phenomenon, as it reflects a common feature of self-organization in coupled systems. This emergence has been observed in different fields, as chemical reactions (Brusselator model \cite{scheutzow1986periodic}), neuroscience \cite{Giacomin:2012,giacomin2015noise,22657695,MR3392551,ditlevsen2017multi,Lucon:2018qy, Quininao:2020,2018arXiv181100305L,Cerf:2020,Cormier2020} or statistical physics (e.g. spin-flip models \cite{Collet:2015,DaiPra2020}). 
An example of particular interest is given by the FitzHugh-Nagumo model \cite{MR1779040,22657695}, as it is commonly used as a prototype for  excitability in both neuronal models \cite{LINDNER2004321} and physics \cite{PhysRevE.68.036209}. Here excitability refers to the ability for a neuron to emit spikes in the presence of perturbations (due to some noise and/or external input) while this neuron stays at rest in absence of perturbation. Note that for excitable systems a main point of interest is to understand the influence of noise and interaction in the emergence and stability of periodic solutions. It is now well known that some balance has to be found in the quantity of noise and interaction present in the system in order to obtain oscillations (see \cite{LINDNER2004321,Lucon:2018qy,2018arXiv181100305L} for further details).
For the FizHugh-Nagumo case, the long-time dynamics of \eqref{eq:PDE non centered} has been the subject of several previous works, with aim either the existence of equilibria \cite{Mischler2016,Quininao:2020} or of periodic solutions \cite{Lucon:2018qy,2018arXiv181100305L}), depending on the choice of parameters.

We refer here to the companion paper of this work \cite{LP2021a} where, under some hypotheses that we will describe below (see Section~\ref{sec:periodic solution}), the existence of periodic solutions to nonlinear equation such as \eqref{eq:PDE non centered} is proven, together with stability and regularity results of the solution $ \mu_{ t}$ in a neighborhood of the periodic orbit. This result is obtained via a perturbation argument (i.e. we take the parameter $\gd$ in \eqref{eq:EDS_N} small), considering \eqref{eq:PDE non centered} as a slow-fast system (more details on this point will be given in Section~\ref{sec:periodic solution}), relying on normally hyperbolic manifolds arguments \cite{Bates:1998,Bates:2008}. The starting point of the present work will be to assume that \eqref{eq:PDE non centered} possesses a regular periodic solution $\left(\Gamma^\gd_t\right)_{t\geq 0}$ with period $T_{ \delta}>0$.

The main issue of the present work will be to question the proximity of the empirical measure \eqref{eq:muNt} w.r.t. the periodic orbit $ \Gamma^\gd$ of the mean-field PDE \eqref{eq:PDE non centered}. The main result will be that, as long as $ \mu_{ N, 0}$ is initially sufficiently close to $ \Gamma^\gd$, on a time scale proportional to $N$, $ \mu_{ N, tN}$ remains close to $ \Gamma^\gd$ whereas its phase (parameterized by the isochron map $ \Theta^\gd$ defined in a neighborhood of $ \Gamma^\gd$, see Section~\ref{sec:periodic solution} below for precise details) performs a diffusion with constant drift and diffusion coefficient that are explicit in $ \Theta^\gd$. The proof relies on precise proximity estimates of $ \mu_{ N, t}$ around $ \Gamma^\gd$ expressed in terms of weighted Sobolev norms that are adapted to the linear part of system \eqref{eq:EDS_N}. 

\subsection{Related works}
We are interested here in an infinite-dimensional version of a problem that has been widely studied in the literature in the finite dimensional case: the long time behavior of differential equations admitting a stable periodic solution, and perturbed by a Brownian noise of size $\gep$. It is well known that at the timescale $\gep^ {-2}t$ this system stays with high probability close to periodic solution, but with a random dephasing given by a diffusion \cite{Yoshimura2008,GPS2018}. The intuition here is that the stability of the limit cycle prevents the perturbed solution from escaping, but the solution is free to diffuse along the tangent space of the limit cycle. Moreover, when one chooses the isochron map to define the phase of the system (the phase dynamics in the neighborhood of the periodic solution being then simply a rotation with constant speed), the dephasing is given by a Brownian motion with a constant drift.

Note that infinite-dimensional versions of this problem where already considered in the case of the Cahn-Allen SPDE with bi-stable symmetric potential \cite{DMP1995,BBDMP1998}. In this case the dynamics involves a sable manifold of stationary points instead of a limit cycle, and the system is perturbed by a white noise, but the limit dynamics of the phase after re-scaling is also a diffusion. In contrast to these previous works we do not perturb the mean-field PDE \eqref{eq:PDE non centered} by adding a macroscopic noise term, we deal with the {\sl natural} microscopic noise already present in the finite population $(X_{i,t})_{i=1,\ldots,N}$ in \eqref{eq:EDS_N}, which becomes macroscopically apparent on a diffusive time scale: when time is proportional to $N$, the dynamics of the empirical measure $\mu_{N,t}$ defined in \eqref{eq:muNt} can be considered as a noisy perturbation of the PDE \eqref{eq:PDE non centered}. 

This finite-population effects have already been studied in the particular case of the stochastic Kuramoto model, defined as follows \cite{kuramoto2012chemical,acebron2005kuramoto}:
\begin{equation}
\label{eq:Kur}
{\rm d}X_{ i, t} = \delta \omega_{ i} \dd t - \frac{ K}{ N} \sum_{ j=1}^{ N} \sin \left( X_{ j, t}- X_{ i, t}\right) {\rm d}t + \sigma {\rm d}B_{ i, t}.
\end{equation}
Here, each $X_{ i, t}$ is a phase living on the circle $\mathbb{ S}:= \mathbb{ R}/2\pi$ that is subject to noise, interaction and disorder: the intrinsic dynamics of $X_{ i,t}$ reduces to a simple rotation with frequency $ \omega_{ i}\in \mathbb{ R}$. These choices of intrisic dynamics and interaction term make this model invariant by rotation. From a mathematical perspective, \eqref{eq:Kur} has the distinct advantage of being explicitly solvable: one can compute the stationary solutions of the corresponding Fokker-Planck PDE and obtain precise stability estimates concerning the manifold of stationary solutions \cite{GPP2012}. Concerning the long-time dynamics of the empirical measure of \eqref{eq:Kur}, in the non-disordered case ($ \omega_{ i}\equiv 0$, which is the only case when \eqref{eq:Kur} is reversible), it has been shown in \cite{dahms2002long,Bertini:2013aa} that in the supercritical case, that is when the limit PDE admits a non-trivial circle of stationary solutions, the empirical measure of \eqref{eq:Kur} performs a Brownian motion on this curve on a time scale of order $N$. Note that in this case uniform propagation of chaos for the radial dynamics has been recently obtained in \cite{delarue2021uniform}, and that the empirical measure stays also close to this curve for long times in the case of interactions defined by a Erdös-Renyi graph sufficiently dense \cite{Coppini2019}.
Moreover, the presence of nontrivial disorder induces a drift that is observable on a time scale of order $ \sqrt{ N}$ \cite{Lucon:2017}. 

One can see the present paper as a natural continuation of these works to more general systems, that in particular no longer live on the compact $ \mathbb{ S}$ but on the whole $ \mathbb{ R}^{ d}$. In itself, this generalisation is a major complication to the analysis: first, the existence of periodic orbits is no longer intrinsic to the model and therefore requires a proof. We refer to \cite{Lucon:2018qy} and the companion paper \cite{LP2021a} where we give rigorous proof of existence of limit cycles for \eqref{eq:PDE non centered}. Secondly, the fact that the state space is no longer compact poses serious technical issues that were not present in the case of the Kuramoto model, particularly when the term $F$ is unbounded (and this is the case in many interesting models, such as for FitzHugh-Nagumo of Morris-Lecar). We refer to \cite{Lucon:2018qy} where these issues were addressed for the analysis of the existence of periodic solutions for the PDE \eqref{eq:PDE non centered}. Nonetheless, we restrict ourselves in this work to the case where $F$ is bounded as well as all its derivatives: the stability results for the periodic solutions we obtained in \cite{Lucon:2018qy} in the unbounded case are not sufficient to study the long time dynamics of the empirical measure, while for the bounded case we obtained more satisfying results in the companion paper \cite{LP2021a}, relying on the theory of approximately invariant manifolds \cite{Bates:2008}. Note moreover that the existence and regularity of the isochron map is particularly well-covered in the literature in the finite dimension case \cite{guckenheimer1975isochrons,hirsch1977invariant} but similar existence and regularity results in infinite dimension is less known in the literature. One of the purpose of the companion paper \cite{LP2021a} is precisely to obtain the desired existence and regularity estimates concerning the isochron in our case.

\section{Setting and main results}

\subsection{Reformulating the model}
For the analysis of \eqref{eq:EDS_N}, it will be convenient to distinguish between the empirical mean of the system and the re-centered process: define
\begin{align}
Y_{i,t}=X_{i,t}-m_{N,t},\ i=1,\ldots, N
\end{align}
where
\begin{equation}
m_{N,t}=\frac{1}{N}\sum_{j=1}^N X_{j,t},\ t\geq 0.
\end{equation} 
Reformulating \eqref{eq:EDS_N} in terms of $(Y_{ i}, _{ N})$ we obtain equivalently
\begin{equation}
\label{eq:eds_Yi}
\begin{cases}
\dd Y_{i,t}=&\left(\gd\left( F_{m_{N,t}}(Y_{i,t})-\frac{1}{N}\sum_{j=1}^N  F_{m_{N,t}}(Y_{j,t})\right)-K\ Y_{i,t}\right)\dd t\\
&+\sqrt{2} \gs  \dd B_{i,t}-\frac{\sqrt{2}}{N}\gs \sum_{j=1}^N \dd B_{j,t}\\
d m_{N,t} =& \frac{\gd}{N}\sum_{j= 1}^N F_{m_{N,t}}(Y_{j,t})\dd t +\frac{\sqrt{2}}{N}\gs \sum_{j=1}^N \dd B_{j,t}.
\end{cases}
\end{equation}
Therefore, the knowledge of the empirical measure defined in \eqref{eq:muNt} is exactly equivalent to the knowledge of
\begin{equation}
\label{eq:muNt2}
\mu_{N,t}:= \left(p_{N, t}, m_{ N, t}\right),\ t\geq0,
\end{equation}
where
\begin{equation}
p_{ N, t}:= \frac{ 1}{ N}\sum_{k=1}^N \delta_{Y_{i,t}},\ t\geq0, N\geq0.
\end{equation}
Note here that we make a minor abuse of notations in identifying \eqref{eq:muNt} with \eqref{eq:muNt2}. One obtains from the same standard propagation of chaos estimates that, on each time interval $[0,T]$ the couple $\mu_{N,t}$ converges weakly to $\mu_t=(p_t,m_t)$ solution to the system
\begin{equation}\label{eq:syst PDE}
\left\{
\begin{array}{rl}
\partial_t p_t &= \nabla\cdot (\gs^2 \nabla p_t)+\nabla \cdot(p_t Kx) -\gd \nabla\cdot \left(p_t \left( F_{m_t}-\int F_{m_t} \dd p_t\right)\right)\\
\dot m_t &=  \gd\int F_{m_t} \dd p_t 
\end{array}
\right. ,
\end{equation}
with $F_u(x)=F(x+u)$. 
We denote $t\mapsto T^t$ the semi-flow associated to the system \eqref{eq:syst PDE}. We will make hypotheses on $F,K,\gs$ such that this system admits a stable limit cycle $(q^\gd_t,\gamma^\gd_t)_{t\in[0,T_\gd]}$.

\subsection{Main hypotheses and notations}
\subsubsection{Main assumptions}
Throughout the paper, we assume that $F: \mathbb{ R}^{ d}\to \mathbb{ R}^{ d}$ is $ \mathcal{ C}^{ \infty}$ and bounded as well as all its derivatives. Suppose also $ \delta>0$ and that $K$ and $ \sigma$ are diagonal matrices with positive coefficients: $K= diag(k_{ 1}, \ldots, k_{ d})$, $ \sigma= diag(\sigma_{ 1}, \ldots, \sigma_{ d})$ with $k_{ i}>0$ and $ \sigma_{ i}>0$.

\subsubsection{Weighted Sobolev norms}
Let $A$ be a positive matrix. Denote by
\begin{equation}
\left\vert x \right\vert_{ A}= \left( x\cdot A x\right)^{ 1/2}
\end{equation}
the Euclidean norm twisted by $A$. For any $ \theta\in \mathbb{ R}$, define
\begin{equation}
w_{ \theta}(x)= \exp \left( -\frac{ \theta}{ 2} \left\vert x \right\vert_{ K \sigma^{ -2}}^{ 2}\right),
\end{equation}
and denote by $L^{ 2}_{ \theta}$ the $L^{ 2}$-space with weight $ w_{ \theta}$ that is associated with the scalar product
\begin{equation}
\label{eq:norm_Lw}
\left\langle f,g \right\rangle _{ L^{ 2}_{ \theta}}=
 \int_{ \mathbb{ R}^{ d}}  f(x) g(x) w_{ \theta}(x) {\rm d}x.
\end{equation}
For $ \theta>0$, we consider 
the scalar products
\begin{equation}
\label{eq:prod_Hr}
\langle f,g\rangle_{H^r_\theta} = \left\langle f, (1-\cL_\theta^{ \ast})^r g\right\rangle_{L^2_{ \theta}},
\end{equation}
where $\cL_\theta^\ast$ is the operator
\begin{equation}
\label{eq:L_OU}
\cL_{ \theta}^{ \ast} f= \nabla\cdot (\gs^2\nabla f)- \theta Kx \cdot\nabla f,
\end{equation}
and we denote $H^r_{ \theta}$ the completion of the space of smooth function $f$ satisfying $\Vert f\Vert_{H^r_{ \theta}}<\infty$. For $r\in \bbN$ the norm $\Vert f\Vert_{H^r_{  \theta}}$ is equivalent to the norm (see \cite{LP2021a} for a proof)
\begin{equation}
\label{eq:equiv_normHr}
\sqrt{\sum_{(i_1,\ldots,i_d)\in \{0,\ldots,r\}^d, \, i_1+\ldots+i_d\leq r} \left\Vert \partial^{i_1}_{x_1}\ldots \partial^{i_d}_{x_d} f\right\Vert_{L^2_{ \theta}}^2}.
\end{equation}
We define $H^{-r}_\theta$ the dual of $H^r_\theta$.
We rely on a ``pivot" space structure: since the natural injection $H^r_\theta\rightarrow L^2_\theta$ is dense, we have a dense injection $(L^2_\theta)'\rightarrow H^{-r}_\theta$. Using the identification of $(L^2_\theta)'$ with $L^2_{-\theta}$, (which means identifying $\langle\cdot,\cdot\rangle_{(L^2_\theta)'\times L^2_\theta}$ with $\langle\cdot,\cdot\rangle_{L^2}=\langle\cdot,\cdot\rangle$), we obtain, for all $u\in L^2_{-\theta}\subset H^{-r}_\theta$ and all $f\in H^r_\theta$,
\[
\langle u, f\rangle_{H^{-r}_\theta\times H^r_\theta} = \langle u,f\rangle.
\]
Remark moreover that for $u\in H^{-r}_\theta$ and $f\in H^{r+1}_\theta$ we have
\begin{equation}
\left|\langle \nabla u ,f\rangle\right|=\left|\langle u ,\nabla f\rangle\right|\leq C \Vert u\Vert_{H^{-r}_\theta}\Vert f\Vert_{H^{r+1}_\theta},
\end{equation}
so that
\begin{equation}\label{eq: nabla u H'}
\Vert \nabla u\Vert_{ H^{-(r+1)}_\theta}\leq C \Vert u\Vert_{H^{-r}_\theta}.
\end{equation}

The fact that the empirical measure $p_{N,t}$ belongs to $H^{-r}_\theta$ for $r$ sufficiently large is provided by the following Lemma (see Appendix~\ref{app:weighted sobolev} for its proof).
\begin{lemma}
\label{lem:control_delta_x}
Fix $ \theta>0$ and $r> d/2$. For all $ x\in \mathbb{ R}^{ d}$, the linear form $h \mapsto \left\langle \delta_{ x}\, ,\, h\right\rangle= h(x)$ is continuous on $ H^{ r}_{ \theta}$ and we have the following bound: for all $ \eta>0$ there exists a constant $C_{ \eta, r}>0$ such that for all $x\in \mathbb{ R}^{ d}$, $ \left\Vert \delta_{ x} \right\Vert_{H^{-r}_{ \theta}} \leq C_{ \eta, r} \exp \left(\frac{ \theta \left\vert x \right\vert^{ 2}_{ K \sigma^{ -2}}}{ 4- \eta}\right)$.
\end{lemma}
The space $H_{ \theta}^{ -r}$ is the proper functional space for the process $p_{ N, t}$. In order to understand $ \mu_{ N, t}$ given by \eqref{eq:muNt2}, define finally the space
\begin{equation}
\label{eq:Hr}
\mathbf{ H}^{ r}_{ \theta}:= H_{ \theta}^{ r}\times \mathbb{ R}^{ d},
\end{equation}
endowed with the scalar product
\begin{equation}
\left\langle (f,m)\, ,\, (g,m^{ \prime})\right\rangle_{ \mathbf{ H}^{ r}_{ \theta}}:= \left\langle f\, ,\, g\right\rangle_{ H_{ \theta}^{ r}} + m\cdot m^{ \prime},
\end{equation}
and $\mathbf{H}_\theta^{-r}$ the dual of $\mathbf{H}_\theta^{r}$. Relying on the same ``pivot" space structure as above, we have
\begin{equation}
\llangle[\big] (\eta, h);(f, m)\rrangle[\big]_{\mathbf{H}_\theta^{-r}\times\mathbf{H}_\theta^{r}} = \llangle[\big] (\eta, h);(f, m)\rrangle[\big],
\end{equation}
where
\begin{equation}
\llangle[\big] (\eta, h);(f, m)\rrangle[\big]= \langle \eta, f\rangle + h\cdot m.
\end{equation}
It is proved in \cite{LP2021a}, relying on classical arguments \cite{Henry:1981,sell2013dynamics}, that \eqref{eq:syst PDE} admits mild solutions in $\mathbf{H}_\theta^{-r}$.

\subsection{Slow-fast viewpoint and periodic solution for the infinite population system}\label{sec:periodic solution}

In this section, we review briefly the results obtained in the companion paper \cite{LP2021a}. The idea, that was already used in \cite{Lucon:2018qy,2018arXiv181100305L}, to obtain the existence of periodic solutions to \eqref{eq:PDE non centered} is to see \eqref{eq:syst PDE} as a slow-fast system. Indeed, for $\gd=0$, the solution of this system is $(e^{t\cL}p_0,m_0)$, where $\cL p = \nabla\cdot (\gs^2 \nabla p)+\nabla \cdot(p Kx) $ is simply an Ornstein Uhlenbeck operator. It is well known that in this case $e^{t\cL}p_0$ converges exponentially fast to the density of the Gaussian distribution on $\bbR^d$ with mean $0$ and variance $\gs^2 K^{-1}$, that we denote $\rho$. So, recalling \eqref{eq:syst PDE}, when $\gd$ is small, $m_t$ satisfies at first order
\begin{equation}
\dot m_t\approx  \gd \int_{\bbR^d} F(x)q(x-m_t) \dd x.
\end{equation}
With this heuristic in view, in order to obtain a periodic behavior for the PDE \eqref{eq:syst PDE} when $\gd$ is small, we will suppose that the ordinary differential equation
\begin{equation}\label{eq:zt}
\dot z_t = \int_{\bbR^d} F(x)\rho(x-z_t) \dd x = \int_{\bbR^d} F_{z_t}(x)\rho(x) \dd x
\end{equation}
admits a stable periodic solution $(\alpha_t)_{t\in[0,T_\ga]}$. More precisely, we will express the stability hypothesis in terms of Floquet formalism: let us denote $t\mapsto \phi_t$ the flow defined by \eqref{eq:zt}, and $\pi_{u+t,u}$ the principal matrix solution associated to $\ga$, that is defined by $\pi_{u+t,u} n = D \phi_t(\ga_u) \cdot n$, and thus the solution to
\begin{equation}
\partial_ t \pi_{u+t,u} = \left( \int_{\bbR^d} DF_{\ga_{u+t}}(x)\rho(x) \dd x\right) \pi_{u+t,u}.
\end{equation}
We suppose that $\pi_{u+t,u}$ satisfies the decomposition
\begin{equation}
\left|\pi_{u+t,u} P^s_u \right| \leq C_\ga e^{-\lambda_\ga t}\left|n\right| ,\text{ and } c_\ga \left|n\right|\leq   \left|\pi_{u+t,u} P^c_u \right|\leq C_\ga \left|n\right|,
\end{equation}
where $P^c_u$, $P^s_u$ are projection satisfying $P^c_u+P^s_u=I_d$, with $u\mapsto P^c_u$ and  $u\mapsto P^s_u$ are smooth, and $c_\ga$, $C_\ga$ and $\lambda_\ga$ are positive constants. $\pi_{u+t,u}$ characterizes the linearized dynamics around $\ga$, and these hypotheses mean that it is a contraction on a supplementary space of the tangent space to $\ga$.

As an example, as stated in \cite{Lucon:2018qy,LP2021a} the existence of a stable periodic solution for \eqref{eq:zt} occurs in the case of the FitzHugh Nagumo model, defined in dimension $d=2$ by
\begin{equation}\label{def:F}
F(x,y) = \left(x-\frac{x^3}{3}-y,\frac1c\left(x+a-by \right) \right),
\end{equation}
for an accurate choice of parameters. In particular, taking $a=\frac13$, $b=1$ and $c=10$, the local dynamics $\dot x_t = \gd F(x_t)$ has globally attracting fixed-point, while for $\frac{\gs_1^2}{K_1}$ is not too large (for example with value $0.2$) the system defined by \eqref{eq:zt} admits a stable periodic solution.
We refer to \cite{Lucon:2018qy} for a detailed study of the bifurcations in this model when tuning the value of  $\frac{\gs_1^2}{K_1}$ and a discussion on the periodic behaviors induced by the common effect of noise and interaction occuring in this model. Of course the function $F$ defined in \eqref{def:F} is not bounded, and thus cannot be considered as is in this work, but one can consider a smooth cut-off of $F$ such that \eqref{eq:zt} still possesses a stable periodic solution. Take for example $F(x)\psi(\gep|x|)$ with $\gep $ small and $\psi:\bbR^+\rightarrow \bbR^+$ smooth, non-increasing and that satisfies $\psi(t)=1$ for $t\leq 1$ and $\psi(t)=0$ for $t=2$. For $\gep$ small enough, \eqref{eq:zt} admits a stable periodic solution (see \cite{LP2021a} for more details).

With these hypotheses, it is proved in \cite{LP2021a}, relying on the deep approximately invariant manifold results provided by \cite{Bates:2008}, that the PDE \eqref{eq:syst PDE} admits a stable periodic solution. More precisely we have the following Theorem (corresponding to Theorem 1.4 in \cite{LP2021a}), in which the stability of the periodic solution $\Gamma^\gd_t$ is expressed using the linearized semi-group $\Phi_{u+t,u+s}$ and the projections $\Pi^{\gd,c}_u, \Pi^{\gd,s}$, which have the same role as $\pi_{u+t,u}$ and $P^c_u,P^s_u$ respectively for $\ga_t$. 

\begin{theorem}\label{th:Gamma}
Suppose that \eqref{eq:zt} admits a stable periodic solution. Then there exists $r_0>0$ such that for all $r\geq r_0$ and $\theta\in (0,1]$ there exists $\gd_0>0$ such that for all $\gd \in  (0,\gd_0)$ the system \eqref{eq:syst PDE} admits a periodic solution  $\left(\Gamma_{ t}^{ \delta}\right)_{ t\in [0, T_{ \delta}]}:=(q^\gd_t,\gamma^\gd_t)_{t\in [0,T_\gd]}$ in $ \mathbf{ H}^{-r+2}_\theta$ with period $T_\gd>0$. Moreover $q^\gd_t$ is a probability distribution for all $t\geq 0$, and $t\mapsto \partial_t\Gamma^\gd_t$ and $t\mapsto \partial^2_t \Gamma^\gd_t$ are in $C([0,T_\gd),\mathbf{H}^{-r+2}_\theta)$.

Denoting $\cM^\gd=\{\Gamma^\gd_t:\, t\in [0,T_\gd)\}$ and $
\Phi_{u+s,u}(\nu) = D T^s(\Gamma^\gd_u)[\nu]
$, there exists families of projection $\Pi^{\gd,c}_u$ and $\Pi^{\gd,s}_t$ that commute with $\Phi$, i.e. that satisfy
\begin{equation}\label{eq:Phi commutes}
\Pi^{\gd,\iota}_{u+t}\Phi_{u+t,u}=\Phi_{u+t,u}\Pi^{\gd,\iota}_{u},\quad \text{for } \iota = c,s.
\end{equation}
Moreover $\Pi^{\gd,c}_t$ is a projection on the tangent space to $\cM^\gd$ at $\Gamma^\gd_t$, $\Pi^{\gd,c}_t+\Pi^{\gd,s}_t=I_d$, $t\mapsto \Pi^{\gd,c}_t \in C^1([0,T_\gd),\cB(\mathbf{H}^{-r}_\theta))$, and there exist positive constants $c_{\Phi,\gd}$, $C_{\Phi,\gd}$ and $\lambda_{\gd}$ such that 
\begin{equation}
c_{\Phi,\gd}\left\Vert \Pi^{\gd,c}_{u}(\nu)\right\Vert_{\mathbf{H}^{-r}_\theta}\leq \left\Vert \Phi_{u+t,u}\Pi^{\gd,c}_{u}(\nu)\right\Vert_{\mathbf{H}^{-r}_\theta}
\leq C_{\Phi,\gd}\left\Vert \Pi^{\gd,c}_{u}(\nu)\right\Vert_{\mathbf{H}^{-r}_\theta},
\end{equation}
\begin{equation}\label{eq:Phi contracts}
\left\Vert \Phi_{u+t,u}\Pi^{\gd,s}_{u}(\nu)\right\Vert_{\mathbf{H}^{-r}_\theta}\leq C_{\Phi,\gd}\, t^{-\frac{\ga}{2}}e^{-\lambda_{\delta} t}\left\Vert \Pi^{\gd,s}_{u}(\nu)\right\Vert_{\mathbf{H}^{-(r+\ga)}_\theta},
\end{equation}
and
\begin{equation}\label{eq:Phi bounded}
\left\Vert \Phi_{u+t,u}\nu\right\Vert_{\mathbf{H}^{-r}_\theta}\leq C_{\Phi,\gd}\left(1+t^{-\frac{\ga}{2}}\right)e^{-\lambda_{\delta} t}\left\Vert \nu\right\Vert_{\mathbf{H}^{-(r+\ga)}_\theta}.
\end{equation}
\end{theorem}

Moreover in \cite{LP2021a} the existence of an isochron map is also proved, as stated in the following Theorem (corresponding to Theorem 1.6 in \cite{LP2021a}). The $ \mathcal{ C}^2$-regularity of this isochron map will be useful when applying Itô's Lemma at the end of this paper.

\begin{theorem}\label{th:Theta}
Under the same hypotheses as in Theorem~\ref{th:Gamma}, and for $r$ and $\gd$ as in Theorem~\ref{th:Gamma}, there exist a neighborhood $\cW^\gd\in \mathbf{H}^{-r}_{\theta}$ of $\cM^\gd$ and a $ \mathcal{ C}^2$-mapping $\Theta^\gd:\cW^\gd\rightarrow \bbR/T_\gd\bbZ$ that satisfies, for all $\mu \in \cW^\gd$, denoting $\mu_t=T^t \mu$,
\begin{equation}
\Theta^\gd(\mu_t)=\Theta^\gd(\mu)+t \quad \text{mod } T_\gd,
\end{equation}
and there exists a positive constant $C_{\Theta,\gd}$ such that, for all $\mu\in \cW^\gd$ with $\mu_t=T^t\mu$,
\begin{equation}
\left\Vert \mu_t - \Gamma^\gd_{\Theta^\gd(\mu)+t}\right\Vert_{\mathbf{H}^{-r}_\theta}\leq C_{\Theta,\gd} e^{-\lambda_\gd t}\left\Vert \mu - \Gamma^\gd_{\Theta^\gd(\mu)}\right\Vert_{\mathbf{H}^{-r}_\theta}.
\end{equation}
Moreover $\Theta^\gd$ satisfies, for all $\mu \in \cW^\gd$,
\begin{equation}
\left\Vert D^2\Theta^\gd(\mu) - D^2\Theta^\gd\left(\Gamma^\gd_{\Theta^\gd(\mu)}\right)\right\Vert_{\cB\cL (\mathbf{H}^{-r}_\theta)}\leq C_{\Theta,\gd} \left\Vert \mu - \Gamma^\gd_{\Theta^\gd(\mu)}\right\Vert_{\mathbf{H}^{-r}_\theta}.
\end{equation}
\end{theorem}

To simplify the notations we will denote $\Gamma_t=\Gamma^\gd_t$ and $\Theta=\Theta^\gd$ in the rest of the article.

\subsection{Main result}
\label{sec:main_result}

We fix a time $t_f>0$ and we will be interested in the behavior of the solution to \eqref{eq:syst PDE} up to time $Nt_f$. Let us fix some constants
\begin{equation}\label{hyp gamma xi}
0 <\gamma<1, \quad 0<\xi <\frac12.
\end{equation}
Consider now the integer $m$ defined as
\begin{equation}\label{hyp m}
m= \left\lfloor \max\left(\frac{1}{2\xi},4\right)\right\rfloor +1,
\end{equation}
and fix
\begin{equation}\label{hyp:theta}
0 <\theta < \frac{3\gamma}{4m}
\end{equation}
Consider finally $r$ satisfying
\begin{equation}\label{hyp:r}
r \geq 6+ \max \left(\frac{ d}{ 2}, r_{ 0}\right),
\end{equation}
where $r_0$ is given in Theorem~\ref{th:Gamma}. Let us comment on these parameters: an important step in the following will be to show that the empirical measure $ \mu_{ N, t}$ remains at a distance at most $ N^{ - \frac{ 1}{ 2} + \xi}$ from the invariant manifold defined by $ \Gamma$. A crucial tool for this will be to control the moments of order $m$ (given by \eqref{hyp m}) of $ \mu_{ N}$ in $ \mathbf{ H}_{ \theta}^{ -r}$ for $ \theta$ and $r$ given by \eqref{hyp:theta} and \eqref{hyp:r} respectively.

Define the Gaussian density
\begin{equation}
g_N (x)= \frac{1}{(2\pi)^\frac{d}{2}\left(1-\frac1N\right)^d \sqrt{\det(K^{-1}\gs^2)}}w_{\frac{N}{N-1}}(x).
\end{equation}
We can now state the main result of this paper.

\begin{theorem}\label{th:main}
Suppose that the hypotheses of Theorem~\ref{th:Gamma} are satisfied, and let $(r, \gamma, \theta)$ be fixed as in \eqref{hyp gamma xi}, \eqref{hyp:r} and \eqref{hyp:theta}. There exists a $\gd_1\in (0,\gd_0)$ such that for all $\gd\in (0,\gd_1)$ the following holds: if we suppose that
\begin{equation}
\label{hyp:control_pN0}
\sup_{N\geq 1} \mathbf{E}\left[\langle p_{N,0}, w_{-\gamma}\rangle \right] <\infty,
\end{equation}
that there exists a constant $\kappa_0$ such that
\begin{equation}
\mathbf{P}\left(\Vert p_{N,0}-g_N\Vert_{H^{-r+2}_\theta}\leq \kappa_0 \right) \underset{N\rightarrow\infty}{\longrightarrow} 1,
\end{equation}
and that there exists $u_0$ such that for all $\gep>0$,
\begin{equation}
\bbP\left( \left\Vert \mu_{N,0}-\Gamma_{u_0} \right\Vert_{\mathbf{H}^{-r}_\theta}\leq \gep\right)\underset{N\rightarrow\infty}{\longrightarrow} 1,
\end{equation}
then for all $\gep>0$ we have
\begin{equation}
\bbP\left(\sup_{t\in [0,t_f]} \left\Vert \mu_{N,N t}-\Gamma_{u_{0}+Nt+v_{N,t}} \right\Vert_{\mathbf{H}^{-r}_\theta}\leq \gep\right)\underset{N\rightarrow\infty}{\longrightarrow} 1,
\end{equation}
where the random process $v_{N,t}$ satisfies $v_{N,0}=0$ and converges weakly to $v_t=bt+a^2 w_t$, where $w$ is a standard Brownian motion and $b$ and $a$ are constant coefficients that depend on $\Gamma$ and the first two derivatives of $\Theta$ taken on $(\Gamma_t)_{u\in [0,T_\gd]}$.
\end{theorem}

\begin{rem}
\begin{enumerate}
\item Explicit formula for $b$ and $a^2$ are given in \eqref{eq: formula b} and \eqref{eq: formula a2} respectively.
\item For simplicity, the result of Theorem~\ref{th:main} is proved for initial conditions close to a point $\Gamma^\gd_{u_0}$ of the stable periodic solution, but as it was done in \cite{Bertini:2013aa,Lucon:2017} one could start close to a point lying in the basin of attraction of the periodic solution, and show that the empirical measures first reaches a neighborhood of size $\gep$ of the periodic solution (after a time interval of length independent from $N$). 
\end{enumerate}
\end{rem}

\subsection{Structure of the paper}

We will first prove in Section~\ref{sec:first bound} a uniform bound for $\left\Vert p_{N,t}\right\Vert_{H_\theta^{-r+2}}$. Then in Section~\ref{sec:close to M} we will show  that after a time of order $\log N$ the empirical measure $\mu_{N,t}$ stay with high probability at most at a distance of order $N^{-\frac12+\xi}$.
Due to a lack of uniform in the time estimates, we will often work with time steps of length $T=kT_\gd$ with $k$ some integer large enough (see the hypotheses \eqref{hyp:T 1} and \eqref{hyp:T 2}). Depending on our objective we will rely on two different definitions of the phase, as it was already done in \cite{GPS2018} for systems in finite dimension: to prove that the empirical measure is close to periodic solution $\Gamma$ we will consider phases adapted to the stability result \eqref{eq:Phi contracts} (see Section~\ref{sec:close to M}), while to define $v_{N,t}$ we will rely on the isochron map given by Theorem~\ref{th:Theta}.

\section{A first bound}\label{sec:first bound}
The aim of this section is to prove the following bound:
\begin{proposition}\label{prop: bound mu1t}
Let $(r, \gamma, \theta)$ be fixed as in \eqref{hyp gamma xi}, \eqref{hyp:r} and \eqref{hyp:theta}. Under \eqref{hyp:control_pN0}, there exists a $\gd_1>0$ such that the following holds: for $\gd\in (0,\gd_1)$, suppose that
\begin{equation}
\sup_{N\geq 1} \mathbf{E}\left[\langle p_{N,0}, w_{-\gamma}\rangle \right] <\infty,
\end{equation}
and that there exists a constant $\kappa_0>0$ such that
\begin{equation}
\mathbf{P}\left(\left\Vert p_{N,0}-g_N\right\Vert_{H_\theta^{-r+2}}\leq \kappa_0\right)\underset{N\rightarrow\infty}{\longrightarrow }1,
\end{equation}
then there exists $\kappa_1>0$ such that
\begin{equation}
\label{eq:uniform_control_pN }
\mathbf{P}\left(\sup_{0\leq t\leq Nt_f} \left\Vert p_{N,t}- g_{ N}\right\Vert_{H_\theta^{-r+2}}\leq \kappa_1\right)\underset{N\rightarrow\infty}{\longrightarrow }1.
\end{equation}
\end{proposition}

We first prove the following Lemma that will allow us to control the $H^{-r}_\theta$ norms of $\gd_{Y_{i,t}}$.

\begin{lemma}
\label{lem:control_exp_Yi}
Consider $ \gamma\in (0,1)$ and $\theta$ and $r$ satisfying \eqref{hyp:theta} and \eqref{hyp:r}. Let $ \bar \delta>0$, suppose \eqref{hyp:control_pN0} and denote by 
\begin{equation}
C_{ 0, \gamma}:=\sup_{ N\geq 1}\mathbf{ E} \left[ \left\langle p_{ N, 0}\, ,\, w_{ - \gamma}\right\rangle\right]< \infty.
\end{equation} 
Then there exists a constant $\kappa_2=\kappa_2(F, K, \sigma, \gamma, \bar\gd)$ such that for all $\gd \in (0, \bar \gd)$,
\begin{equation}
\label{eq:control_wgamma}
\sup_{ t\geq0,\ N\geq1}\mathbf{ E} \left[ \frac{ 1}{ N} \sum_{ i=1}^{ N}  w_{ - \gamma}(Y_{ i, t})\right] \leq \max \left(\kappa_2, C_{ 0, \gamma}\right).
\end{equation}
In particular, we have
\begin{equation}
\label{eq:uniform_control_norm_mu1}
\sup_{ t\geq 0, \, N\geq 1} \mathbf{ E} \left[ \left\Vert p_{ N, t} \right\Vert_{ H_{ \theta}^{ -r+2}}\right] \leq C_{ 1, r}\max \left( \kappa_{ 2}, C_{ 0, \gamma}\right)<\infty,
\end{equation} where $C_{ 1, r}$ is given by Lemma~\ref{lem:control_delta_x} for the choice of $ \eta=1$.
\end{lemma}
Note that the uniform bound \eqref{eq:uniform_control_norm_mu1} only depends on $ \bar \delta$, not on the particular choice of $ \delta\in (0, \bar \delta)$. In particular, taking some smaller $ \delta$ if necessary in the following will not change the uniform bound \eqref{eq:uniform_control_norm_mu1}.
\begin{proof}[Proof of Lemma~\ref{lem:control_exp_Yi}]
Let $ \gamma\in(0, 1)$.  By Ito's formula, one obtains, using \eqref{eq:eds_Yi}
\begin{align*}
w_{ - \gamma}(Y_{ i, t})=& w_{ - \gamma}(Y_{ i, 0}) \\
&+ \gamma \int_{ 0}^{t}  w_{ - \gamma}(Y_{ i, s}) \left\lbrace -\left(1-\gamma\left(1- \frac{ 1}{ N}\right)\right)\left\vert K^{ \frac{ 1}{ 2}}Y_{ i, s} \right\vert^{ 2}_{ K \sigma^{ -2}} + \left(1- \frac{ 1}{ N}\right)  {\rm Tr}(K)\right\rbrace {\rm d}s\\
&+ \gamma  \gd\int_{ 0}^{t}  w_{ - \gamma}(Y_{ i, s}) \left\langle K \sigma^{ -2}Y_{ i, s}\, ,\, F_{ m_{ N, s}}(Y_{ i, s}) - \frac{ 1}{ N}\sum_{j=1}^{ N}F_{ m_{ N, s}}(Y_{ j, s})\right\rangle {\rm d}s\\
&+ \gamma \sqrt{ 2} \sum_{ k=1}^{ d} \int_{ 0}^{t} \left( K \sigma^{ -2}Y_{ i, s}\right)_{ k} w_{ - \gamma}(Y_{ i, s}) \left( \left( \sigma \cdot {\rm d} B_{ i, s}\right)_{ k}- \frac{ 1}{ N} \left( \sigma \cdot \sum_{ j=1}^{ N} {\rm d}B_{ j, s}\right)_{ k}\right),
\end{align*}
so that
\begin{align*}
\frac{ {\rm d}}{ {\rm d}t}\mathbf{ E} \left[w_{ - \gamma}(Y_{ i, t})\right]&= \gamma \mathbf{ E} \bigg[w_{ - \gamma}(Y_{ i, t}) \bigg\lbrace -\left(1-\gamma\left(1- \frac{ 1}{ N}\right)\right)\left\vert K^{ \frac{ 1}{ 2}}Y_{ i, t} \right\vert^{ 2}_{ K \sigma^{ -2}}\\
&\qquad \qquad \qquad \qquad\qquad \qquad\qquad \qquad\qquad + \left(1- \frac{ 1}{ N}\right)  {\rm Tr}(K)\bigg\rbrace \bigg]\\
&+ \gamma \gd\mathbf{ E} \left[w_{ - \gamma}(Y_{ i, t}) \left\langle K \sigma^{ -2}Y_{ i, t}\, ,\, F_{ m_{ N, t}}(Y_{ i, t}) - \frac{ 1}{ N}\sum_{j=1}^{ N}F_{ m_{ N, t}}(Y_{ j, t})\right\rangle\right].
\end{align*}
Since $F$ is bounded, one obtains that
\begin{equation*}
 \frac{ {\rm d}}{ {\rm d}t}\mathbf{ E} \left[w_{ - \gamma}(Y_{ i, t})\right]\leq  \gamma  \mathbf{ E} \left[w_{ - \gamma}(Y_{ i, t}) p \left(\left\vert Y_{ i, t} \right\vert_{ K \sigma^{ -2}}\right)\right] 
\end{equation*}
with 
\begin{equation}
p(u):= -\left(1-\gamma\left(1- \frac{ 1}{ N}\right)\right)k_{ min} u^{ 2} + \frac{ 2 \gd \left\Vert F \right\Vert_{ \infty}k_{ max}^{ \frac{ 1}{ 2}}}{ \sigma_{ min}} u + \left(1- \frac{ 1}{ N}\right)  {\rm Tr}(K).
\end{equation}
Using \cite{Lucon:2018qy} Lemma~A.1, we have that 
\begin{equation}
e^{ \frac{ \gamma}{ 2} u^{ 2}}p(u)\leq C-e^{ \frac{ \gamma}{ 2} u^{ 2}}
\end{equation}
for the constant
\begin{multline}
C:= \left( \frac{\gd \left\Vert F \right\Vert_{ \infty}^{ 2}k_{ max}}{\left(1-\gamma\left(1- \frac{ 1}{ N}\right)\right)k_{ min}\sigma_{ min}^{ 2}} +  \left(1- \frac{ 1}{ N}\right)  {\rm Tr}(K)+1\right) \\
\exp \left( \frac{ \gamma}{ 2\left(1-\gamma\left(1- \frac{ 1}{ N}\right)\right)k_{ min}} \left( \frac{4\gd  \left\Vert F \right\Vert_{ \infty}^{ 2}k_{ max}}{ \left(1-\gamma\left(1- \frac{ 1}{ N}\right)\right)k_{ min}\sigma_{ min}^{ 2}} + 2\left(1- \frac{ 1}{ N}\right)  {\rm Tr}(K)\right)\right)
\end{multline}
that is anyway smaller than (note that $ \gamma\in (0, 1)$ ensures that $ 1 -\gamma \left(1-\frac{1}{ N}\right) >1 - \gamma$ uniformly in $N$)
\begin{multline}
\kappa_2:= \left( \frac{\bar \gd\left\Vert F \right\Vert_{ \infty}^{ 2}k_{ max}}{ (1- \gamma) k_{ min}\sigma_{ min}^{ 2}} + {\rm Tr}(K)+1\right) \\
\times\exp \left( \frac{ 1}{ 2 (1- \gamma) k_{ min}} \left( \frac{4 \bar \gd\left\Vert F \right\Vert_{ \infty}^{ 2}k_{ max}}{ (1- \gamma) k_{ min}\sigma_{ min}^{ 2}} + 2 {\rm Tr}(K)\right)\right).
\end{multline}
Hence, summing over $i=1, \ldots, N$
\begin{equation*}
\frac{ {\rm d}}{ {\rm d}t}\mathbf{ E} \left[ \frac{ 1}{ N} \sum_{ i=1}^{ N} w_{ - \gamma}(Y_{ i, t})\right]\leq  \gamma \left(\kappa_2- \mathbf{ E} \left[ \frac{ 1}{ N} \sum_{ i=1}^{ N} w_{ - \gamma}(Y_{ i, t})\right] \right)
\end{equation*}
which gives \eqref{eq:control_wgamma}. Since by \eqref{hyp:r}, we have $r -2> \frac{ d}{ 2}$ and by \eqref{hyp:theta}, we have in particular $ \frac{ \theta}{ 3} \leq \frac{ \gamma}{ 2}$, taking $ \eta=1$ in Lemma~\ref{lem:control_delta_x}, we obtain directly \eqref{eq:uniform_control_norm_mu1}.
\end{proof}



To prove Proposition~\ref{prop: bound mu1t} we will rely on stability properties given by the operator $\cL_N$ defined as
\begin{equation}
\cL_N p = \left(1-\frac1N\right)\nabla\cdot (\gs^2\nabla p)+ \nabla\cdot  \left(pKx \right),
\end{equation}
with dual
\begin{equation}
\cL^*_N f = \left(1-\frac1N\right)\nabla\cdot (\gs^2\nabla f)- Kx \cdot\nabla f.
\end{equation}
Proposition A.3 in \cite{LP2021a} gives the following estimates on this operator (the uniformity of the constants with respect to $N$ can be made apparent in the proof given in \cite{LP2021a}).

\begin{lemma}\label{lem:bound OU}
There exists constants $C_\cL>0$ and $\lambda_\cL>0$ independent from $N$ such that for all $0<s<t$, $0\leq \beta\leq 2$ and all $f\in H^r_\theta$,
\begin{align}
\left\Vert e^{t\cL^*_N}f\right\Vert_{H^{r+\beta}_\theta} &\leq C_\cL \left(1+ t^{-\frac{\beta}{2}}\right) \left\Vert f\right\Vert_{H^r_\theta},\label{eq:cLN bound 1}\\
\left\Vert \left(e^{t\cL^*_N}-1\right)f\right\Vert_{H^{r-\beta}_\theta} &\leq  C_\cL \left( t^{\frac{\beta}{2}}\wedge 1\right)\left\Vert f\right\Vert_{H^r_\theta}, \label{eq:cLN bound 3}\\
\left\Vert \left(e^{t\cL^*_N}-e^{s\cL^*_N}\right)f\right\Vert_{H^{r-\beta}_\theta} &\leq  C_\cL\left((t-s)^{\frac{\beta}{2}}\wedge 1\right)\left\Vert f\right\Vert_{H^r_\theta}, \label{eq:cLN bound 4}\\
\left\Vert \nabla e^{t\cL^*_N}f\right\Vert_{H^{r}_\theta}& \leq C_\cL e^{- \lambda_\cL t}\left\Vert f\right\Vert_{H^{r+1}_\theta}. \label{eq:cLN bound 5}
\end{align} 
\end{lemma}

We prove the uniform a priori bound result of Proposition~\ref{prop: bound mu1t} through a discretization of the process $ p_{ N}$: set $T>0$ satisfying
\begin{equation}\label{hyp:T 1}
C_\cL e^{-\lambda_\cL T}\leq \frac14,
\end{equation}
define
\begin{equation}\label{def:nf}
n_{ f}= N \left(\left\lfloor \frac{t_f}{T}\right\rfloor+1\right),
\end{equation}
and for all $n\in \left\lbrace0, \ldots, n_{ f}\right\rbrace$
\begin{equation}
p_{ N, t}^{(n)}:= p_{N, nT+t},\ t\in [0, T).
\end{equation}

\begin{lemma}
Let $(r, \gamma, \theta)$ be fixed as in \eqref{hyp:r} and \eqref{hyp:theta} and $\bar \delta>0$. For all $ \delta\in (0, \bar \delta)$, for all $n\in\{0,\ldots,n_f-1\}$, the process $\left(p_{ N, t}^{ (n)} \right)_{ t\in [0, T]}$ satisfies the following equation in $ \mathcal{ C} \left([0, T], H_{ \theta}^{ -r+2}\right)$, written in a mild form: for all $t\in [0,T]$,
\begin{align}
p_{N, t}^{ (n)}-g_N =\, & e^{t \cL_N}\left(p_{N,0}^{ (n)}-g_N\right)\nonumber \\
&\, - \delta\int_0^t e^{(t-s)\cL_N}\nabla\cdot \left\lbrace\left(F_{m_{N,nT+s}}-\langle p_{N,s}^{ (n)},F_{m_{N,nT+s}}\rangle\right)p_{ N, s}^{ (n)}\right\rbrace {\rm d}s+V_{N,t}^{ (n)}\label{eq:decomp mu1-g},
\end{align}
where $V_{ N, t}^{ (n)}$ is the limit in $H_{ \theta}^{ -r+2}$ as $t^{ \prime} \nearrow t$ of $V_{ N, t^{ \prime}, t}^{ (n)}$ given by, for all $f\in H^{r-2}_\theta$,
\begin{equation}
\label{eq:Vnttprime}
V_{N,t',t}^{ (n)} =\frac{\sqrt{2}}{N} \sum_{i=1}^N\int_0^{t'} \nabla e^{(t-s)\cL_N^*} f(Y_{i,nT+s})\cdot \gs \left( \dd B_{i,nT+s } -\frac{1}{N}\sum_{j=1}^N \dd B_{j,nT+s}\right),
\end{equation}
that we denote by
\begin{equation}
V_{N,t}^{ (n)}(f) = \sum_{i=1}^N\int_0^t \nabla e^{(t-s)\cL_N^*} f(Y_{i,nT+s})\cdot \gs \left( \dd B_{i,nT+s } -\frac{1}{N}\sum_{j=1}^N \dd B_{j,nT+s}\right). 
\end{equation}
\end{lemma}

\begin{proof}
We prove this result for $n=0$, the proof being similar for $n\geq 1$ and we simply write $p_{ N, s}= p_{ N, s}^{ (0)}$. Consider any test function $t \mapsto \varphi_t(x) \in \mathcal{ C}_{ b}^{ 1, 2} \left([0, +\infty) \times \mathbb{ R}^{ d}, \mathbb{ R}\right)$. By Ito's formula, we have
\begin{align}
\varphi_t(Y_{i,t})=\, &\varphi_0(Y_{i,0})+\int_0^t \partial_s \varphi_s(Y_{i,s})\dd s
+\left(1-\frac1N\right) \int_0^t \nabla\cdot\left(\gs^2\nabla \varphi_s(Y_{i,s})\right)\dd s \nonumber\\
&+\int_0^t \nabla\varphi_s(Y_{i,s})\cdot \left(F_{m_{N,s}}(Y_{i,s})-KY_{i,s}-\frac{1}{N}\sum_{j=1}^N F_{m_{N,s}}(Y_{j,s}) \right)\dd s \\
&+\sqrt{2}\int_0^t \nabla \varphi_s(Y_{i,s})\cdot \gs\left(\dd B_{i,s}-\frac{1}{N}\sum_{j=1}^N \dd B_{j,s} \right), \nonumber
\end{align}
and we obtain after summation
\begin{align}
\label{eq:muN_weak}
\left\langle p_{N,t}, \varphi_t\right\rangle =&\, \left\langle p_{N,0},\varphi_0\right\rangle \nonumber\\
&+ \int_0^t \bigg\langle p_{N,s},\partial_s \varphi_s + \left(1-\frac1N\right)\nabla\cdot(\gs^2\nabla \varphi_s)\nonumber\\
&\qquad\qquad\qquad\qquad\qquad+\nabla \varphi_s\cdot \left(F_{m_{N,s}}-Ky -\langle p_{N,s},F_{m_{N,s}}\rangle \right) \bigg\rangle \dd s \nonumber\\
& +\frac{\sqrt{2}}{N}\sum_{i=1}^N\int_0^t \nabla \varphi_s(Y_{i,s})\cdot \gs \left( \dd B_{i,s } -\frac{1}{N}\sum_{j=1}^N \dd B_{j,s}\right). 
\end{align}
Take now some $f\in \mathcal{ C}^{ \infty}_{ b}(\mathbb{ R}^{ d}, \mathbb{ R})$ and consider $s\in [0, t] \mapsto \varphi_{ s}:= e^{ (t-s) \mathcal{ L}_{ N}^{ \ast}}f$. Then $\varphi_t=f$  and $\partial_s\varphi_s +\cL_N^* \varphi_s=0$. Hence, one obtains from \eqref{eq:muN_weak} that
\begin{align}
\label{eq:muN_weak_f}
\left\langle p_{N,t}, f\right\rangle =&\, \left\langle p_{N,0}, e^{ t \mathcal{ L}_{ N}^{ \ast}}f\right\rangle + \int_0^t \bigg\langle p_{N,s}, \nabla \left(e^{ (t-s) \mathcal{ L}_{ N}^{ \ast}}f\right)\cdot \left(F_{m_{N,s}} -\langle p_{N,s},F_{m_{N,s}}\rangle \right) \bigg\rangle \dd s \nonumber\\
& +\frac{\sqrt{2}}{N}\sum_{i=1}^N\int_0^t \nabla \left(e^{ (t-s) \mathcal{ L}_{ N}^{ \ast}}f\right)(Y_{i,s})\cdot \gs \left( \dd B_{i,s } -\frac{1}{N}\sum_{j=1}^N \dd B_{j,s}\right). 
\end{align}
The point now is to justify the integration by parts leading to \eqref{eq:decomp mu1-g}. Look first at the initial condition: consider a sequence $(p_{ l})_{ l\geq 1}$ of elements of $ \mathcal{ C}^{ \infty}_{ c} \left(\mathbb{ R}^{ d}\right)\subset L_{ - \theta}^{ 2}$ converging to $ p_{ N, 0}$ in $ H_{ \theta}^{ -r+2}$. Then for $f\in \mathcal{ C}^{ \infty}_{ b}(\mathbb{ R}^{ d}, \mathbb{ R})$, we have, for $l\geq1$
\begin{equation*}
\left\langle p_{ l}\, ,\, e^{ t \mathcal{ L}_{ N}^{ \ast}}f\right\rangle= \left\langle e^{ t \mathcal{ L}_{ N}} p_{ l}\, ,\, f\right\rangle
\end{equation*}
By continuity of $ e^{ t \mathcal{ L}_{ N}}$ on $H_{ \theta}^{ -r+2}$, $ e^{ t \mathcal{ L}_{ N}} p_{ l}$ converges in $H_{ \theta}^{ -r+2}$ to $ e^{ t \mathcal{ L}_{ N}} p_{ N, 0}$ as $N\to\infty$. In particular, $ \left\vert \left\langle e^{ t \mathcal{ L}_{ N}} p_{ N, 0} \, ,\, f\right\rangle - \left\langle e^{ t \mathcal{ L}_{ N}} p_{ l}\, ,\, f\right\rangle \right\vert \leq \left\Vert f \right\Vert_{ \theta}^{ r-2} \left\Vert e^{ t \mathcal{ L}_{ N}} p_{ l}-e^{ t \mathcal{ L}_{ N}} p_{ N, 0}  \right\Vert_{ \theta}^{ -r+2} \xrightarrow[ l\to\infty]{}0$. Hence, for all $t\geq0$,
\begin{equation*}
\left\langle p_{ N, 0}\, ,\, e^{ t \mathcal{ L}_{ N}^{ \ast}}f\right\rangle= \left\langle e^{ t \mathcal{ L}_{ N}} p_{ N, 0}\, ,\, f\right\rangle
\end{equation*}
for all regular test function $f$. We now turn to the second term in \eqref{eq:muN_weak_f}: consider $(w_{ s, l})_{ l\geq1}$ a sequence of elements in $ \mathcal{ C}^{ \infty}_{ c}(\mathbb{ R}^{ d})$ converging in $H_{ \theta}^{ -r+2}$ to $ p_{ N, s}$ (take for example $ w_{ s, l}= \phi_{ l} \ast p_{ N, s}$ for a regular approximation of the identity) and remark that
\begin{multline}
\left\langle w_{ s, l}\, ,\, \nabla \left(e^{ (t-s) \mathcal{ L}_{ N}^{ \ast}}f\right)\cdot \left(F_{m_{N,s}} -\langle p_{N,s},F_{m_{N,s}}\rangle \right)\right\rangle \\
= - \left\langle e^{ (t-s) \mathcal{ L}_{ N}}\nabla \cdot \left(w_{ s, l} \left(F_{m_{N,s}} -\langle p_{N,s},F_{m_{N,s}}\rangle \right) \right)\, ,\,   f\right\rangle.
\end{multline}
The lefthand part of the previous inequality converges, as $l\to\infty$ to
\[
\left\langle p_{ N, s}\, ,\, \nabla \left(e^{ (t-s) \mathcal{ L}_{ N}^{ \ast}}f\right)\cdot \left(F_{m_{N,s}} -\langle p_{N,s},F_{m_{N,s}}\rangle \right)\right\rangle.
\]
Moreover, concerning the righthand term, relying on Lemma~\ref{lem:bound OU},
\begin{align*}
&\left\vert  \left\langle e^{ (t-s) \mathcal{ L}_{ N}}\nabla \cdot \left((w_{ s, l}- p_{ N, s}) \left(F_{m_{N,s}} -\langle p_{N,s},F_{m_{N,s}}\rangle \right) \right)\, ,\,   f\right\rangle \right\vert \\ &\leq \left\Vert f \right\Vert_{H_{ \theta}^{ r-2}} \left\Vert e^{ (t-s) \mathcal{ L}_{ N}}\nabla \cdot \left((w_{ s, l}- p_{ N, s}) \left(F_{m_{N,s}} -\langle p_{N,s},F_{m_{N,s}}\rangle \right) \right) \right\Vert_{ H_{ \theta}^{ -r+2}}\\
&\leq \frac{ C}{ \sqrt{ t-s}} e^{ - \lambda_\cL(t-s)}\left\Vert f \right\Vert_{H_{ \theta}^{ r-2}} \left\Vert \left((w_{ s, l}- p_{ N, s}) \left(F_{m_{N,s}} -\langle p_{N,s},F_{m_{N,s}}\rangle \right) \right) \right\Vert_{ H_{ \theta}^{ -r+2}}
\end{align*}
which gives, since $F$ is bounded as well as all its derivatives
\begin{align*}
&\left\vert  \left\langle e^{ (t-s) \mathcal{ L}_{ N}}\nabla \cdot \left((w_{ s, l}- p_{ N, s}) \left(F_{m_{N,s}} -\langle p_{N,s},F_{m_{N,s}}\rangle \right) \right)\, ,\,   f\right\rangle \right\vert \\ 
&\leq \frac{ C^{ \prime}}{ \sqrt{ t-s}} e^{ - \lambda_\cL(t-s)}\left\Vert f \right\Vert_{H_{ \theta}^{ r}} \left\Vert w_{ s, l}- p_{ N, s} \right\Vert_{ H_{ \theta}^{ -r+2}}
\end{align*}
and since this is true for all regular $f$, we have
\begin{multline}
\left\Vert e^{ (t-s) \mathcal{ L}_{ N}}\nabla \cdot \left((w_{ s, l}- p_{ N, s}) \left(F_{m_{N,s}} -\langle p_{N,s},F_{m_{N,s}}\rangle \right) \right) \right\Vert_{ H_{ \theta}^{ -r+2}} \\
\leq \frac{ C^{ \prime}}{ \sqrt{ t-s}} e^{ - \lambda_\cL(t-s)} \left\Vert w_{ s, l}- p_{ N, s} \right\Vert_{ H_{ \theta}^{ -r+2}},
\end{multline}
which goes to $0$ as $l\to \infty$. This implies in particular that for all $f$ regular
\begin{multline}
\left\langle p_{ N, s}\, ,\, \nabla \left(e^{ (t-s) \mathcal{ L}_{ N}^{ \ast}}f\right)\cdot \left(F_{m_{N,s}} -\langle p_{N,s},F_{m_{N,s}}\rangle \right)\right\rangle \\
= - \left\langle e^{ (t-s) \mathcal{ L}_{ N}}\nabla \cdot \left(p_{ N, s} \left(F_{m_{N,s}} -\langle p_{N,s},F_{m_{N,s}}\rangle \right) \right)\, ,\,   f\right\rangle
\end{multline}
Similarly as before, we also have that
\begin{equation}
\left\Vert e^{ (t-s) \mathcal{ L}_{ N}}\nabla \cdot \left(p_{ N, s} \left(F_{m_{N,s}} -\langle p_{N,s},F_{m_{N,s}}\rangle \right) \right) \right\Vert_{ H_{ \theta}^{ -r+2}} \leq \frac{ C^{ \prime}}{ \sqrt{ t-s}} \left\Vert p_{ N, s} \right\Vert_{ H_{ \theta}^{ -r+2}}
\end{equation}
Now observe that \eqref{eq:uniform_control_norm_mu1} gives that for all $t\geq 0$, 
\begin{equation}
\mathbf{ E} \left[\int_{ 0}^{t} \left\Vert e^{ (t-s) \mathcal{ L}_{ N}}\nabla \cdot \left(p_{ N, s} \left(F_{m_{N,s}} -\langle p_{N,s},F_{m_{N,s}}\rangle \right) \right) \right\Vert_{ H_{ \theta}^{ -r+2}}  {\rm d}s\right]< \infty,
\end{equation}
and hence, is almost surely finite. Using \cite{MR617913}, Theorem~1, p.~133, this entails that the integral $\int_{ 0}^{t}  e^{ (t-s) \mathcal{ L}_{ N}}\nabla \cdot \left(p_{ N, s} \left(F_{m_{N,s}} -\langle p_{N,s},F_{m_{N,s}}\rangle \right) \right)  {\rm d}s$ makes sense as a Bochner integral in $H_{ \theta}^{ -r+2}$ and is a continuous function of $t$.
It remains to treat the noise term in \eqref{eq:muN_weak_f}: the precise control is made below (see Lemma~\ref{lem:bound Vnt-Vns} and its proof). We deduce in particular from the argument given below and an application of Kolmogorov Lemma that the almost-sure limit of $V_{ n, t^{ \prime}, t}$ defined in \eqref{eq:Vnttprime} below exists in $H_{ \theta}^{ -r+2}$ and defines a continuous process $ t \mapsto V_{ n, t}$ in $H_{ \theta}^{ -r+2}$.
The result follows by remarking that $e^{t\cL_N^*} g_N =g_N$.
\end{proof}

Recall the definition of $V_{ N, t^{ \prime}, t}^{ (n)}$ in \eqref{eq:Vnttprime}: we have for $0<s'<s<t$, $s'<t'<t$,
\begin{equation}
V_{N,t',t}^{ (n)}-V_{N,s',s}^{ (n)}=V^{(n, 1)}_{N,s',s,t}+V^{(n, 2)}_{N,s',s,t}+V^{(n, 3)}_{N,s',t',t}+V^{(n, 4)}_{N,s',t',t},
\end{equation}
with
\begin{align}
V^{(n, 1)}_{N,s',s,t}(f) & =\frac{\sqrt{2}}{N} \sum_{i=1}^N\int_0^{s'} \nabla\left( e^{(t-v)\cL_N^*}-e^{(s-v)\cL_N^*}\right) f(Y_{i,nT+v})\cdot \gs \dd B_{i,nT+v },\\
V^{(n, 2)}_{N,s',s,t}(f) & =-\frac{\sqrt{2}}{N} \sum_{j=1}^N\int_0^{s'} \left\{\frac{1}{N}\sum_{i=1}^N \nabla\left( e^{(t-v)\cL_N^*}-e^{(s-v)\cL_N^*}\right) f(Y_{i,nT+v})\right\}\cdot \gs \dd B_{j,nT+v },\\
V^{(n, 3)}_{N,s,t',t}(f) & =\frac{\sqrt{2}}{N} \sum_{i=1}^N\int_{s'}^{t'} \nabla e^{(t-v)\cL_N^*} f(Y_{i,nT+v})\cdot \gs \dd B_{i,nT+v },\\
V^{(n, 4)}_{N,s,t',t}(f) & =-\frac{\sqrt{2}}{N} \sum_{j=1}^N\int_{s'}^{t'} \left\{\frac{1}{N}\sum_{i=1}^N \nabla e^{(t-v)\cL_N^*} f(Y_{i,nT+v})\right\}\cdot \gs \dd B_{j,nT+v }.
\end{align}

\begin{lemma}\label{lem:VNi}
Let $(r, \gamma, \theta)$ be fixed as in \eqref{hyp gamma xi}, \eqref{hyp:r} and \eqref{hyp:theta} and $\bar \delta>0$. For all $ \delta\in (0, \bar \delta)$, for all $n\in\{0,\ldots,n_f-1\}$, the processes $\left(V^{(n, i)}_{N,t',t}\right)_{t'\in [0,t)}$, $n=0,\ldots, n_f-1$, $i=0,\ldots,4$, are  martingales in $H^{-r+2}_\theta$, with trajectories almost surely continuous, and with Meyer processes (see \cite{metivier2011semimartingales})
\begin{align}
\left\langle V^{(n, 1)}_{N,\cdot,s,t}\right\rangle_{s'}& = 2\sum_{i=1}^N \int_0^{s'}\left\Vert \frac1N \gs \cdot \left[ \nabla e^{(t-v)\cL_N^*}(\cdot)- \nabla e^{(s-v)\cL_N^*}(\cdot)\right](Y_{i,nT+v})\right\Vert^2_{H^{-r+2}_\theta} \dd v, \\
\left\langle V^{(n, 2)}_{N,\cdot,s,t}\right\rangle_{s'}& = \frac{2}{N}\sum_{i=1}^N \int_0^{s'}\left\Vert \frac1N \sum_{i=1}^N \gs \cdot \left[\nabla e^{(t-v)\cL_N^*}(\cdot)-\nabla e^{(s-v)\cL_N^*}(\cdot)\right](Y_{i,nT+v})\right\Vert^2_{H^{-r+2}_\theta} \dd v, \\
\left\langle V^{(n, 3)}_{N,\cdot,s,t}\right\rangle_{s'}& = 2\sum_{i=1}^N \int_{s'}^{t'}\left\Vert \frac1N \gs \cdot \nabla e^{(t-v)\cL_N^*}(\cdot)(Y_{i,nT+v})\right\Vert^2_{H^{-r+2}_\theta} \dd v, \\
\left\langle V^{(n, 4)}_{N,\cdot,s,t}\right\rangle_{s'}& = \frac{2}{N}\sum_{i=1}^N \int_{s'}^{t'}\left\Vert \frac1N \sum_{i=1}^N \gs \cdot \nabla e^{(t-v)\cL_N^*}(\cdot)(Y_{i,nT+v})\right\Vert^2_{H^{-r+2}_\theta} \dd v.
\end{align}
\end{lemma}
\begin{proof}[Proof of Lemma~\ref{lem:VNi}]
We prove the statement only for $V_{ N, \cdot, s, t}^{ (n, 1)}$ for $n=0$, the other terms being treated similarly and we write $V_{ N,s^{ \prime}}^{ (1)}$ for simplicity. Let us first prove continuity of the trajectories of $V_{ N, s^{ \prime}}^{ (1)}$. We follow here closely an argument similar to \cite{Fernandez1997}, Prop.~4.4. Let $(f_{ l})_{ l\geq0}$ be a complete orthonormal basis in $ H_{ \theta}^{ r-2}$. We have
\begin{align}
\sum_{l\geq 0}&\mathbf{E}\left[\sum_{i=1}^{ N} \int_{ 0}^{s^{ \prime}}  \left\vert \frac{ 1}{ N} \sigma\cdot \left[\nabla e^{(t-v)\cL^*_N}(f_{ l})- \nabla e^{(s-v) \cL^*_N}(f_{ l}) \right](Y_{ i, v}) \right\vert^{ 2} {\rm d}v\right]\nonumber\\
&=2\sum_{l\geq 0}\sum_{i=1}^{ N} \int_{ 0}^{s^{ \prime}}  \left\vert \frac{ 1}{ N} \sigma\cdot \left[\nabla e^{(t-v)\cL^*_N}(f_{ l})- \nabla e^{(s-v)\cL^*_N}(f_{ l}) \right](Y_{ i, v}) \right\vert^{ 2} {\rm d}v\nonumber\\
&=2\sum_{i=1}^{ N} \int_{ 0}^{s^{ \prime}}  \mathbf{ E} \left[\left\Vert \frac{ 1}{ N} \sigma\cdot \left[\nabla e^{(t-v)\cL^*_N}(\cdot)- \nabla e^{(s-v)\cL^*_N}(\cdot) \right](Y_{ i, v}) \right\Vert^{ 2}_{ H_{ \theta}^{ -r+2}}\right] {\rm d}v .
\end{align}
Now, for any $f\in H_\theta^{ r-2}$, relying on Lemma~\ref{lem:bound OU} we obtain,
\begin{align}
\bigg|\sigma\cdot \Big[\nabla e^{(t-v)\cL_n^*}(f)-&\nabla e^{(s-v)\cL^*_N}(f) \Big](Y_{ i, v})\bigg|
\nonumber\\
&\leq C_1\left\Vert \gd_{Y_{i,v}}\right\Vert_{  H_{ \theta}^{ -(r-3)}}\left\Vert \nabla e^{(t-v)\cL^*_N}(f)- \nabla e^{(s-v)\cL^*_N}(f)  \right\Vert_{H_{ \theta}^{ r-3}}\nonumber\\
&\leq C_2\left\Vert \gd_{Y_{i,v}}\right\Vert_{H_{ \theta}^{ -(r-3)}}\left\Vert   e^{(t-v)\cL^*_N}(f)-  e^{(s-v)\cL^*_N}(f) \right\Vert_{  H_{ \theta}^{ r-2}}\\
&\leq C_3 w_{ \frac{ -2 \theta}{ 3}} \left(Y_{ i, v}\right) \Vert f\Vert_{ H_{ \theta}^{ r-2}},\nonumber\\
\end{align}
by Lemma~\ref{lem:control_delta_x} with $\eta=1$ (remark that hypothesis \eqref{hyp:r} implies in particular $r-3> \frac{ d}{ 2}$)
which leads to
\begin{equation}
\left\Vert \frac{ 1}{ N} \sigma\cdot \left[\nabla e^{(t-s)\cL^*_N}(\cdot)- \nabla e^{(s-v)\cL^*_N}(\cdot) \right](Y_{ i, v}) \right\Vert^{ 2}_{ H_{ \theta}^{ -r+2}}\leq \frac{C_3}{N^2}w_{ \frac{ -4 \theta}{3}} \left(Y_{ i, v}\right).
\end{equation}
Relying on Lemma~\ref{lem:control_exp_Yi} we then deduce, remarking that hypothesis \eqref{hyp:theta} implies in particular $ 4 \theta/3 \leq \gamma $,
\begin{align}
\sum_{l\geq 0}\mathbf{E}&\left[\sum_{i=1}^{ N} \int_{ 0}^{s^{ \prime}}  \left\vert \frac{ 1}{ N} \sigma\cdot \left[\nabla e^{(t-v)\cL^*_N}(f_{ l})- \nabla e^{(s-v) \cL^*_N}(f_{ l}) \right](Y_{ i, v}) \right\vert^{ 2} {\rm d}v\right]\nonumber\\
&\leq \frac{C_4}{N}\int_{ 0}^{s^{ \prime}}  \frac{ 1}{ N} \sum_{ i=1}^{ N}  \mathbf{ E} \left[w_{ \frac{ -4 \theta}{ 3} }(Y_{ i, v})\right]\dd v <\infty,
\end{align}
by Lemma~\ref{lem:control_exp_Yi}.
We deduce that, for fixed $N$ and $l\geq0$, $V^{(1)}_{N, s^{ \prime}}(f_{ l})$ is a real-valued martingale with brackets 
\begin{equation}
\left\langle V_{ N, \cdot}^{ (1)}(f_{ l})\right\rangle_{ s^{ \prime}}= 2\sum_{i=1}^{ N} \int_{ 0}^{s^{ \prime}}  \left\vert \frac{ 1}{ N} \sigma\cdot \left[\nabla e^{(t-v)\cL^*_N}(f_{ l})- \nabla e^{(s-v) \cL^*_N}(f_{ l}) \right](Y_{ i, v}) \right\vert^{ 2} {\rm d}v.
\end{equation}
Hence, by Doob's inequality
\begin{equation}
\sum_{ l=0}^{ +\infty} \mathbf{ E} \left[ \sup_{ s^{ \prime}\in [0,s)} V_{ N, s^{ \prime}}^{ (1)}(f_{ l})^{ 2}\right]< \infty ,
\end{equation}
and in particular, for fixed $N$, almost surely, for all $ \varepsilon>0$, there exists $L_{ 0}$ such that
\begin{equation}
\sum_{ l=L_{ 0}+1}^{ +\infty} \sup_{ s^{ \prime}\in[0,s)} V_{N, s^{ \prime}}^{ (1)}(f_{ l})^{ 2} < \frac{ \varepsilon}{ 6}.
\end{equation}
Let now consider some sequence $s_{ m}^{ \prime}$ such that $s_{ m}^{ \prime} \xrightarrow[ m\to\infty]{}s$. Then,
\begin{align*}
\left\Vert V_{ N, s^{ \prime}_{ m}}^{ (1)} - V_{ N, s^{ \prime}}^{ (1)}\right\Vert^{ 2}_{ H_{ \theta}^{- r+2}}&= \sum_{ l\geq0} \left[V_{ N, s^{ \prime}_{ m}}^{ (1)}(f_{ l}) - V_{ N, s^{ \prime}}^{ (1)}(f_{ l})\right]^{ 2} \\
&\leq \sum_{ l=0}^{ L_{ 0}} \left[V_{ N, s^{ \prime}_{ m}}^{ (1)}(f_{ l}) - V_{ N, s^{ \prime}}^{ (1)}(f_{ l})\right]^{ 2} + \frac{ 4 \varepsilon}{ 6}.
\end{align*}
By the continuity of the process $s^{ \prime} \mapsto V_{N, s^{ \prime}}^{ (1)}(g)$ for fixed $g$, the first term can be made smaller than $ \frac{ \varepsilon}{ 3}$ for $m$ sufficiently large. This gives the almost-sure continuity of $V_{ N, s^{ \prime}}^{ (1)}$ in $  H_{ \theta}^{- r+2}$. Its Meyer process is then simply calculated as
\begin{align*}
\left\langle V_{ N, \cdot}^{ (1)}\right\rangle_{ s^{ \prime}}&= \sum_{ l=0}^{ +\infty} 2\sum_{i=1}^{ N} \int_{ 0}^{s^{ \prime}}  \left\vert \frac{ 1}{ N} \sigma\cdot \left[\nabla e^{(t-v)\cL^*_N}(f_{ l})- \nabla e^{(s-v)\cL^*_N}(f_{ l}) \right](Y_{ i, v}) \right\vert^{ 2} dv\\
&=2\sum_{i=1}^{ N} \int_{ 0}^{s^{ \prime}}  \left\Vert \frac{ 1}{ N} \sigma\cdot \left[\nabla e^{(t-v)\cL^*_N}(\cdot)- \nabla e^{(s-v)\cL^*_N}(\cdot) \right](Y_{ i, v}) \right\Vert^{ 2}_{ H_{ \theta}^{- r+2}}  dv.
\end{align*}
The other terms can be treated in a same way.
\end{proof}

\begin{lemma}\label{lem:bound Vnt-Vns}
Let $(r, \gamma, \theta)$ be fixed as in \eqref{hyp:r} and \eqref{hyp:theta} and $m$ defined by \eqref{hyp m}. Fix $\bar \delta>0$. Then there exists a constant $\kappa_3=\kappa_3(T, \bar \delta, \gamma)>0$ such that for all $n\in\{0,\ldots, n_f-1\}$, $ \delta\in (0, \bar \delta)$ and all $0\leq s<t\leq T$,
\begin{equation}
\mathbf{E}\left[\left\Vert V_{N,t}^{ (n)}-V_{N,s}^{ (n)}\right\Vert^{2m}_{H^{-r+2}_\theta}\right]\leq \frac{\kappa_3}{N^m}  (t-s)^{ \frac{ m}{ 2}}.
\end{equation}
\end{lemma}

\begin{proof}
Our aim is to get an upper bound for $ \mathbf{ E} \left[\left\Vert V_{N,t', t}^{ (n)} - V_{N, s',  s}^{ (n)}\right\Vert_{H_{ \theta}^{-r+2}}^{ 2m}\right]$ and then make $t' \nearrow t$ and $s' \nearrow s$. Remark that it suffices to estimate $ \mathbf{ E} \left[\left\Vert V_{ N}^{ (n, i)} \right\Vert_{ H_{ \theta}^{-r+2}}^{ 2m}\right] $ for all $i=1,\ldots, 4$, relying on the inequality $ \left\Vert a+b \right\Vert^{ 2m} \leq 2m \left(\left\Vert a \right\Vert^{ 2m} + \left\Vert b \right\Vert^{ 2m}\right)$. We will treat only the case $n=0$, the calculations being similar for $n\geq 1$.

Let us first treat $V^{(0, 1)}_N$. Applying the Burkholder-Davis-Gundy inequality in Hilbert spaces (see \cite{Marinelli2016}, Th.~1.1), we obtain
\begin{multline} \label{eq:bound_norm_VN1}
\mathbf{ E} \left[ \left\Vert V_{N, s^{ \prime}, s, t}^{ (0, 1)} \right\Vert^{ 2m}_{ H_{ \theta}^{ -r+2}}\right]\leq C_1\, \mathbf{ E} \left[ \left\langle V_{N, \cdot, s, t}^{ (0, 1)}\right\rangle_{ s^{ \prime}}^{ m}\right]\\
\leq \frac{C_2}{ N^{m}}  \mathbf{ E} \left[ \left( \frac{ 1}{ N}\sum_{i=1}^{ N} \int_{ 0}^{s^{ \prime}} \left\Vert \left[\nabla e^{(t-v)\cL^*_N}(\cdot)- \nabla e^{(s-v)\cL^*_N}(\cdot)  \right](Y_{ i, v}) \right\Vert_{H^{-r+2}_{ \theta}}^{ 2} {\rm d}v\right)^{ m}\right].
\end{multline}
Let us estimate the norm above apart: applying \eqref{eq:cLN bound 4} and Lemma~\ref{lem:control_delta_x} (recall that $r-5> \frac{ d}{ 2}$ by \eqref{hyp:r}), we have, almost surely for all $f\in H^{ r-2}_{ \theta}$, 
\begin{align}
\bigg| \Big[\nabla e^{(t-v)\cL^*_N}(f)-& \nabla e^{(s-v)\cL^*_N}(f) \Big](Y_{ i,v})\bigg|\nonumber\\
&\leq C_{1,r-5}\left\Vert \gd_{Y_{i,v}}\right\Vert_{ H_{ \theta}^{ -r+5}}\left\Vert \nabla e^{(t-v)\cL^*_N}(f)- \nabla e^{(s-v)\cL^*_N}(f)  \right\Vert_{ H_{ \theta}^{ r-5}}\nonumber\\
&\leq C_3\left\Vert \gd_{Y_{i,v}}\right\Vert_{ H_{ \theta}^{ -r+5}}\left\Vert e^{(t-v)\cL^*_N}(f)- e^{(s-v)\cL^*_N}(f)  \right\Vert_{ H_{ \theta}^{ r-4}}\\
&\leq C_4(t-s) w_{ \frac{ -2 \theta}{ 3}} \left(Y_{ i, v}\right)\Vert f\Vert_{ H_{ \theta}^{ r-2}}.\nonumber
\end{align}
Hence, 
\begin{equation}
\label{eq:gradient_diff_ecL}
\left\Vert \left[\nabla e^{(t-v)\cL^*_N}(\cdot)- \nabla e^{(s-v)\cL^*_N}(\cdot) \right](Y_{ i, v}) \right\Vert_{H^{-r+2}_{ \theta} }\leq C_4 (t-s)w_{ \frac{ -2 \theta}{ 3}} \left(Y_{ i, v}\right) .
\end{equation}
Going back to \eqref{eq:bound_norm_VN1}, we obtain,
\begin{align}
\mathbf{ E} \left[ \left\Vert V_{N, s^{ \prime}, s, t}^{ (0, 1)} \right\Vert^{ 2m}_{  H_{ \theta}^{ -r+2}}\right]&\leq \frac{C_5(t-s)^{2m}}{ N^{m}}  \mathbf{ E} \left[ \left(\int_{ 0}^{s^{ \prime}}  \frac{ 1}{ N}\sum_{i=1}^{ N} w_{ \frac{ -4 \theta}{ 3}} \left(Y_{ i, v}\right) {\rm d}v\right)^{ m}\right]\\
&\leq \frac{C_5(t-s)^{2m}T^{m-1}}{ N^{m}}\mathbf{ E} \left[ \int_{ 0}^{s^{ \prime}}  \frac{ 1}{ N}\sum_{i=1}^{ N} w_{ \frac{ -4 m \theta}{ 3}} \left(Y_{ i, v}\right) {\rm d}v\right]\\
&\leq \frac{C_6}{ N^{m}}(t-s)^{2m},
\end{align}
where in the last line we have used Lemma~\ref{lem:control_exp_Yi}, since \eqref{hyp:theta} and \eqref{hyp m} imply in particular that$ \frac{ 4 m \theta}{ 3}\leq  \gamma$. Remark that here $C_6$ is of order $T^m$. The same analysis gives the same bound for $\mathbf{ E} \left[ \left\Vert V_{N, s^{ \prime}, s, t}^{ (0, 2)} \right\Vert^{ 2m}_{ H_{ \theta}^{- r+2}}\right]$. For $V^{(0, 3)}_N$, relying again of Lemma~\ref{lem:control_psit}, we obtain
\begin{align}
\left|\nabla e^{(t-v)\cL^*_N}(f)(Y_{i,v}) \right|
&\leq C_{1,r-3}\left\Vert \gd_{Y_{i,v}}\right\Vert_{  H_{ \theta}^{ -r+3}}\left\Vert \nabla e^{(t-v)\cL^*_N}(f)  \right\Vert_{ H_{ \theta}^{ r-3}}\nonumber\\
&\leq C_7 w_{ \frac{ -2 \theta}{3}} \left(Y_{ i, v}\right) \Vert f\Vert_{  H_{ \theta}^{ r-2}},
\end{align}
so that
\begin{equation}
\left\Vert \nabla e^{(t-v)\cL^*_N}(\cdot)(Y_{ i,v}) \right\Vert_{H^{-r+2}_{ \theta} }\leq C_7 w_{ \frac{ -2 \theta}{3}} \left(Y_{ i, v}\right) .
\end{equation}
Now we have (recalling that $ \frac{ m}{ 2}>1$)
\begin{align}
\mathbf{ E} \left[ \left\Vert V_{N, s^{ \prime}, s, t}^{ (0, 3)} \right\Vert^{ 2m}_{  H_{ \theta}^{ -r+2}}\right]&\leq \frac{C_8}{ N^{m}}  \mathbf{ E} \left[ \left(\int_{ s'}^{t^{ \prime}}  \frac{ 1}{ N}\sum_{i=1}^{ N} w_{ \frac{ -4 \theta}{ 3}} \left(Y_{ i, v}\right) {\rm d}v\right)^{ m}\right]\nonumber\\
&\leq \frac{C_8}{ N^{m}}\left(\int_{s'}^{t'}dv\right)^{ \frac{ m}{ 2}}\mathbf{ E} \left[\left( \int_{ s^{ \prime}}^{t^{ \prime}}  \left(\frac{ 1}{ N}\sum_{i=1}^{ N} w_{ \frac{ -4  \theta}{ 3}} \left(Y_{ i, v}\right)\right)^{2} { d}v\right)^{ \frac{ m}{ 2}}\right]\nonumber\\
&\leq \frac{C_8T^{ \frac{ m}{ 2}-1}}{ N^{m}}(t'-s')^{ \frac{ m}{ 2}}\mathbf{ E} \left[ \int_{ s'}^{t^{ \prime}}  \frac{ 1}{ N}\sum_{i=1}^{ N} w_{ \frac{ -4 m \theta}{ 3}} \left(Y_{ i, v}\right) {\rm d}v\right]\nonumber\\
&\leq \frac{C_9}{ N^{m}}(t'-s')^{ \frac{ m}{ 2}},
\end{align}
where $C_9$ is of order $T^{ \frac{ m}{ 2}}$. $V^{(0, 4)}_N$ can be treated similarly to obtain the same upper bound. All the estimates above lead to, for some positive constant $C_{10}$ depending on $T$ and $\zeta$ (and which is, for $\zeta$ fixed, of order $T^m$), 
\begin{equation}
\mathbf{ E} \left[\left\Vert V_{ N,t', t}^{ (0)} - V_{N, s', s}^{ (0)}\right\Vert_{\mathbf{ H}_{ \theta}^{-r+2}}^{ 2m}\right]\leq\frac{C_{10}}{N^m}\left( (t-s)^{ \frac{ m}{ 2}}+(t'-s')^{ \frac{ m}{ 2}}\right),
\end{equation}
which implies the result by Fatou lemma, considering $s'$ and $t'$ converging increasingly to $s$ and $t$, and defining $\kappa_3=2C_{10}$.
\end{proof}
We now rely on the following lemma, which proof can be found in \cite{SV2006}.

\begin{lemma}[Garsia-Rademich-Rumsey]\label{lem:GRR} Let $\chi$ and $\psi$ continuous and strictly increasing functions on $(0,\infty)$, such that $\chi(0)=\psi(0)=0$ and $\lim_{t\rightarrow\infty}\psi(t)=\infty$. Given $T>0$ and $\phi$ continuous on $(0,T)$ and taking its values in a Banach space $(E,\Vert \cdot\Vert)$, if
\begin{equation}
\int_0^{T}\int_0^{T} \psi\left(\frac{\Vert \phi(t)-\phi(s)\Vert}{\chi(|t-s|)}\right)\dd s \, \dd t \leq M <\infty,
\end{equation}
then for all $0\leq s\leq t\leq T$,
\begin{equation}
\Vert \phi(t)-\phi(s)\Vert \leq 8\int_0^{t-s}\psi^{-1}\left(\frac{4M}{u^2}\right)\chi(\dd u).
\end{equation}
\end{lemma}
Define now the event
\begin{equation}
\cA_N = \left\{\max_{n\in \{0,\ldots,n_f-1\}}\sup_{t\in [0,T]} \left\Vert V_{N,t}^{ (n)}\right\Vert_{H^{-r+2}_\theta}\leq N^{-\frac12+\xi}\right\}.
\end{equation}

\begin{proposition}\label{prop:Vn small}
Let $(r, \gamma, \theta)$ be fixed as in \eqref{hyp gamma xi}, \eqref{hyp:r} and \eqref{hyp:theta} and $m$ defined by \eqref{hyp m}. Fix $\bar \delta>0$. Then, for all $ \delta\in (0, \bar \delta)$, $\mathbf{P}(\cA_N)\rightarrow 1$ as $N\rightarrow\infty$.
\end{proposition}

\begin{proof}
We apply Lemma~\ref{lem:GRR} with $\phi(t)=V_{N,t}^{ (n)}$, $\chi(u) = u^{ \frac{ 1}{ 4} + \frac{ 1}{ 4m}}$ and $\psi(u)=u^{2m}$. Hence, by Lemma~\ref{lem:bound Vnt-Vns}, the random variable
\begin{equation}
M =\int_0^{T}\int_0^{T}\frac{\left\Vert V_{N,t}^{ (n)}-V_{N,s}^{ (n)}\right\Vert_{H_\theta^{-r+2}}^{2m}}{|t-s|^{\frac{m}{2}+ \frac12}}\dd s \, \dd t,
\end{equation}
satisfies
\begin{equation}
\mathbf{E}[M]\leq \frac{\kappa_3}{N^m}\int_0^T\int_0^T |t-s|^{-\frac12} \dd s\, \dd t\leq\frac{C_1}{N^m},
\end{equation}
and thus,
\begin{equation}
\left\Vert V_{N,t}^{ (n)}-V_{N,s}^{ (n)}\right\Vert_{H_\theta^{-r+2}}^{2m}\leq C_2 M (t-s)^{\frac{m}{4}}.
\end{equation}
Hence
\begin{equation}
\mathbf{E}\left(\sup_{0\leq s\leq t\leq T}\frac{\left\Vert V_{N,t}^{ (n)}-V_{N,s}^{ (n)}\right\Vert_{H_\theta^{-r+2}}^{2m}}{(t-s)^{\frac{m}{4}}}\right)\leq C_2\mathbf{E}[M]\leq \frac{C_3}{N^m},
\end{equation}
which leads to
\begin{align}
\mathbf{P}\bigg(\sup_{0\leq t\leq T}\big\Vert V_{N,t}^{ (n)}&\big\Vert_{H_\theta^{-r+2}}\geq N^{-\frac12+\xi}\bigg)\nonumber\\
&\leq N^{\left(1-2\xi\right)m}\,\mathbf{E}\left[\sup_{0\leq t\leq T}\left\Vert V_{N,t}\right\Vert_{H_\theta^{-r+2}}^{2m}\right]\nonumber\\
&\leq N^{\left(1-2\xi\right)m}T^{\frac{m}{4}}\, \mathbf{E}\left[\sup_{0\leq t\leq T}\frac{\left\Vert V_{N,t}\right\Vert_{H_\theta^{-r+2}}^{2m}}{t^{\frac{m}{4}}}\right]\leq C_4 N^{-2\xi m}.
\end{align}
Finally,
\begin{equation}
\mathbf{P}\left(\sup_{0\leq n\leq n_f}\sup_{0\leq t\leq T}\left\Vert V_{N,t}\right\Vert_{H_\theta^{-r+2}}\geq N^{-\frac12+\xi}\right)\leq C_3 n_f N^{-2\xi m},
\end{equation}
which tends to $0$ as $N$ goes to infinity since $n_f$ is of order $N$ and since $  2 \xi m<1$ by \eqref{hyp m}. 
\end{proof}

For the proof of Proposition~\ref{prop: bound mu1t} we use the following Gronwall-Henry Lemma (which is given by \cite[Lemma 7.1.1]{Henry:1981} in the particular cas $\beta=\frac12$).

\begin{lemma}\label{lem:GH}
Let $a_t$ be a non negative and locally integrable function on $[0,T]$ and $b$ a positive constant. There exists a constant $c$ such that if $u_t$ is a non negative locally integrable function on $[0,T]$ satisfying for $0\leq t<T$
\begin{equation}
u_t \leq a_t + b\int_0^t \frac{u_s}{\sqrt{t-s}}\dd s,
\end{equation}
then for $0\leq t<T$ we have
\begin{align}
u_t &\leq a_t + b^{ 2} \pi\int_0^t  \left(\frac{1}{ b \pi\sqrt{ t-s}} + 2e^{b^{ 2} \pi (t-s)}\right)a_s \dd s\nonumber\\
&= a_t + \int_0^t  \left(\frac{b}{\sqrt{ t-s}} + 2b^{ 2} \pi e^{b^{ 2} \pi (t-s)}\right)a_s \dd s.
\end{align}
\end{lemma}

\begin{proof}[Proof of Proposition~\ref{prop: bound mu1t}]
We suppose that the event $\cA_N\bigcap \left\{\left\Vert p_{N,0}-g_N\right\Vert_{H^{-r+2}_\theta}\leq \kappa_0\right\}$ is realised, and we proceed by induction on $n$. Suppose that $\left\Vert p_{N,nT}-g_N\right\Vert_{H^{-r+2}_\theta}\leq C_0$ for some constant $C_0\geq \kappa_0$.
Recall \eqref{eq:decomp mu1-g}:
\begin{align}
p_{N, t}^{ (n)}-g_N =\, & e^{t \cL_N}\left(p_{N,0}^{ (n)}-g_N\right)\nonumber \\
&\, - \delta\int_0^t e^{(t-s)\cL_N}\nabla\cdot \left\lbrace\left(F_{m_{N,nT+s}}-\langle p_{N,s}^{ (n)},F_{m_{N,nT+s}}\rangle\right)p_{ N, s}^{ (n)}\right\rbrace {\rm d}s+V_{N,t}^{ (n)},
\end{align}
and using Lemma~\ref{lem:bound OU} we obtain
\begin{align*}
&\left\Vert \int_0^t e^{(t-s)\cL_N}\nabla\cdot \left\lbrace\left(F_{m_{N,nT+s}}-\langle p_{N, s}^{ (n)},F_{m_{N,nT+s}}\rangle\right)p_{ N, s}^{ (n)}\right\rbrace {\rm d}s \right\Vert_{ H_{ \theta}^{ -r+2}} \\ 
&\leq \int_0^t \left\Vert e^{(t-s)\cL_N}\nabla\cdot \left\lbrace\left(F_{m_{N,nT+s}}-\langle p_{N,s}^{ (n)},F_{m_{N,nT+s}}\rangle\right)p_{ N, s}^{ (n)}\right\rbrace \right\Vert_{ H_{ \theta}^{ -r +2}}{\rm d}s \\
&\leq C_{  \mathcal{ L}}\int_0^t \frac{ e^{ - \lambda_\cL(t-s)}}{ (t-s)^{ \frac{ 1}{ 2}}}\left\Vert\left(F_{m_{N,nT+s}}-\langle p_{N,s}^{ (n)},F_{m_{N,nT+s}}\rangle\right)p_{ N, s}^{ (n)}\right\Vert_{ H_{ \theta}^{ -r +2}}{\rm d}s
\end{align*}
Firstly, for $h\in H_{ \theta}^{r-2}$
\begin{align*}
\left\vert \left\langle F_{ m_{ N, nT +s}} p_{ N, s}^{ (n)}\, ,\, h\right\rangle \right\vert&= \left\vert \left\langle  p_{ N, s}^{ (n)}\, ,\, F_{ m_{ N, nT +s}} h\right\rangle \right\vert,\\
&\leq \left\Vert p_{ N, s}^{ (n)} \right\Vert_{ H_{ \theta}^{ -r+2}} \left\Vert F_{ m_{ N, nT +s}} h \right\Vert_{ H_{ \theta}^{ r-2}} \leq C_{ F} \left\Vert p_{ N, s}^{ (n)} \right\Vert_{ H_{ \theta}^{ -r+2}} \left\Vert h \right\Vert_{ H_{ \theta}^{ r-2}}
\end{align*}
since $F$ is bounded as well as all its derivatives. Secondly,
\begin{align*}
\left\vert \left\langle \langle p_{N,s}^{ (n)},F_{m_{N,nT+s}}\rangle p_{ N, s}^{ (n)}\, ,\, h\right\rangle \right\vert& =\left\vert \left\langle  p_{ N, s}^{ (n)}\, ,\, \langle p_{N,s}^{ (n)},F_{m_{N,nT+s}}\rangle h\right\rangle \right\vert,\\
&\leq \left\Vert F \right\Vert_{ \infty} \left\Vert p_{N,s}^{ (n)} \right\Vert_{ H_{  \theta}^{ -r+2}} \left\Vert h \right\Vert_{ H_{  \theta}^{ r-2}}
\end{align*}
since $p_{ N, s}^{ (n)}$ is a probability measure. Then
\begin{align*}
&\left\Vert \int_0^t e^{(t-s)\cL_N}\nabla\cdot \left\lbrace\left(F_{m_{N,nT+s}}-\langle p_{N,s}^{ (n)},F_{m_{N,nT+s}}\rangle\right)p_{ N, s}^{ (n)}\right\rbrace {\rm d}s \right\Vert_{ H_{ \theta}^{ -r +2}} \\ 
&\leq C_{  \mathcal{ L}} (C_{ F}+ \left\Vert F \right\Vert_{ \infty})\int_0^t \frac{ e^{ - \lambda_\cL(t-s)}}{ (t-s)^{ \frac{ 1}{ 2}}}\left\Vert p_{ N, s}^{ (n)}\right\Vert_{ H_{ \theta}^{ -r +2}}{\rm d}s,\\
&\leq C_{  \mathcal{ L}} (C_{ F}+ \left\Vert F \right\Vert_{ \infty})\int_0^t \frac{ e^{ - \lambda_\cL(t-s)}}{ (t-s)^{ \frac{ 1}{ 2}}}\left\Vert p_{ N, s}^{ (n)} - g_{ N}\right\Vert_{ H_{ \theta}^{ -r +2}}{\rm d}s, \\
&\quad + C_{  \mathcal{ L}} (C_{ F}+ \left\Vert F \right\Vert_{ \infty})\int_0^t \frac{ e^{ - \lambda_\cL(t-s)}}{ (t-s)^{ \frac{ 1}{ 2}}}\left\Vert g_{ N}\right\Vert_{ H_{ \theta}^{ -r+2}}{\rm d}s.
\end{align*}
Then for all $h\in H_{ \theta}^{ r}$, for $X_{ N}\sim g_{ N}$
\begin{align*}
\left\vert \left\langle g_{ N}\, ,\, h\right\rangle \right\vert &= \left\vert \mathbb{ E}\left(h(X_{ N})\right) \right\vert\\
&\leq \left(\mathbb{ E} \left(h(X_{ N})^{ 2}\right)\right)^{ \frac{ 1}{ 2}} = \left(\frac{1}{(2\pi)^\frac{d}{2}\left(1-\frac1N\right)^d \sqrt{\det(K^{-1}\gs^2)}} \int \left\vert h \right\vert^{ 2} w_{ \frac{ N}{ N-1}}\right)^{ \frac{ 1}{ 2}}\\ 
&\leq  \left(\frac{1}{(2\pi)^\frac{d}{2} \sqrt{\det(K^{-1}\gs^2)}} \right)^{ \frac{ 1}{ 2}}\left\Vert h \right\Vert_{ L_{ \theta}^{ 2}}
\end{align*}
since $ \frac{ N}{ N-1}\geq \theta$. Hence, $ \left\vert \left\langle g_{ N}\, ,\, h\right\rangle \right\vert\leq C_{ g} \left\Vert h \right\Vert_{ H_{ \theta}^{ r-2}}$ for a constant independent of $N$, so that $ \left\Vert g_{ N} \right\Vert_{ H_{ \theta}^{ -r+2}}\leq C_{ g}$. Gathering the previous estimates gives, recalling that $\Gamma(1/2)=\sqrt{\pi}$ where $\Gamma(t)=\int_0^\infty x^{t-1}e^{-x}dx$, we obtain
\begin{align}
\left\Vert p_{N,t}^{ (n)}-g_N \right\Vert_{ H_{ \theta}^{ -r+2}} \leq\, & C_0C_\cL  e^{-\lambda_\cL t} + \delta C_{  \mathcal{ L}} C_{ g}(C_{ F}+ \left\Vert F \right\Vert_{ \infty}) \sqrt{ \frac{ \pi}{ \lambda_\cL}}\nonumber \\
&\, + \delta C_{  \mathcal{ L}} (C_{ F}+ \left\Vert F \right\Vert_{ \infty})\int_0^t \frac{ e^{ - \lambda_\cL(t-s)}}{ (t-s)^{ \frac{ 1}{ 2}}}\left\Vert p_{ N, s}^{ (n)} - g_{ N}\right\Vert_{ H_{ \theta}^{ -r+2}}{\rm d}s +N^{-\frac12+\frac{\xi}{3}}. \label{eq:bound mu1-gN}
\end{align}

Take now $ \delta$ such that $ \delta C_{  \mathcal{ L}} C_{ g}(C_{ F}+ \left\Vert F \right\Vert_{ \infty}) \sqrt{ \frac{ \pi}{ \lambda_\cL}} \leq 1$, then 
\begin{align*}
\left\Vert p_{N,t}^{ (n)}-g_N \right\Vert_{ H_{ \theta}^{ -r+2}} \leq\, & C_{ 1} + \delta C_{ 2}\int_0^t \frac{ e^{ - \lambda_\cL(t-s)}}{ (t-s)^{ \frac{ 1}{ 2}}}\left\Vert p_{ N,s}^{ (n)} - g_{ N}\right\Vert_{ H_{ \theta}^{ -r +2}}{\rm d}s
\end{align*}
for $C_{ 1}:= C_{ 0} C_{  \mathcal{ L}} +2$ and $C_{ 2}= C_{  \mathcal{ L}} (C_{ F}+ \left\Vert F \right\Vert_{ \infty})$, independent of $ \delta$. Then, $u_t = e^{ \lambda_\cL t} \big\Vert p_{t}^{ (n)} - g_{ N}\big\Vert_{H^{-r+2}_\theta}$ satisfies
\begin{align*}
u_{ t} \leq\, & C_{ 1}e^{ \lambda_\cL t} + \delta C_{ 2}\int_0^t \frac{ u_{ s}}{ (t-s)^{ \frac{ 1}{ 2}}}{\rm d}s
\end{align*}
and from Lemma~\ref{lem:GH} we deduce that
\begin{equation}
u_t \leq C_{ 1}e^{ \lambda_\cL t} + \delta C_{ 1}C_{ 2}\int_0^t  \left(\frac{1}{\sqrt{ t-s}} + 2\delta C_{ 2} \pi e^{(\delta C_{ 2})^{ 2} \pi (t-s)}\right)e^{ \lambda_\cL s} \dd s.
\end{equation}
This implies directly, recalling the fact that $\big\Vert p_{N, t}^{ (n)} - g_{ N}\big\Vert_{H^{-r+2}_\theta}=e^{- \lambda_\cL t} u_t$, and supposing that $\gd\leq \max( \frac{ 1}{ C_{ 2}} \sqrt{ \frac{ \lambda_\cL}{ 2 \pi}},1)$, for $t\in [0, T]$
\begin{align}
\big\Vert p_{N, t}^{ (n)} - g_{ N}\big\Vert_{H^{-r+2}_\theta} &\leq C_{ 1} + \delta C_{ 1}C_{ 2}\int_0^t  \left(\frac{1}{\sqrt{ t-s}} + 2\delta C_{ 2} \pi e^{(\delta C_{ 2})^{ 2} \pi (t-s)}\right)e^{ -\lambda_\cL (t-s)} \dd s, \nonumber\\
&\leq C_{ 1} + C_{ 1}C_{ 2}\int_0^t  \left(\frac{e^{ -\lambda_\cL (t-s)}}{\sqrt{ t-s}} + 2 C_{ 2} \pi e^{ -\frac{\lambda_\cL}{ 2} (t-s)}\right) \dd s \nonumber\\
&\leq C_{ 1} + C_{ 1}C_{ 2} \left(\sqrt{ \frac{ \pi}{ \lambda_\cL}} + 2 C_{ 2} \pi \frac{ 2}{ \lambda_\cL}\right):= C_{ 4}.\label{aux:C4}
\end{align}
But then, coming back to \eqref{eq:bound mu1-gN}, we obtain
\begin{align}
\left\Vert p_{N,n(T+1)}-g_N \right\Vert \leq\, & C_0C_\cL  e^{-\lambda_\cL T} + \delta C_{  \mathcal{ L}} (C_{ F}+ \left\Vert F \right\Vert_{ \infty}) \sqrt{ \frac{ \pi}{ \lambda_\cL}}\left(C_{ g} +  C_{ 4}\right) +N^{-\frac12+\frac{\xi}{3}},\\
&\leq C_0\left( C_\cL e^{- \lambda_\cL T}+C_5\gd \right)+C_6 \gd +N^{-\frac12+\frac{\xi}{3}}.
\end{align}
where $C_5$ and $C_6$ are positive constants that depend only on $c, C_\cL, C_F, \pi$ and $ \lambda_\cL$ (recall that $C_{ 1}=C_{ 0} C_{  \mathcal{ L}}+2$ and the definition of $C_{ 4}$ in \eqref{aux:C4}). Defining $C_0=\max\{\kappa_0, 2C_6 +2\}$ and supposing $\gd \leq \max\left(\frac{1}{4C_5}, \frac{C_{ 0}}{ 4 C_{ 6}}\right)$, we obtain for $N$ large enough, recalling \eqref{hyp:T 1},
\begin{equation}
\big\Vert p_{N, (n+1)T}-g_N\big\Vert_{H^{-r+2}_\theta}\leq C_0.
\end{equation}
This ends the recursion and implies the result, setting $\kappa_1=C_4$. Proposition~\ref{prop: bound mu1t} is proved.
\end{proof}

\section{Long time dynamics close to periodic solution.}\label{sec:close to M}

The aim of this section is to prove the following Proposition.

\begin{proposition}\label{prop:closeness to M}
Suppose that the hypotheses of Theorem~\ref{th:main} are satisfied. Then there exists an integer  $n_0=n_0(N)$ of order $\log N$ such that for any $\gep>0$, there exists an event $\cC_{\gep,N}$ with $\mathbf{P}(\cC_{\gep,N})\rightarrow 1$ as $N$ goes to infinity and such that when $\cC_{\gep,N}$ is realized,
\begin{equation}
\sup_{0\leq n\leq n_0} \quad \sup_{t\in [0, T]}\quad  \left\Vert \mu_{N,nT+t}-\Gamma_{\Theta(\mu_{N,nT})+t}\right\Vert_{\mathbf{H}^{-r}_\theta}\leq \gep,
\end{equation}
and
\begin{equation}
\sup_{n_0+1\leq n\leq n_F} \quad \sup_{t\in [0,T]} \quad \left\Vert \mu_{N,nT+t}-\Gamma_{\Theta(\mu_{N,nT})+t}\right\Vert_{\mathbf{H}^{-r}_\theta}\leq N^{-\frac12+\xi}.
\end{equation}
\end{proposition}

\subsection{Adjoint of the semigroup}
Recall that $ \Gamma_{ t}=(q_{ t}, \gamma_{ t})$ is the periodic solution to \eqref{eq:syst PDE} mentioned in Theorem~\ref{th:Gamma} and that $\Phi_{t,s}$ is the semigroup of the linearized dynamics in the neighborhood of the periodic solution $\Gamma_t$. This means in particular that, for any $0\leq s\leq t$, for any regular test function $\varphi_t$ and trajectory $\psi_t$ we have
\begin{multline}
\label{eq:nutht}
\llangle[\big] \Phi_{ u+t, u+s}(\nu);(\varphi_t,\psi_t)\rrangle[\big]= \llangle[\big] \nu;(\varphi_s,\psi_s)\rrangle[\big]\\+ \int_{ s}^{t} \llangle[\big]  \nu_{ v}; \partial_{ v}(\varphi_{ v}, \psi_{ v}) + \mathscr{ L}_{ [q, \gamma]}^{ u+v} \left( \varphi_{ v}, \psi_{ v}\right)\rrangle[\big] {\rm d}v,
\end{multline}
where
\begin{align}
\label{eq:Lqgamma}
\mathscr{ L}_{ [q, \gamma]}^{ u}(\varphi, \psi):= \begin{pmatrix}
 \nabla \cdot(\gs^2 \nabla \varphi)+\nabla \varphi\cdot \left(F_{\gamma_{u}}-Kx -\dot \gamma_{u}\right) - F_{ \gamma_{ u}} \cdot \left\langle q_{ u}\, ,\, \nabla \varphi\right\rangle + \psi\cdot F_{\gamma_{u}}\\
\left\langle q_{u}, DF_{\gamma_{u}}^{ \ast}\nabla \varphi\right\rangle - \left\langle q_{ u}\, ,\, \left\langle q_{ u}\, ,\, DF_{ \gamma_{ u}}^{ \ast}\right\rangle \nabla \varphi\right\rangle + \left\langle q_{ u}\, ,\, DF_{ \gamma_{ u}}^{ \ast}\right\rangle \psi
\end{pmatrix},
\end{align}
where $A^*$ denotes the transposed matrix of $A$. For fixed $u$ and $t\geq0$, $f\in \mathcal{ C}_{ c}^{ \infty}( \mathbb{ R}^{ d})$, $ m\in \mathbb{ R}^{ d}$, consider $(\tilde \varphi_{v}, \tilde \psi_{v})$ with initial condition $(\tilde \varphi_{0}, \tilde \psi_{0}):=(f, m)$ solution to 
\begin{equation}
\left\{\begin{array}{rl}
\label{eq:phipsitilde}
\partial_v \tilde \varphi_{v} &=\,\nabla\cdot (\gs^2\nabla \tilde\varphi_v)+\nabla \tilde\varphi_v\cdot (F_{\gamma_{u+t-v}}-Kx-\dot\gamma_{u+t-v})- F_{\gamma_{u+t-v}}\cdot \langle q_{u+t-v},\nabla \tilde\varphi_v\rangle\\
&\qquad\qquad\qquad \qquad\qquad\qquad \qquad\qquad\qquad  \qquad\qquad \qquad\qquad  +F_{\gamma_{u+t-v}}\cdot \tilde\psi_v\\
 \frac{ {\rm d}}{ {\rm d}v}\tilde \psi_v & =\, \langle q_{u+t-v},DF_{\gamma_{u+t-v}}^*\rangle \tilde\psi_v+\langle q_{u+t-v}, DF_{\gamma_{u+t-v}}^*\nabla \tilde\varphi_v\rangle\\
 &\qquad\qquad\qquad \qquad\qquad\qquad \qquad\quad\qquad-\langle q_{u+t-v},DF_{\gamma_{u+t-v}}^*\langle q_{u+t-v}, \nabla\tilde\varphi_v\rangle\rangle
\end{array}
\right.,
\end{equation}
 and consider for $s\leq t$,
\begin{equation}
\label{eq:Psi12}
\Psi^u_{ t,s}:=\Psi^u_{ t,s}(f, m)= \left(\Psi_{t,s}^{u, 1}(f, m), \Psi_{t,s}^{ u,2}(f, m)\right)= (\tilde{ \varphi}_{ t-s}, \tilde{ \psi}_{ t-s}).
\end{equation}
With  this definition, $ s \mapsto \Psi^u_{ t,s}$ solves 
\begin{equation}
\label{eq:Psi_st_L}
\partial_{ s} \Psi^u_{ t,s}= - \mathscr{ L}_{ [q, \gamma]}^{ u+s}( \Psi^u_{ t,s}),\ s\leq t
\end{equation}
with $\Psi^u_{ t,t}=(f, m)$. We focus in this paragraph on the regularity of the adjoint $ \Psi_{ t, s}^{ u}$ defined by \eqref{eq:Psi12}.
\begin{lemma}
\label{lem:control_psit}
For any $g=(f,m)\in \mathbf{H}^r_\theta$, the evolution equation \eqref{eq:phipsitilde} has a unique solution $( \tilde{ \varphi}_{ v}, \tilde{ \psi}_{ v})$ in $ \mathbf{ H}_{ \theta}^{ r}$ such that $ (\tilde{ \varphi}_{ 0}, \tilde{\psi}_{ 0})=g$. Moreover there exists a constant $C_{\Psi}=C_\Psi(T)>0$ such that for all $g\in \mathbf{ H}_{ \theta}^{ r}$, for all $\beta$ satisfying $0\leq \beta\leq 2$, for all $u\in \bbR$ and for all $0<v<s<t\leq T$,
\begin{equation}\label{eq:control_Psi}
\left\Vert \Psi^u_{ t,s}g \right\Vert_{ \mathbf{ H}_{ \theta}^{ r+\beta}}\leq C_\Psi(t-s)^{- \frac{\beta}{2}}\Vert g\Vert_{ \mathbf{ H}^{ r}_{ \theta}}, 
\end{equation}
\begin{equation}\label{eq:control_Psi_diff}
\left\Vert \left(\Psi^u_{ t,v}- \Psi^u_{ s,v}\right)g\right\Vert_{ \mathbf{ H}_{ \theta}^{ r-\beta}} \leq C_\Psi(t-s)^{\frac{\beta}{2}} \left\Vert g\right\Vert_{ \mathbf{ H}_{ \theta}^{ r}},
\end{equation}
and
\begin{equation}
\left\Vert \left(\Psi^u_{ t,v}- \Psi^u_{ s,v}\right)g\right\Vert_{ \mathbf{ H}_{ \theta}^{ r}} \leq C_\Psi(t-s)^{\frac{\beta}{2}}(s-v)^{-\frac{\beta}{2}}\left\Vert g\right\Vert_{ \mathbf{ H}_{ \theta}^{ r}}.
\end{equation}
Finally, for all $0\leq s \leq t$, $u\geq0$, $\nu\in \mathbf{H}_{\theta}^{-r}$ and all $(f,m)\in \mathbf{H}^r_\theta$ 
\begin{equation}
\label{eq:Phi_Psi_adjoint}
\llangle[\big]\Phi_{u+t,u+s}(\nu),(f,m)\rrangle[\big] = \llangle[\big] \nu, \Psi^u_{t,s}(f,m)\rrangle[\big].
\end{equation}
\end{lemma}

\begin{proof}[Proof of Lemma~\ref{lem:control_psit}]
The existence and uniqueness result is classical, see for example \cite{Henry:1981,sell2013dynamics}. Considering now $( \varphi_{ v}, \psi_{ v})= \Psi_{t, v}^{ u}(f, m)$ for the choice of  $g=(f,m)$ with $f\in \mathcal{ C}_{ c}^{ \infty}(\mathbb{ R}^{ d})$  into \eqref{eq:nutht}, we obtain  by \eqref{eq:Psi_st_L} that  $\llangle[\big]\Phi_{u+t,u+s}(\nu,h),g\rrangle[\big] = \llangle[\big] (\nu,h), \Psi^u_{t,s}(g)\rrangle[\big]$. From Theorem~\ref{th:Gamma}, we obtain
\[
\left| \llangle[\big] (\nu, h), \Psi^u_{t,s}(g)\rrangle[\big]\right| \leq C\left(1+(t-s)^{- \frac{ \beta}{2}} e^{ -\lambda(t-s)}\right)\left\Vert (\nu,h)\right\Vert_{\mathbf{H}^{-(r+\beta)}_\theta}\left\Vert g\right\Vert_{\mathbf{H}^r_\theta},
\]
which implies that $g \mapsto \Psi^u_{t,s}(g)$ can be uniquely extended to a continuous operator on $ \mathbf{ H}_{ \theta}^{ r+ \beta}$ which satisfies \eqref{eq:Phi_Psi_adjoint} and whose norm satisfies \eqref{eq:control_Psi}.
For the second bound of the lemma, remark first that
\begin{multline}
 \llangle[\big] \nu,(\Psi^u_{ t,v}- \Psi^u_{ s,v})g\rrangle[\big] = \llangle[\big](\Phi_{u+t,u+v}-\Phi_{u+s,u+v}) \nu,g\rrangle[\big] \\=\llangle[\big](\Phi_{u+t,u+s}-I_d)\Phi_{u+s,u+v}\nu,g\rrangle[\big] 
\end{multline}
But from \cite[Theorem 7.1.3.]{Henry:1981} we have
\begin{equation}
\left\Vert (\Phi_{u+t,u+s}-I_d)\nu\right\Vert_{\mathbf{H}^{-r}_\theta}\leq C_{\beta} (t-s)^{\frac{\beta}{2}}\left\Vert \nu\right\Vert_{\mathbf{H}^{-r+\beta}_\theta},
\end{equation}
and thus
\begin{align}
 \llangle[\big] \nu,(\Psi^u_{ t,v}- \Psi^u_{ s,v})g\rrangle[\big]& \leq C_{\beta} (t-s)^{\frac{\beta}{2}}\left\Vert \Phi_{u+s,u+v}\nu\right\Vert_{\mathbf{H}^{-r+\beta}_\theta}\left\Vert g\right\Vert_{\mathbf{H}^r_\theta}\\
 &\leq C_{\beta'}C' (t-s)^{\frac{\beta}{2}}\left\Vert \nu\right\Vert_{\mathbf{H}^{-r+\beta}_\theta}\left\Vert g\right\Vert_{\mathbf{H}^r_\theta},
\end{align}
which gives the results, since we have the bounds $\left\Vert \Phi_{u+s,u+v}\nu\right\Vert_{\mathbf{H}^{-r+\beta}_\theta}\leq C\left\Vert \nu\right\Vert_{\mathbf{H}^{-r+\beta}_\theta}$ and $\left\Vert \Phi_{u+s,u+v}\nu\right\Vert_{\mathbf{H}^{-r+\beta}_\theta}\leq C(s-v)^{-\frac{\beta}{2}}\left\Vert \nu\right\Vert_{\mathbf{H}^{-r}_\theta}$.
\end{proof}

\subsection{Mild formulation and iterative scheme}

The following Lemma is a direct consequence of the implicit function Theorem (relying on the fact that $t\mapsto \Pi^{\gd,s}_t$ is $C^1$) and the fact that the periodic solution defines a compact manifold $\cM=\{\Gamma_t:\, t\in [0,T_\gd]\}$. This Lemma gives the existence of a projection $\mathrm{proj}$, and we emphasize the fact that the phase obtained with this projection is different from the one given by the map $\Theta$ given by Theorem~\ref{th:Theta}.
\begin{lemma}
There exists a $R_{\mathrm{proj}}>0$ such that for $\mu$ such that $\dist_{\mathbf{H}^{-r}_\theta}(\mu,\cM)\leq R_{\mathrm{proj}} $, there exists a unique $\ga=:\mathrm{proj}(\mu)\in \bbR/(T_\gd \bbZ)$ such that $\Pi^{\gd,s}\left(\mu-\Gamma_\ga\right)=0$, and the mapping $\mu\mapsto \mathrm{proj}(\mu)$ is $C^1$ with derivative uniformly bounded in the $R_{\mathrm{proj}}$-neighborhood of $\cM_\gd$ by a constant $C_{\mathrm{proj}}$.
\end{lemma}
Recall here the estimates of Theorem~\ref{th:Gamma}. We suppose in the following that $T$ satisfies
\begin{equation}\label{hyp:T 2}
C_PC_\Phi e^{-\lambda T}\leq \frac14.
\end{equation}
and we consider the following couple of stopping times (recall that $n_f$ was defined in \eqref{def:nf})
\begin{equation}\label{eq:def ntau tau}
(n_\tau,\tau)=\inf\left\{n\in \{0,\ldots,n_f-1\}\times [0,T]:\, \left\Vert \mu_{nT+\tau}-\Gamma_{\Theta(\mu_{nT})+\tau}\right\Vert_{\mathbf{H}^{-r}_\theta}>R_{\mathrm{proj}}\right\}.
\end{equation}
Then, denoting
\begin{equation}
\tau^n = \left\{
\begin{array}{lll}
T & \text{if} & n<n_\tau\\
\tau & \text{if} & n\geq n_\tau
\end{array}
\right.,
\end{equation}
we define, for $n\in \{1,\ldots,n_f-1\}$,
\begin{equation}
\ga_n = \mathrm{proj}\left(\mu_{N,(n\wedge n_\tau)T}\right).
\end{equation}
as well as
\begin{equation}
\nu_{N,t}^{ (n)}:=(\eta_{N,t}^{ (n)},h_{N,t}^{ (n)})
\end{equation}
where
\begin{equation}
\nu_{N,t}^{ (n)} := \mu_{N,(n\wedge n_\tau)T+t\wedge \tau^n} -\Gamma_{\ga_n+t\wedge\tau^n}.
\end{equation}
Recall below the definition of $ \Psi_{ t, s}^{ u}$ in \eqref{eq:Psi12}.
\begin{lemma}
For all $n\in\{0,\ldots,n_f-1\}$, the process $\left(\nu_{ N, t}^{ (n)} \right)_{ t\in [0, T]}$ satisfies the following equation in $ \mathcal{ C} \left([0, T], \mathbf{ H}_{ \theta}^{ -r}\right)$, written in a mild form: for all $t\in [0,T]$,
\begin{equation}
\label{eq:mild_nuN}
\nu_{N,t}^{ (n)}=\Phi_{\ga_n+t\wedge \tau^n,\ga_n}\nu_{N,0}^{ (n)}+\int_0^{t\wedge \tau^n} \Phi_{\ga_n+t\wedge \tau^n,\ga_n+s} R_{\ga_n,s}(\nu_{N,s}^{ (n)})\dd s +Z_{N,t\wedge \tau^n}^{ (n)},
\end{equation}
where for $\nu=(\eta,h)\in\mathbf{H}^{-r}_\theta$,
\begin{equation}
\label{eq:Ralphas}
R_{\ga,s}(\eta,h) = (R^1_{\ga,s}(\eta,h),R^2_{\ga,s}(\eta,h)),
\end{equation}
with
\begin{align}
R^1_{\ga,s}(\eta,h) :=\, & -\frac1N \nabla\cdot\gs^2\nabla( q_{\ga+s}+\eta) - \delta\nabla\cdot\left(\eta\left(F_{\gamma_{\ga+s}+h}-F_{\gamma_{\ga+s}}\right)\right) \label{eq:R1}\\
&+ \delta\nabla\cdot\left( \eta \langle \eta, F_{\gamma_{\ga+s}+h}\rangle\right) + \delta\nabla\cdot\left( \eta \langle q_{\ga+s}, F_{\gamma_{\ga+s}+h} - F_{ \gamma_{ \alpha+s}}\rangle\right)\nonumber\\
&- \delta\nabla\cdot\left(q_{\ga+s}\left(F_{\gamma_{\ga+s}+h}- F_{\gamma_{\ga+s}} -DF_{\gamma_{\ga+s}}[h]\right) \right)\nonumber\\
&+ \delta\nabla\cdot\big( q_{\ga+s}\big( \langle \eta,F_{\gamma_{\ga+s}+h} - F_{\gamma_{\ga+s}}\rangle \nonumber\\
&\qquad \qquad+  \langle q_{\ga+s},F_{\gamma_{\ga+s}+h}- F_{\gamma_{\ga+s}}-DF_{\gamma_{\ga+s}}[h] \rangle \big)\big)\nonumber,
\end{align}
and
\begin{equation}
R^2_{\ga,s}(\eta,h) :=\,  \delta \left(\langle\eta, F_{\gamma_{\ga+s}+h} - F_{\gamma_{\ga+s}}\rangle + \langle q_{\ga+s}, F_{\gamma_{\ga+s}+h} - F_{\gamma_{\ga+s}}-DF_{\gamma_{\ga+s}}[h] \rangle\right) \label{eq:R2}
\end{equation}
and where $Z_{ N, t}^{ (n)}$ is the limit in $ \mathbf{ H}_{ \theta}^{ -r}$ as $t^{ \prime} \nearrow t$ of $Z_{ N, t^{ \prime}, t}^{ (n)}$ given by, for all $ g\in \mathbf{ H}^{r}_\theta$,
\begin{multline}
\label{eq:ZNttprime}
Z_{N,t',t}^{ (n)} =\frac{\sqrt{2}}{N}\sum_{i=1}^N\int_0^{ t^{ \prime}} \nabla \Psi_{t,v}^{ \ga_n, 1}(g)(Y_{i,v})\cdot \gs \left( \dd B_{i,nT+v} -\frac{1}{N}\sum_{j=1}^N \dd B_{j,nT+v}\right)\\
+\frac{\sqrt{2}}{N}\int_0^{ t^{ \prime}} \Psi_{ t, v}^{ \ga_n, 2}(g) \cdot \gs \sum_{j=1}^N \dd B_{j,nT+v},
\end{multline}
that we denote by
\begin{multline}
Z_{N,t}^{ (n)}(g):= \frac{\sqrt{2}}{N}\sum_{i=1}^N\int_0^{ t} \nabla \Psi_{t,v}^{ \ga_n, 1}(g)(Y_{i,v})\cdot \gs \left( \dd B_{i,nT+v} -\frac{1}{N}\sum_{j=1}^N \dd B_{j,nT+v}\right)\\ +\frac{\sqrt{2}}{N}\int_0^{ t} \Psi_{ t, v}^{ \ga_n, 2}(g) \cdot \gs \sum_{j=1}^N \dd B_{j,nT+v}.
\end{multline}
\end{lemma}

\begin{proof}
We prove this result for $n=0$, (the calculations being identical for $n\geq 1$) and write for $t\leq \tau^0$, $ \nu_{ N, t}= \nu_{ N, t}^{ (0)}=(\eta_{N,t},h_{N,t})= \left(\mu_{N,t}-\Gamma_{\ga_0+t} \right)$. Remark that for any smooth deterministic trajectory $\psi_t$ in $\bbR^d$, we have
\begin{multline}
\label{eq:mN_weak}
m_{N,t}\cdot \psi_t = m_{N,0}\cdot \psi_0+\int_0^t m_{N,s}\cdot \dot \psi_s ds+\int_0^t \psi_s \cdot  \langle \mu_{N,s}, F_{m_{N,s}}\rangle  \dd s \\
+ \frac{\sqrt{2}}{N}\int_0^t \psi_s \cdot \gs \sum_{j=1}^N \dd B_{j,t}.
\end{multline}
Now, since
\begin{multline}
\label{eq:q_weak}
\langle q_{\ga_0+t},\varphi_t\rangle  = \langle q_{\ga_0},\varphi_0\rangle \\+ \int_0^t \left\langle q_{\ga_0+s}, \partial_s \varphi_s+\nabla\cdot (\gs^2\nabla \varphi_s) +\nabla\varphi_s\cdot\left(F_{\gamma_{\ga_0+s}}-Ky -\langle q_{\ga_0+s},F_{\gamma_{\ga_0+s}}\rangle  \right)\right\rangle \dd s,
\end{multline}
and
\begin{equation}
\label{eq:gamma_weak}
\gamma_{\ga_0+t}\cdot \psi_t =\gamma_{\ga_0}\cdot \psi_0+ \int_0^t \gamma_{\ga_0+s}\cdot \dot \psi_s ds + \int_0^t \langle q_{\ga_0+s}, F_{\gamma_{\ga_0+s}}\rangle \cdot \psi_s \dd s,
\end{equation}
we obtain, for any smooth function $\varphi_t$, recalling \eqref{eq:muN_weak}, \eqref{eq:mN_weak}, \eqref{eq:q_weak}, \eqref{eq:gamma_weak} and the definition of $ \nu_{ N}$ and $h_{ N}$:
\begin{align}
\llangle[\big] &\nu_{N,t}; (\varphi_t,\psi_t)\rrangle[\big] \nonumber\\
&=\,  \llangle[\big] \nu_{N,0};(\varphi_0,\psi_0)\rrangle[\big] \\
&\quad +\int_0^t \left\langle \eta_{N,s},\partial_s \varphi_s+\nabla \cdot(\gs^2 \nabla \varphi_s)+\nabla \varphi_s\cdot \left(F_{\gamma_{\ga_0+s}}-Ky -\dot \gamma_{\ga_0+s}\right) \right\rangle \dd s \nonumber\\
&\quad+\int_0^t \left\langle q_{\ga_0+s}, \nabla \varphi_s\cdot\left(DF_{\gamma_{\ga_0+s}}[h_{N,s}]-\langle q_{\ga_0+s},DF_{\gamma_{\ga_0+s}}[h_{N,s}]\rangle-\langle \eta_{N,s},F_{\gamma_{\ga_0+s}}\rangle\right)\right\rangle \dd s\nonumber \\
&\quad+ \int_0^t h_{N,s}\cdot \dot \psi_s ds +\int_0^t \psi_s\cdot \left(\langle \eta_{N,s}, F_{\gamma_{\ga_0+s}}\rangle +\langle q_{\ga_0+s},DF_{\gamma_{\ga_0+s}}[h_{N,s}]\rangle  \right)\dd s\nonumber\\
&\quad+ A_{N,t} + M_{N,t}\nonumber,
\end{align}
where
\begin{align}
A_{N,t} :=& -\frac1N\int_0^t \langle q_{\ga_0+s}+\eta_{N,s}, \nabla \cdot (\gs^2\nabla \varphi_s) \rangle \dd s \nonumber\\
&+\int_0^t\left\langle \eta_{N,s}, \nabla \varphi_s \cdot\left(F_{\gamma_{\ga_0+s}+h_{N,s}}-F_{\gamma_{\ga_0+s}}\right)\right\rangle \dd s\nonumber\\
&- \int_0^t\left\langle \eta_{N,s}, \nabla \varphi_s \cdot\left( \langle q_{\ga_0+s}+\eta_{N,s}, F_{\gamma_{\ga_0+s}+h_{N,s}}\rangle-\dot\gamma_{\ga_0+s}  \right)\right\rangle \dd s \\
&+\int_0^t\left\langle q_{\ga_0+s}, \nabla \varphi_s \cdot\left(F_{\gamma_{\ga_0+s}+h_{N,s}}-F_{\gamma_{\ga_0+s}} -DF_{\gamma_{\ga_0+s}}[h_{N,s}]\right) \right\rangle \dd s \nonumber\\
&-\int_0^t \big\langle q_{\ga_0+s}, \nabla \varphi_s \cdot\big( \langle q_{\ga_0+s}+\eta_{N,s},F_{\gamma_{\ga_0+s}+h_{N,s}}\rangle \nonumber\\
&  \qquad \qquad \qquad \qquad \qquad-\langle q_{\ga_0+s},F_{\gamma_{\ga_0+s}}+DF_{\gamma_{\ga_0+s}}[h_{N,s}]\rangle -\langle \eta_{N,s},F_{\gamma_{\ga_0+s}}\rangle \big)\big\rangle \dd s \nonumber\\
&+ \int_0^t \psi_s\cdot \big(\langle q_{\ga_0+s}+\eta_{N,s}, F_{\gamma_{\ga_0+s}+h_{N,s}}\rangle\nonumber\\
& \qquad \qquad \qquad \qquad \qquad -\langle q_{\ga_0+s},F_{\gamma_{\ga_0+s}}+DF_{\gamma_{\ga_0+s}}[h_{N,s}]\rangle  -\langle \eta_{N,s},F_{\gamma_{\ga_0+s}}\rangle \big)\dd s\nonumber,
\end{align}
and
\begin{equation}
M_{N,t}:= \frac{\sqrt{2}}{N}\sum_{i=1}^N\int_0^t \nabla \varphi_s(Y_{i,s})\cdot \gs \left( \dd B_{i,s } -\frac{1}{N}\sum_{j=1}^N \dd B_{j,s}\right)+\frac{\sqrt{2}}{N}\int_0^t \psi_s \cdot \gs \sum_{j=1}^N \dd  B_{j,t}.
\end{equation}
We finally deduce that
\begin{align}
\llangle[\big] \nu_{N,t}; (\varphi_t,\psi_t)\rrangle[\big] =& \,\llangle[\big] \nu_{N,0};(\varphi_0,\psi_0)\rrangle[\big]\nonumber\\
&+\int_0^t \llangle[\big] \nu_{N,s}; \partial_{ s}(\varphi_{ s}, \psi_{ s})+ \mathscr{ L}_{ [q, \gamma]}^{ \ga_0+s}(\varphi_s,\psi_s)\rrangle[\big]\dd s\\
&+\int_0^t \llangle[\big] R_{\ga_0,s}(\nu_{N,s}); (\varphi_{ s}, \psi_{ s})\rrangle[\big]\dd s+ M_{N,t}.
\end{align}
It remains to choose $(\varphi_s,\psi_s)=\Psi^{\ga_0}_{t,s}(g)$ with $g$ smooth and  to use approximation arguments to obtain the result.
\end{proof}

\subsection{Control of the noise}

Define
\begin{equation}
\label{eq:BN}
\cB_N=\left\{ \max_{n\in \{0,\ldots, n_f-1\}}\sup_{t\in [0,T]}\left\Vert Z_{N,t}^{ (n)}\right\Vert_{\mathbf{H}^{-r}_\theta}\leq N^{-\frac12+\frac{\xi}{3}}\right\}.
\end{equation}
The aim of this section is to obtain the following result.

\begin{proposition}\label{prop:control noise}
Let $(r, \gamma, \theta)$ be fixed as in \eqref{hyp gamma xi}, \eqref{hyp:theta} and \eqref{hyp:r} and recall $ \delta_{ 1}>0$ given by Proposition~\ref{prop: bound mu1t}. Then for all $\gd\in (0,\gd_1)$, $\bbP(\cB_N)\rightarrow 1$ as $N\rightarrow\infty$.
\end{proposition}

The proof of Proposition~\ref{prop:control noise} follows the same steps as the ones given in Section~\ref{sec:first bound}, with $\Psi^\ga_{t,s}$ and $Z_{n,t}$ taking the role of $e^{-(t-s)\cL_N}$ and $W_{n,t}$. For $n\in \{0,\ldots,n_f-1\}$, $0< t^{ \prime}< t\leq T$ and $g=(f, m)\in \mathbf{H}^r_\theta$, recall the definition of $Z_{ N, t^{ \prime}, t}$ in \eqref{eq:ZNttprime} and the definition of $ \Psi_{ t, s}^{ \ga}$ in \eqref{eq:Psi12}.
Hence, for $0< s^{ \prime}< s < t$, $s^{ \prime}< t^{ \prime}< t\leq T$, the difference $Z_{ N, t^{ \prime}, t}^{ (n)} - Z_{N, s^{ \prime}, s}^{ (n)}$ may be written as
\begin{equation}
Z_{N, t^{ \prime}, t}^{ (n)}(g) - Z_{N, s^{ \prime}, s}^{ (n)}(g)= \sum_{ i=1}^{ 3}Z^{(n, i)}_{N, s^{ \prime}, s, t}(g)+ \sum_{ i=4}^{ 6}Z^{ (n, i)}_{ N, s^{ \prime}, t^{ \prime}, t}(g)
\end{equation} 
where
\begin{align*}
Z^{(n, 1)}_{N, s^{ \prime}, s, t}(g)&= \frac{ \sqrt{ 2}}{ N} \sum_{ i=1}^{ N} \int_{0}^{s^{ \prime}} \left[\nabla \Psi_{t,v}^{ \ga_n, 1}(g)- \nabla \Psi_{s,v}^{\ga_n, 1}(g) \right](Y_{ i, nT+v}) \cdot \sigma {\rm d}B_{ i, nT+v},\\
Z^{(n, 2)}_{N, s^{ \prime}, s, t}(g)&= -\frac{ \sqrt{ 2}}{ N} \sum_{ j=1}^{ N} \int_{0}^{s^{ \prime}} \left\lbrace \frac{ 1}{ N}\sum_{ i=1}^{ N}\left[\nabla \Psi_{t,v}^{ \ga_n, 1}(g)- \nabla \Psi_{s,v}^{\ga_n, 1}(g) \right](Y_{ i, nT+v})\right\rbrace \cdot \sigma {\rm d}B_{ j,nT+ v},\\
Z^{(n, 3)}_{N, s^{ \prime}, s, t}(g)&= \frac{ \sqrt{ 2}}{ N} \sum_{ j=1}^{ N} \int_{0}^{s^{ \prime}} \left[\Psi_{t,v}^{\ga_n, 2}(g)- \Psi_{s,v}^{\ga_n, 2}(g) \right]\cdot \sigma  {\rm d}B_{ j, nT+v},\\
Z^{ (n, 4)}_{ N, s^{ \prime}, t^{ \prime}, t}(g)&= \frac{ \sqrt{ 2}}{ N} \sum_{ i=1}^{ N} \int_{ s^{ \prime}}^{t^{ \prime}} \nabla \Psi_{t, v}^{\ga_n, 1}(g)(Y_{ i, nT+v}) \cdot \sigma{\rm d}B_{ i,nT+ v},\\
Z^{ (n, 5)}_{ N, s^{ \prime}, t^{ \prime}, t}(g)&= -\frac{ \sqrt{ 2}}{ N} \sum_{ j=1}^{ N} \int_{ s^{ \prime}}^{t^{ \prime}} \left\lbrace\frac{ 1}{ N} \sum_{ j=1}^{ N} \nabla \Psi_{t, v}^{\ga_n, 1}(g)(Y_{ i, nT+v}) \right\rbrace \cdot \sigma{\rm d}B_{ j, nT+v},\\
Z^{ (n, 6)}_{ N, s^{ \prime}, t^{ \prime}, t}(g)&= \frac{ \sqrt{ 2}}{ N}\int_{ s^{ \prime}}^{t^{ \prime}}\Psi_{t, v}^{\ga_n, 2}(g)\cdot \sigma \sum_{ j=1}^{ N} {\rm d}B_{ j,nT+ v}.
\end{align*}
\begin{lemma}
\label{lem:MNs}
For all $n=0,\ldots,n_f$ the processes $(Z_{N, s^{ \prime}, s, t}^{ (n, i)})_{s^{ \prime}\in [0, s)}$ for $i=1,\ldots, 3$, and $ \left(Z_{N, s^{ \prime}, t^{ \prime}, t}^{ (n, i)}\right)_{t^{ \prime}\in (s^{ \prime}, t)}$ for $i=4,\ldots,6$ are martingales in $ \mathbf{ H}_{ \theta}^{ -r}$, whose trajectories are almost surely continuous. Their Meyer processes are given by
\begin{align*}
\left\langle Z_{ n, \cdot, s, t}^{ (1)}\right\rangle_{ s^{ \prime}}&= 2\sum_{i=1}^{ N} \int_{ 0}^{s^{ \prime}}  \left\Vert \frac{ 1}{ N} \sigma^{ \dagger}\cdot \left[\nabla \Psi_{t,v}^{\ga_n, 1}(\cdot)- \nabla \Psi_{s,v}^{\ga_n, 1}(\cdot) \right](Y_{ i, nT+v}) \right\Vert^{ 2}_{ \mathbf{ H}_{ \theta}^{ -r}} \dd v,\\ 
\left\langle Z_{ n, \cdot, s, t}^{ (2)}\right\rangle_{ s^{ \prime}}&= \frac{ 2}{ N} \int_{ 0}^{s^{ \prime}} \left\Vert \frac{ 1}{ N} \sum_{ i=1}^{ N} \sigma^{ \dagger}\cdot\left[\nabla \Psi_{t,v}^{\ga_n, 1}(\cdot)- \nabla \Psi_{s,v}^{\ga_n, 1}(\cdot) \right](Y_{ i, nT+v})\right\Vert^{ 2}_{  \mathbf{ H}_{ \theta}^{ -r}}  \dd v,\\
\left\langle Z_{ n, \cdot, s, t}^{ (3)}\right\rangle_{ s^{ \prime}}&= \frac{ 2}{ N}\int_{ 0}^{s^{ \prime}}  \left\Vert \sigma^{ \dagger}\cdot \left[\Psi_{t,v}^{\ga_n, 2}(\cdot)- \Psi_{s,v}^{\ga_n, 2}(\cdot) \right] \right\Vert^{ 2}_{  \mathbf{ H}_{ \theta}^{- r}} \dd v,
\end{align*}
and 
\begin{align*}
\left\langle Z_{n, s^{ \prime}, \cdot,t}^{ (4)}\right\rangle_{ t^{ \prime}}&= 2\sum_{i=1}^{ N} \int_{ s^{ \prime}}^{t^{ \prime}} \left\Vert \frac{ 1}{ N} \sigma^{ \dagger}\cdot\nabla \Psi_{t,v}^{\ga_n, 1}(\cdot)(Y_{ i,nT+ v})\right\Vert^{ 2}_{  \mathbf{ H}_{ \theta}^{- r}}  \dd v,\\
\left\langle Z_{ n, s^{ \prime}, \cdot,t}^{ (5)}\right\rangle_{ t^{ \prime}}&= \frac{ 2}{ N} \int_{ s^{ \prime}}^{t^{ \prime}} \left\Vert \frac{ 1}{ N} \sum_{ i=1}^{ N} \sigma^{ \dagger}\cdot\nabla \Psi_{t,v}^{\ga_n, 1}(\cdot)(Y_{ i,nT+ v})\right\Vert^{ 2}_{  \mathbf{ H}_{ \theta}^{ -r}}  \dd v,\\
\left\langle Z_{ n, s^{ \prime}, \cdot,t}^{ (6)}\right\rangle_{ t^{ \prime}}&=  \frac{ 2}{ N}\int_{ s^{ \prime}}^{t^{ \prime}} \left\Vert \sigma^{ \dagger}\cdot \Psi_{t,v}^{\ga_n, 2}(\cdot)\right\Vert^{ 2}_{  \mathbf{ H}_{ \theta}^{ -r}}  \dd v.
\end{align*}
\end{lemma}
\begin{proof}[Proof of Lemma~\ref{lem:MNs}]
The proof of this lemma is similar to the proof of Lemma~\ref{lem:VNi}, remarking for example that for $(g_{ l})_{ l\geq0}$ a complete orthonormal basis in $ \mathbf{ H}_{ \theta}^{ r}$ we have
\begin{align}
\sum_{l\geq 0}\mathbf{E}\left[\left\langle Z_{ 0, \cdot,s,t}^{ (1)}(g_{ l})\right\rangle_{ s^{ \prime}}\right]&=2\sum_{l\geq 0}\sum_{i=1}^{ N} \int_{ 0}^{s^{ \prime}}  \left\vert \frac{ 1}{ N} \sigma^{ \dagger}\cdot \left[\nabla \Psi_{t,v}^{\ga_0, 1}(g_{ l})- \nabla \Psi_{s,v}^{\ga_0, 1}(g_{ l}) \right](Y_{ i, v}) \right\vert^{ 2} {\rm d}v\nonumber\\
&=2\sum_{i=1}^{ N} \int_{ 0}^{s^{ \prime}}  \mathbf{ E} \left[\left\Vert \frac{ 1}{ N} \sigma^{ \dagger}\cdot \left[\nabla \Psi_{t,v}^{\ga_0, 1}(\cdot)- \nabla \Psi_{s,v}^{\ga_0, 1}(\cdot) \right](Y_{ i, v}) \right\Vert^{ 2}_{ \mathbf{ H}_{ \theta}^{ -r}}\right] {\rm d}v \nonumber\\
&\leq \frac{C}{N}\int_{ 0}^{s^{ \prime}}  \frac{ 1}{ N} \sum_{ i=1}^{ N}  \mathbf{ E} \left[w_{ -\frac{4 \theta}{ 3} }(Y_{ i, v})\right]\dd v <\infty,
\end{align}
where we have used in particular the inequality
\begin{equation}
\left\Vert \frac{ 1}{ N} \sigma^{ \dagger}\cdot \left[\nabla \Psi_{t,v}^{\ga_0, 1}(\cdot)- \nabla \Psi_{s,v}^{\ga_0, 1}(\cdot) \right](Y_{ i, v}) \right\Vert^{ 2}_{ \mathbf{ H}_{ \theta}^{ -r}}\leq \frac{C'}{N^2}w_{ -\frac{4 \theta}{ 3}} \left(Y_{ i, v}\right).
\end{equation}
Indeed, for any $g\in \mathbf{H^r_\theta}$, using Lemma~\ref{lem:control_delta_x} with $ \eta=1$ and Lemma~\ref{lem:control_psit}:
\begin{align}
\left|\sigma^{ \dagger}\cdot \left[\nabla \Psi_{t,v}^{\ga_0, 1}(g)- \nabla \Psi_{s,v}^{\ga_0, 1}(g) \right](Y_{ i, v})\right|
&\leq C_1\left\Vert \gd_{Y_{i,v}}\right\Vert_{ \mathbf{ H}_{ \theta}^{ -(r+1)}}\left\Vert \nabla \Psi_{t,v}^{\ga_0, 1}(g)- \nabla \Psi_{s,v}^{\ga_0, 1}(g)  \right\Vert_{ \mathbf{ H}_{ \theta}^{ r+1}}\nonumber\\
&\leq C_2\left\Vert \gd_{Y_{i,v}}\right\Vert_{ \mathbf{ H}_{ \theta}^{ -(r+1)}}\left\Vert  \Psi_{t,v}^{\ga_0, 1}(g)- \Psi_{s,v}^{\ga_0, 1}(g)  \right\Vert_{ \mathbf{ H}_{ \theta}^{ r}}\\
&\leq C ' w_{ -\frac{2 \theta}{3}} \left(Y_{ i, v}\right) \Vert g\Vert_{ \mathbf{ H}_{ \theta}^{ r}}.\nonumber
\end{align}
Here we have considered $\gd_y$ as an element of $\mathbf{H}^{-r}_\theta$, with
\begin{equation}
\llangle \gd_y; (\varphi,\psi)\rrangle = \varphi(y).
\end{equation}
The rest of the proof follows from arguments identical to the ones given for Lemma~\ref{lem:VNi}.
\end{proof}

\begin{lemma}
\label{lem:ZNt}
Let $(r, \gamma, \theta)$ be fixed as in \eqref{hyp:r} and \eqref{hyp:theta} and recall $ \delta_{ 1}>0$ given by Proposition~\ref{prop: bound mu1t}. Then, there exists a constant $\kappa_4=\kappa_4(T, \delta_{ 1})>0$ such that for all $\gd\in (0,\gd_1)$, for any $0\leq  s < t\leq T$ and any $n=0,\ldots,n_f-1$, one has
\begin{equation}
 \mathbf{ E} \left[\left\Vert Z_{ N, t}^{ (n)} - Z_{ N,  s}^{ (n)}\right\Vert_{\mathbf{ H}_{ \theta}^{-r}}^{ 2m}\right]\leq \frac{ \kappa_4}{ N^{ m}} (t-s)^{ \frac{ 3m}{ 4}}.
\end{equation} 
\end{lemma}

\begin{proof}[Proof of Lemma~\ref{lem:ZNt}]
The proof is similar as the one of Lemma~\ref{lem:bound Vnt-Vns} and boils down to estimating (for $n=0$) $ \mathbf{ E} \left[\left\Vert Z_{ N}^{ (0, i)} \right\Vert_{ \mathbf{ H}_{ \theta}^{-r}}^{ 2m}\right] $ for all $i=1,\ldots, 6$. For the term $Z^{(0, 1)}_N$ we get
\begin{align}
\mathbf{ E} \left[ \left\Vert Z_{N, s^{ \prime}, s, t}^{ (0, 1)} \right\Vert^{ 2m}_{ \mathbf{ H}_{ \theta}^{ -r}}\right]\leq \frac{C_1}{ N^{m}}  \mathbf{ E} \left[ \left( \frac{ 1}{ N}\sum_{i=1}^{ N} \int_{ 0}^{s^{ \prime}} \left\Vert \left[\nabla \Psi_{t,v}^{\ga_0, 1}(\cdot)- \nabla \Psi_{s,v}^{\ga_0, 1}(\cdot)  \right](Y_{ i, v}) \right\Vert_{\mathbf{ H}^{-r}_{ \theta}}^{ 2} {\rm d}v\right)^{ m}\right]. \label{eq:bound_norm_MN1}
\end{align}
But applying Lemma~\ref{lem:control_delta_x} with $ \eta=1$ and Lemma~\ref{lem:control_psit} we obtain, almost surely for all $g\in \mathbf{ H}^{ r}_{ \theta}$,
\begin{align}
\left| \left[\nabla \Psi_{t,v}^{\ga_0, 1}(g)- \nabla \Psi_{s,v}^{\ga_0, 1}(g) \right](Y_{ i, v})\right|
&\leq C_{1,r-3}\left\Vert \gd_{Y_{i,v}}\right\Vert_{ \mathbf{ H}_{ \theta}^{ -\left(r-3\right)}}\left\Vert \nabla \Psi_{t,v}^{\ga_0, 1}(g)- \nabla \Psi_{s,v}^{\ga_0, 1}(g)  \right\Vert_{ \mathbf{ H}_{ \theta}^{ r-3}}\nonumber\\
&\leq C_2\left\Vert \gd_{Y_{i,v}}\right\Vert_{ \mathbf{ H}_{ \theta}^{ -\left(r-3\right)}}\left\Vert  \Psi_{t,v}^{\ga_0, 1}(g)- \Psi_{s,v}^{\ga_0, 1}(g)  \right\Vert_{ \mathbf{ H}_{ \theta}^{ r-2}}\\
&\leq C_3 w_{ -\frac{ 2 \theta}{3}} \left(Y_{ i, v}\right)(t-s) \Vert g\Vert_{ \mathbf{ H}_{ \theta}^{ r}}.\nonumber
\end{align}
Hence,
\begin{equation}
\label{eq:gradient_diff_Psi}
\left\Vert \left[\nabla \Psi_{t,v}^{\ga_0, 1}(\cdot)- \nabla \Psi_{s,v}^{\ga_0, 1}(\cdot)  \right](Y_{ i, v}) \right\Vert_{\mathbf{ H}^{-r}_{ \theta} }\leq C_3 w_{ -\frac{2 \theta}{3}} \left(Y_{ i, v}\right) (t-s),
\end{equation}
and going back to \eqref{eq:bound_norm_MN1} and relying on Lemma~\ref{lem:control_exp_Yi}, we obtain
\begin{align}
\mathbf{ E} \left[ \left\Vert Z_{ N, s^{ \prime}, s, t}^{ (0, 1)} \right\Vert^{ 2m}_{  \mathbf{ H}_{ \theta}^{ -r}}\right]&\leq \frac{C_4(t-s)^{2m}}{ N^{m}}  \mathbf{ E} \left[ \left(\int_{ 0}^{s^{ \prime}}  \frac{ 1}{ N}\sum_{i=1}^{ N} w_{ -\frac{4 \theta}{3}} \left(Y_{ i, v}\right) {\rm d}v\right)^{ m}\right]\\
&\leq \frac{C_4(t-s)^{2m}T^{m-1}}{ N^{m}}\mathbf{ E} \left[ \int_{ 0}^{s^{ \prime}}  \frac{ 1}{ N}\sum_{i=1}^{ N} w_{ -\frac{4 m \theta}{3}} \left(Y_{ i, v}\right) {\rm d}v\right]\leq \frac{C_5}{ N^{m}}(t-s)^{2m}.
\end{align}
The same analysis gives the same bound for $\mathbf{ E} \left[ \left\Vert Z_{ N, s^{ \prime}, s, t}^{ (0, 2)} \right\Vert^{ 2m}_{ \mathbf{ H}_{ \theta}^{ r}}\right]$.
Concerning $Z_{ N}^{ (0, 3)}$, applying Lemma~\ref{lem:control_psit} with $\beta= \frac{ 3}{ 4}$, we get
\begin{equation}
\left\Vert \Psi_{t,v}^{\ga_0, 2}(\cdot)- \Psi_{s,v}^{\ga_0, 2}(\cdot) \right\Vert_{  \mathbf{ H}_{ \theta}^{ -r}} \leq C_{\Psi} (t-s)^{ \frac{ 3}{ 8}}(s-v)^{- \frac{ 3}{ 8}},
\end{equation}
which leads to
\begin{align}
\mathbf{ E} \left[ \left\Vert Z_{ N, s^{ \prime}, s, t}^{ (0, 3)} \right\Vert^{ 2m}_{  \mathbf{ H}_{ \theta}^{ -r}}\right]&\leq \frac{ C_6(t-s)^{ \frac{ 3m}{ 4}}}{ N^{ m}}\left(\int_{ 0}^{s^{ \prime}} (s-v)^{- \frac{ 3}{ 4}} {\rm d}v\right)^{ m} \leq \frac{C_7}{N^m}(t-s)^{ \frac{ 3m}{ 4}}. \label{eq:bound_def_M3}
 \end{align}
For $Z^{(0, 4)}_N$, using again of Lemma~\ref{lem:control_delta_x} and Lemma~\ref{lem:control_psit}, we obtain
\begin{align}
\left|\nabla \Psi^{\ga_0,1}_{t,v}(g)(Y_{i,v}) \right|
&\leq C_{1,r-1}\left\Vert \gd_{Y_{i,v}}\right\Vert_{ \mathbf{ H}_{ \theta}^{ -\left(r-1\right)}}\left\Vert \nabla \Psi_{t,v}^{\ga_0, 1}(g)  \right\Vert_{ \mathbf{ H}_{ \theta}^{ r-1}}\nonumber\\
&\leq C_8\left\Vert \gd_{Y_{i,v}}\right\Vert_{ \mathbf{ H}_{ \theta}^{ -\left(r-1\right)}}\left\Vert \Psi_{t,v}^{\ga_0, 1}(g)  \right\Vert_{ \mathbf{ H}_{ \theta}^{ r}}\leq C_9 w_{ -\frac{ 2 \theta}{3}} \left(Y_{ i, v}\right) \Vert g\Vert_{ \mathbf{ H}_{ \theta}^{ r}},
\end{align}
so that
\begin{equation}
\left\Vert\nabla \Psi_{t,v}^{\ga_0, 1}(\cdot)(Y_{ i, v}) \right\Vert_{\mathbf{ H}^{-r}_{ \theta} }\leq C_9 w_{ -\frac{2 \theta}{3}} \left(Y_{ i, v}\right) ,
\end{equation}
and we deduce, by H\"older inequality with exponents $p= 4$ and $q= \frac{ 4}{ 3}$ (recall that $ \frac{ m}{ 4}>1$)
\begin{align}
\mathbf{ E} \left[ \left\Vert Z_{N, s^{ \prime}, s, t}^{ (0, 4)} \right\Vert^{ 2m}_{  H_{ \theta}^{ -r}}\right]&\leq \frac{C_{10}}{ N^{m}}  \mathbf{ E} \left[ \left(\int_{ s'}^{t^{ \prime}}  \frac{ 1}{ N}\sum_{i=1}^{ N} w_{ -\frac{4 \theta}{3}} \left(Y_{ i, v}\right) {\rm d}v\right)^{ m}\right]\nonumber\\
&\leq \frac{C_{11}T^{ \frac{ m}{ 4}-1}}{ N^{m}}(t'-s')^{ \frac{ 3m}{ 4}}\mathbf{ E} \left[ \int_{ s'}^{t^{ \prime}}  \frac{ 1}{ N}\sum_{i=1}^{ N} w_{ -\frac{4 m \theta}{3}} \left(Y_{ i, v}\right) {\rm d}v\right]\nonumber\\
&\leq \frac{C_{12}}{ N^{m}}(t'-s')^{ \frac{ 3m}{ 4}}.
\end{align}
The term $Z^{(0, 5)}_N$ can be treated similarly. Concerning $Z^{(0, 6)}_N$ since 
$
\left\Vert \Psi_{t,v}^{\ga_n 2}(\cdot)\right\Vert_{  \mathbf{ H}_{ \theta}^{ -r}} \leq C_\Psi 
$, we get directly
\begin{equation}
\mathbf{ E} \left[ \left\Vert Z_{ N, s^{ \prime}, s, t}^{ (0, 6)} \right\Vert^{ 2m}_{  \mathbf{ H}_{ \theta}^{ -r}}\right]\leq \frac{C_{13}}{N^m}(t'-s')^m.
\end{equation}
All the estimates above lead to
\begin{equation}
\mathbf{ E} \left[\left\Vert Z_{ n,t', t} - Z_{ n,s',  s}\right\Vert_{\mathbf{ H}_{ \theta}^{-r}}^{ 2m}\right]\leq\frac{C_{14}}{N^m}\left( (t-s)^{ \frac{ 3m}{ 4}}+(t'-s')^{ \frac{ 3m}{ 4}}\right),
\end{equation}
which implies the result, setting $\kappa_4=2C_{14}$.
\end{proof}

\begin{proof}[Proof of Proposition~\ref{prop:control noise}]
This proof is the same as the one of Proposition~\ref{prop: bound mu1t}, applying Lemma~\ref{lem:GRR} with $\phi(t)=Z_{N,t}^{ (n)}$, $\chi(u) = u^{\frac14+\frac{1}{8m}}$ and $\psi(u)=u^{2m}$.
\end{proof}

\subsection{Proximity to $\cM_\gd$.}
Before proving Proposition~\ref{prop:closeness to M} we first remark that the error term $R_{\ga,s}$ defined in \eqref{eq:Ralphas} is either quadratic or of order $\frac1N$, as stated in the following lemma.
\begin{lemma}
\label{lem:Rnu}
There exists a constant $C_R$ such that for all $s,\ga\in \bbR$ and $\nu=(\eta, h)\in \mathbf{H}^{-(r-1)}_\theta$,
\begin{equation}
\label{eq:control_R}
\left\Vert R_{\ga,s}(\nu)\right\Vert_{\mathbf{H}^{-(r+1)}_\theta}\leq C_R\left(\frac{1}{N}\left(1+\left\Vert\eta \right\Vert_{H^{-(r-1)}_\theta}\right) +\gd\left\Vert\nu \right\Vert_{\mathbf{H}^{-r}_\theta}^2\right).
\end{equation}
\end{lemma}

\begin{proof}[Proof of Lemma~\ref{lem:Rnu}]
Let $ \eta\in H_{ \theta}^{ -(r-1)}$, $h\in \mathbb{ R}^{ d}$ and $ \nu:=(\eta, h)\in \mathbf{ H}_{ \theta}^{ -(r-1)}$. Then, using \eqref{eq: nabla u H'}
\begin{align}
\left\Vert \nabla\sigma^{2}\nabla( q_{\ga+s}+\eta) \right\Vert_{ H_{ \theta}^{ -(r+1)}} \leq C \left(\left\Vert q_{ \alpha+s} \right\Vert_{ H_{ \theta}^{ -(r-1)}} + \left\Vert  \eta \right\Vert_{ H_{ \theta}^{ -(r-1)}}\right)\leq C^{ \prime}\left(1 + \left\Vert  \eta \right\Vert_{ H_{ \theta}^{ -(r-1)}}\right),
\end{align} 
since $ s \mapsto q_{ \alpha+s}$ is periodic. Similarly, for every $ \varphi\in H_{ \theta}^{ r+1}$, 
\begin{align*}
\left\vert \left\langle \nabla\cdot \eta\left(F_{\gamma_{\ga+s}+h}-F_{\gamma_{\ga+s}}\right)\, ,\, \varphi\right\rangle \right\vert&=\left\vert \left\langle \eta\, ,\, \left(F_{\gamma_{\ga+s}+h}-F_{\gamma_{\ga+s}}\right)\cdot \nabla\varphi\right\rangle \right\vert,\\
&\leq C\left\Vert \eta \right\Vert_{ H_{ \theta}^{ -r}} \left\Vert \left(F_{\gamma_{\ga+s}+h}-F_{\gamma_{\ga+s}}\right)\cdot \nabla\varphi \right\Vert_{ H_{ \theta}^{ r}}.
\end{align*}
Writing $F_{ \gamma_{ \alpha+s}+h}(x) - F_{ \gamma_{ \alpha +s}}(x)= \int_{ 0}^{1} DF_{ \gamma_{ \alpha+s}}(x+u h)[h] {\rm d}u$ and using that 
 by assumption, $F$ is bounded as well as all its derivatives, we obtain
 \begin{align*}
\left\vert \left\langle \nabla\cdot \eta\left(F_{\gamma_{\ga+s}+h}-F_{\gamma_{\ga+s}}\right)\, ,\, \varphi\right\rangle \right\vert&\leq C\left\Vert \eta \right\Vert_{ H_{ \theta}^{ -r}} \left\vert h \right\vert \left\Vert \nabla\varphi \right\Vert_{ H_{ \theta}^{ r}}\leq C\left\Vert \nu \right\Vert_{ \mathbf{ H}_{ \theta}^{ -r}}^{ 2} \left\Vert \varphi \right\Vert_{ H_{ \theta}^{ r+1}}
 \end{align*}
 Hence
\begin{equation}
\left\Vert \delta\nabla\cdot\left(\eta\left(F_{\gamma_{\ga+s}+h}-F_{\gamma_{\ga+s}}\right)\right) \right\Vert_{ H_{ \theta}^{ -(r+1)}} \leq \delta C \left\Vert \nu \right\Vert_{ \mathbf{ H}_{ \theta}^{ -r}}^{ 2}.
\end{equation}
The next two other terms in the definition \eqref{eq:R1} of $R^{ 1}$ can be treated in a same way and lead to the same upper bound.
Moreover, by regularity of $F$, for all $x\in \mathbb{ R}^{ d}$, $F_{ \gamma_{ \alpha +s}+h}(x)- F_{ \gamma_{ \alpha+s}}(x) - DF_{ \gamma_{ \alpha+s}}(x)[h] = \frac{ 1}{ 2}\int_{ 0}^{1} D^{ 2}F_{ \gamma_{ \alpha+s}}(x+uh)[h, h]{\rm d}u:= U(x)[h,h]$. Then, si $D^{ 2}F$ is uniformly bounded and all its derivatives,
\begin{align}
\left\Vert \delta\nabla\cdot\left(q_{\ga+s}\left(F_{\gamma_{\ga+s}+h}- F_{\gamma_{\ga+s}} -DF_{\gamma_{\ga+s}}[h]\right) \right) \right\Vert_{ H_{ \theta}^{ -(r+1)}}&= \delta\left\Vert \nabla \cdot \left(q_{ \alpha +s} U(\cdot)[h, h]\right) \right\Vert_{ H_{ \theta}^{ -(r+1)}}\nonumber\\
&\leq \delta C \left\vert h \right\vert^{ 2},
\end{align}
for some constant $C>0$. The treatment of the last terms in the definitions \eqref{eq:R1} and \eqref{eq:R2} of $R^{ 1}$ and $R^{ 2}$ is the same and leads to a contribution of order $ \delta C\left\Vert \nu \right\Vert_{ \mathbf{ H}_{ \theta}^{ -r}}^{ 2}$.
\end{proof}

We have now all the ingredients to prove Proposition~\ref{prop:closeness to M}.

\begin{proof}[Proof of Proposition~\ref{prop:closeness to M}]
Let us define the following event, with probability converging to $1$ as $N\to\infty$, by Propositions~\ref{prop: bound mu1t} and \ref{prop:control noise},
\begin{equation}
\cC_{\gep,N}=\left\{\left\Vert \mu_{N,0}-\Gamma_{u_0}\right\Vert_{\mathbf{H}^{-r}_\theta}\leq \gep\right\}\cap \cB_N\cap \left\lbrace \sup_{0\leq t\leq Nt_f} \left\Vert p_{N,t}- g_{ N}\right\Vert_{H_\theta^{-(r-1)}}\leq \kappa_1\right\rbrace.
\end{equation}

\textbf{Step 1: approaching the manifold.}
For $\gep>0$ we consider $\gep_0$ small enough such that $\frac{\gep_0}{2C_\Phi(1+C_\Gamma )}<\gep$ (we will fix the value of $ \varepsilon_{ 0}>0$ later). Define for all $n\geq 1$ $ h_{ n}:= \frac{ \varepsilon_{ 0}}{ 2^{ n-1}}$ and set 
\begin{equation}
n_0:= \inf \left\lbrace n,\ h_{ n}< N^{ - \frac{ 1}{ 2}+ \frac{ 2 \xi}{ 3}}\right\rbrace.
\end{equation}
Note that $ n_0$ is of order $ O(\log(N))$. We suppose that $ \cC_{\gep',N}$ is satisfied for $\gep'>0$ small enough, so that $ \left\Vert  \nu_{ N, 0}^{ (1)} \right\Vert_{ \mathbf{ H}_{  \theta}^{ -r}}\leq h_{ 1} $  and then proceed by induction to show that we have $ \left\Vert  \nu_{ N, 0}^{ (n)} \right\Vert_{ \mathbf{ H}_{  \theta}^{ -r}}\leq h_{ n} $ and that we have the bound $\sup_{t\in [0,T]}\left\Vert  \nu_{ N, t}^{ (n)} \right\Vert_{ \mathbf{ H}_{ \theta}^{ -r}}\leq  2 C_{ \Phi} h_{ n}$. Remark that this implies in particular, since $\Theta\left(\mu^N_{nT}\right)=\mathrm{proj}\left(\Gamma_{\Theta\left(\mu^N_{nT}\right)}\right)$, that
\begin{equation}
\left\Vert \mu_{N,nT+t}-\Gamma_{\Theta\left(\mu^N_{nT}\right)+t}\right\Vert_{\mathbf{H}^{-r}_\theta}\leq \left\Vert \nu^n_{N,t}\right\Vert_{\mathbf{H}^{-r}_\theta}+\left\Vert \Gamma_{\ga_n+t}-\Gamma_{\Theta\left(\mu^N_{nT}\right)+t}\right\Vert_{\mathbf{H}^{-r}_\theta},
\end{equation}
which means, since $\ga_n-\Theta\left(\mu^N_{nT}\right)=\Theta\left(\Gamma_{\ga_n}\right)-\Theta\left(\gamma_{\ga_n}+\nu^{n}_{N,0}\right)$, that
\begin{equation}
\left\Vert \mu_{N,nT+t}-\Gamma_{\Theta\left(\mu^N_{nT}\right)+t}\right\Vert_{\mathbf{H}^{-r}_\theta}\leq \left(1+C_\Gamma\right)2 C_\Phi h_n\leq \left(1+C_\Gamma\right)2 C_\Phi \gep_0\leq \gep,
\end{equation}
which is the first point of Proposition~\ref{prop:closeness to M}.

\medskip

By assumption, one has that $ \left\Vert  \nu_{ N, 0}^{ (1)} \right\Vert_{ \mathbf{ H}_{ \theta}^{ -r}} \leq h_{ 1}$. Suppose by inductive hypothesis that $ \left\Vert  \nu_{ N, 0}^{ (n)} \right\Vert_{ \mathbf{ H}_{  \theta}^{ -r}}\leq h_{ n} $ for some $n$ and define
\begin{equation}
t_n:= \inf \left\lbrace t\in[0,  T],\ \left\Vert  \nu_{ N, t}^{ (n)} \right\Vert_{ \mathbf{ H}_{ \theta}^{ -r}} \geq 2 C_{ \Phi} h_{ n}\right\rbrace
\end{equation}
Then, $ t_n>0$ and from Theorem~\ref{th:Gamma}, \eqref{eq:mild_nuN} and Lemma~\ref{lem:Rnu}, we obtain
\begin{align}
\left\Vert \nu_{N,t}^{ (n)}\right\Vert_{\mathbf{H}^{-r}_{\theta}}&\leq  C_\Phi e^{-\lambda t}h_{ n}\nonumber\\
&+C_\Phi C_R\int_0^t  \left\lbrace \left(1+ \frac{e^{-\lambda(t-s)}}{\sqrt{t-s}}\right)\left(\frac{1}{N}\left(1+\left\Vert  \eta_{N,s}^{ (n)} \right\Vert_{H^{-(r-1)}_\theta}\right) +\gd \left\Vert\nu_{N,s}^{ (n)} \right\Vert_{\mathbf{H}^{-r}_\theta}^{ 2} \right)\right\rbrace {\rm d}s\nonumber\\
&+\left\Vert Z_{N,t\wedge \tau^n}^{ (n)} \right\Vert_{\mathbf{H}^{-r}_\theta}.\label{eq:bound mild nun}
\end{align}
Using the definition of $\cC_{\gep,N}$ and the fact that $\Gamma_t$ is bounded in $\mathbf{H}^{-r+1}_\theta$, we have the uniform a priori bound 
\begin{equation}
\sup_{ s\in [0,  T]} \sup_{ n\geq 0}\left\Vert \eta_{N,s}^{ (n)} \right\Vert_{H^{-(r-1)}_\theta}\leq \kappa_{ 1} + \sup_{ s\geq0}\left\Vert g_{ N} - q_{s} \right\Vert_{ H_{ \theta}^{ -(r-1)}}:=  \kappa_{ 5}.
\end{equation}
Then, for $t\leq t_n$
\begin{align}
\left\Vert  \nu_{N,t}^{ (n)}\right\Vert_{\mathbf{H}^{-r}_{\theta}}&\leq  C_\Phi e^{-\lambda t}h_{ n}+C_\Phi C_R \left(\frac{1}{N}\left(1+  \kappa_{ 5}\right) +\gd h_{ n}^{ 2} \right) \left( T+ \sqrt{ \frac{ \pi}{ \lambda}}\right)+N^{-\frac12+\frac{\xi}{3}},
\end{align}
by the control on the noise we have obtained in Proposition~\ref{prop:control noise} (recall the definition of $ \mathcal{ B}_N$ in \eqref{eq:BN}). Note that since $n \leq n_0$, we have $ h_{ n} \geq N^{ - \frac{ 1}{ 2} + \frac{ 2 \xi}{ 3}}$ and thus $N^{ - \frac{ 1}{ 2} + \frac{ \xi}{ 3}} \leq h_{ n}N^{ -\frac{ \xi}{ 3}}$ as well as $ \frac{ 1}{ N}\leq h_{ n} N^{ - \frac{ 1}{ 2} - \frac{ 2 \xi}{ 3}}$. Hence, we obtain from the last inequality that
\begin{align}
\left\Vert \nu_{N,t}^{ (n)}\right\Vert_{\mathbf{H}^{-r}_{\theta}}&\leq  C_\Phi e^{-\lambda t}h_{ n}+ C(N, \varepsilon_{ 0}) h_{ n},
\end{align}
with $C(N, \varepsilon_{ 0})\to 0$ as $N\to\infty$ and $ \varepsilon_{ 0}\to 0$. Hence, choosing $N$ sufficiently large and $ \varepsilon_{ 0}\to 0$ such that $C(N, \varepsilon_{ 0})<1$, we obtain that $t_n\geq T$ and we obtain
\begin{align}
\left\Vert  \nu_{N,  T}^{ (n)}\right\Vert_{\mathbf{H}^{-r}_{\theta}}&\leq  C_\Phi e^{-\lambda  T}h_{ n}+ C(N, \varepsilon_{ 0}) h_{ n}\leq \frac{ h_{ n}}{ 4 C_{ P}} + C(N, \varepsilon_{ 0})h_{ n} \leq \frac{ 3}{ 8 C_{ P}} h_{ n},
\end{align}
choosing $N$ large enough and $ \varepsilon_{ 0}$ small enough. It remains to show that $ \left\Vert  \nu_{ N, 0}^{ (n+1)} \right\Vert_{ \mathbf{ H}_{ \theta}^{ -r}} \leq \frac{ h_{ n}}{ 2}$ in order to conclude the recursion. Namely, since $ \nu_{N,0}^{ (n+1)}=\Gamma_{ \ga_{n+1}}+ \nu_{N,  T}^{ (n)}-\Gamma_{ \ga_n+  T}$ and $P^{st}_{ \ga_{n+1}} \nu_{N,0}^{ (n+1)}= \nu_{N,0}^{ (n+1)}$, we have
\begin{equation}\label{eq:decom Pst nu n+1}
 \nu_{N,0}^{ (n+1)} = P^{st}_{ \ga_{n+1}}\left(\Gamma_{ \ga_{n+1}}-\Gamma_{ \ga_n+ T} \right)+\left( P^{st}_{  \ga_{n+1}}-P^{st}_{ \ga_n+T}\right)\nu_{N,T}^{ (n)}+P^{st}_{ \ga_n+  T} \nu_{N,  T}^{ (n)}.
\end{equation}
Let us estimates these three terms separately. For the last term we directly get
\begin{align}\label{eq:term1}
\left\Vert P^{st}_{ \ga_n+ T}  \nu_{N, T}^{ (n)}\right\Vert_{\mathbf{H}^{-r}_{\theta}}&\leq C_P\left\Vert  \nu_{N, T}^{ (n)}\right\Vert_{\mathbf{H}^{-r}_{\theta}} \leq \frac{ 3}{ 8} h_{ n}
\end{align}
For the first term, relying on the regularity in time of $\Gamma_t$ provided by Theorem~\ref{th:Gamma}, we get
\begin{equation}
\left\Vert \Gamma_{\ga+u}-\Gamma_{\ga}-u\, \partial_\ga \Gamma_{\ga}\right\Vert_{\mathbf{H}^{-r}_\theta}\leq C_\Gamma |u|^2,
\end{equation}
we have, recalling that $P^{st}_{ \ga_{n+1}}\partial_{\ga} \Gamma_{ \ga_{n+1}}=0$,
\begin{equation}\label{eq:bound norm Pst delta Gamma}
 \left\Vert P^{st}_{ \ga_{n+1}}\left(\Gamma_{ \ga_{n+1}}-\Gamma_{ \ga_n+T} \right)\right\Vert_{\mathbf{H}^{-r}_\theta} \leq C_\Gamma C_P\left| \ga_{n+1}- \ga_n-T\right|^2.
\end{equation}
So, remarking that
\begin{equation}\label{eq:bound delta ga}
\left| \ga_{n+1}- \ga_{n}-T\right| =\left|\mathbf{proj}\left(\mu_{(n+1) T}\right)-\mathbf{proj}\left(\Gamma_{(n+1) \tilde{T}}\right)\right|  \leq C_{\mathbf{proj}}\left\Vert \nu_{N, T}^{ (n)}\right\Vert_{\mathbf{H}^{-r}_{\theta}},
\end{equation}
we obtain, recalling \eqref{eq:bound norm Pst delta Gamma},
\begin{align}
 \Big\Vert P^{st}_{ \ga_{n+1}}\big(\Gamma_{ \ga_{n+1}}&-\Gamma_{ \ga_n+ T} \big)\Big\Vert_{\mathbf{H}^{-r}_\theta}
\leq C_\Gamma C_P C_{\mathbf{proj}}^2 \left( \frac{ 3}{ 8 C_{ P}} h_{ n}\right)^2 \leq \frac{ 1}{ 16} h_{ n},
\end{align}
for $ \varepsilon_{ 0}$ small enough. Finally, for the second term of \eqref{eq:decom Pst nu n+1} we get, using the same estimates,
\begin{align}
\Big\Vert \big( P^{st}_{ \ga_{n+1}}&-P^{st}_{ \ga_n+ T}\big) \nu_{n, T}\Big\Vert_{\mathbf{H}^{-r}_\theta} \leq C_P \left| \ga_{n+1}-  \ga_{n}- T\right| \left\Vert\nu_{n, T}\right\Vert_{\mathbf{H}^{-r}_{\theta}}\nonumber\\
&\leq C_PC_{\mathbf{proj}}\left( \frac{ 3}{ 8 C_{ P}} h_{ n}\right)^2 \leq \frac{ 1}{ 16} h_{ n}.
\end{align}
Gathering the previous estimates gives finally that $ \left\Vert \nu_{ N, 0}^{ (n+1)} \right\Vert_{ \mathbf{ H}_{ \theta}^{ -r}} \leq \frac{ h_{ n}}{ 2}$, which concludes the recursion.

\bigskip

\textbf{Second step: staying close to the manifold.}
We proceed by again by induction to show that for all $n\in \{n_0,\ldots, n_f\}$ we have $ \left\Vert  \nu_{ N, 0}^{ (n)} \right\Vert_{ \mathbf{ H}_{  \theta}^{ -r}}\leq N^{-\frac12+\frac{2\xi}{3}}$ and that $\sup_{t\in [0,T]}\left\Vert  \nu_{ N, t}^{ (n)} \right\Vert_{ \mathbf{ H}_{ \theta}^{ -r}}\leq 2C_\Phi N^{-\frac12+\frac{2\xi}{3}}$. With the same estimates as above this leads to the second part of Proposition~\ref{prop:closeness to M}.  Suppose that for some $n$ we have $ \left\Vert  \nu_{ N, 0}^{ (n)} \right\Vert_{ \mathbf{ H}_{  \theta}^{ -r}}\leq N^{-\frac12+\frac{2\xi}{3}}$  and again that the event $\cC_{\gep',N}$ is satisfied. Define
\begin{equation}
\tilde t_n:= \inf \left\lbrace t\in[0,  T],\ \left\Vert  \nu_{ N, t}^{ (n)} \right\Vert_{ \mathbf{ H}_{ \theta}^{ -r}} \geq 2C_\Phi N^{-\frac12+\frac{2\xi}{3}}\right\rbrace.
\end{equation}
Clearly $\tilde t_n>0$ and \eqref{eq:bound mild nun} implies this time that
\begin{align}
\left\Vert \nu_{N,t}^{ (n)}\right\Vert_{\mathbf{H}^{-r}_{\theta}}&\leq  C_\Phi e^{-\lambda t}N^{-\frac12+\frac{2\xi}{3}}\nonumber\\
&\quad +C_\Phi C_R\int_0^t  \left\lbrace \left(1+ \frac{e^{-\lambda(t-s)}}{\sqrt{t-s}}\right)\left(\frac{1}{N}\left(1+\kappa_5\right) +\gd \frac{ N^{-1+2\xi}}{\left(1+C_\Gamma\right)^2} \right)\right\rbrace {\rm d}s
+N^{-\frac12+\frac{\xi}{3}},
\end{align}
so that $\tilde t_n\geq T$ for $N$ large enough. It remains to prove that $ \left\Vert  \nu_{ N, 0}^{ (n+1)} \right\Vert_{ \mathbf{ H}_{  \theta}^{ -r}}\leq N^{-\frac12+\frac{2\xi}{3}}$. To obtain this estimate, remark that, recalling \eqref{hyp:T 2}, we get for $N$ large enough
\begin{equation}
\left\Vert \nu_{N,T}^{ (n)}\right\Vert_{\mathbf{H}^{-r}_{\theta}}\leq \frac{3}{8C_P}N^{-\frac12+\frac{\xi}{3}},
\end{equation}
and we rely again on \eqref{eq:decom Pst nu n+1} with this time $h_n$ replaced by $N^{-\frac12+\frac{2\xi}{3}}$. This time the three terms can be estimated as follows:
\begin{align}
\left\Vert P^{st}_{ \ga_n+ T}  \nu_{N, T}^{ (n)}\right\Vert_{\mathbf{H}^{-r}_{\theta}}&\leq C_P\left\Vert  \nu_{N, T}^{ (n)}\right\Vert_{\mathbf{H}^{-r}_{\theta}} \leq \frac{ 3}{ 8}N^{-\frac12+\frac{2\xi}{3}},
\end{align}
\begin{align}
 \Big\Vert P^{st}_{ \ga_{n+1}}\big(\Gamma_{ \ga_{n+1}}&-\Gamma_{ \ga_n+ T} \big)\Big\Vert_{\mathbf{H}^{-r}_\theta}
\leq C_\Gamma C_P C_{\mathbf{proj}}^2 \left( \frac{ 3}{ 8 C_{ P}}N^{-\frac12+\frac{2\xi}{3}}\right)^2 ,
\end{align}
and
\begin{align}
\Big\Vert \big( P^{st}_{ \ga_{n+1}}&-P^{st}_{ \ga_n+ T}\big) \nu_{n, T}\Big\Vert_{\mathbf{H}^{-r}_\theta} \leq  C_PC_{\mathbf{proj}}\left( \frac{ 3}{ 8 C_{ P}} N^{-\frac12+\frac{2\xi}{3}}\right)^2 ,
\end{align}
and we deduce that $ \left\Vert  \nu_{ N, 0}^{ (n+1)} \right\Vert_{ \mathbf{ H}_{  \theta}^{ -r}}\leq N^{-\frac12+\frac{2\xi}{3}}$ for $N$ taken large enough. This concludes the proof of Proposition~\ref{prop:closeness to M}.
\end{proof}

\section{Proof of Theorem \ref{th:main}}

Define, for $\mu=(p,m) \in \mathbf{H}^{-r+2}_\theta$,
\begin{equation}
\cU_{\mu} = \Big(\nabla\cdot (\gs^2 \nabla p)+\nabla \cdot(p Kx) -\nabla\cdot \big(p \left(F_{m}-\langle p,F_{m} \rangle\right)\big),\langle p,F_m\rangle\Big),
\end{equation}
and consider the stopping time (recall Proposition~\ref{prop: bound mu1t})
\begin{equation}
\tau_{r} = \inf\left\{t:\, \left\Vert p_{N,t}- g_{ N}\right\Vert_{H_\theta^{-r+2}}> \kappa_1\right\}.
\end{equation}

\begin{lemma}
Suppose the hypotheses of Theorem~\ref{th:main} satisfied. Then, for all $g=(\varphi,\psi)\in \mathbf{H}^{r}_\theta$, we have the 
following identity in $\mathbf{H}^{-r}_\theta$: for $t\in [0, N t_f]$
\begin{multline}
\llangle[\big] \mu_{N,t\wedge \tau_{r} }; g\rrangle[\big]  =\llangle[\big] \mu_{N,0}; g\rrangle[\big]
+\int_0^{t\wedge \tau_{r} } \llangle[\big] \cU_{\mu_{N,s}}; g\rrangle[\big]\dd s 
-\frac1N \int_0^{t\wedge \tau_{r} } \left\langle \nabla \cdot (\gs^2\nabla p_{N,s}),\varphi\right\rangle \dd s\\+ M_{N,t\wedge \tau_{r} }(g),
\end{multline}
where
\begin{equation}
M_{N,t} =\frac{\sqrt{2}}{N}\sum_{i=1}^N\sum_{k=1}^d \gs_k \int_0^t \left(-\partial_{x_k}\left(\gd_{Y_{i,s}}- p_{N,s}\right),e_k\right) \dd B_{i,k,s }
\end{equation}
is a Martingale in $\mathbf{H}^{-r}_\theta$, with tensor quadratic variation $ \llbracket M_{N,\cdot}\rrbracket_t$ given by
\begin{equation}\label{eq:def Mbracket}
\llbracket M_{N,\cdot}\rrbracket_t  =\frac{2}{N^2}\sum_{i=1}^N\sum_{k=1}^d \gs_k^2\int_0^t \left(-\partial_{x_k}\left(\gd_{Y_{i,s}}-p_{N,s}\right),e_k\right)\otimes \left(-\partial_{x_k}\left(\gd_{Y_{i,s}}-p_{N,s}\right),e_k\right)\dd s.
\end{equation}
\end{lemma}

\begin{proof}
It is a direct consequence of the It\^o's Lemma previously applied in \eqref{eq:muN_weak} and \eqref{eq:mN_weak}, remarking that we obtain 
\begin{align}
M_{N,t}(g)&=\frac{\sqrt{2}}{N}\sum_{i=1}^N\int_0^t \nabla \varphi(Y_{i,s})\cdot \gs \left( \dd B_{i,s } -\frac{1}{N}\sum_{j=1}^N \dd B_{j,s}\right)+\frac{\sqrt{2}}{N}\int_0^t \psi \cdot \gs \sum_{j=1}^N \dd B_{j,t}\nonumber\\ 
&=\frac{\sqrt{2}}{N}\sum_{i=1}^N\int_0^t \left( \nabla \varphi({Y_{i,s}})-\langle p_{N,s},\nabla \varphi\rangle+\psi\right)\cdot \gs  \dd B_{i,s }\nonumber\\
&=\frac{\sqrt{2}}{N}\sum_{i=1}^N\sum_{k=1}^d \gs_k \int_0^t\llangle[\big] \left(-\partial_{x_k}\left(\gd_{Y_{i,s}}- p_{N,s}\right),e_k\right);(\varphi,\psi)\rrangle[\big]\,  \dd B_{i,k,s }.
\end{align}
Similar arguments as given in the proofs of Lemma~\ref{lem:VNi} and Lemma~\ref{lem:MNs} show that $M_{N,t}$ is a martingale in $\mathbf{H}^{-r}_\theta$ with continuous trajectories.
\end{proof}

It is well know (see for example \cite{MR3155209}) that the operator$ \mathcal{ L}_{ \theta}^{ \ast}$, defined in \eqref{eq:L_OU}, admits the decomposition, for all $l\in \bbN^d$,
\begin{equation}\label{eq:decomp Ltheta}
\cL_{ \theta}^{ \ast} \psi_{l,\theta} = -\lambda_l \psi_{l,\theta}, \quad \text{with} \quad \lambda_l = \theta\sum_{i=1}^d k_il_i\quad \text{and}\quad  \psi_{l,\theta}(x) = \prod_{i=1}^dh_{l_i}\left(\sqrt{\frac{ \theta k_i}{\gs_i^2}}x_i\right),
\end{equation}
where $h_n$ is the $n^\text{th}$ renormalized Hermite polynomial:
\begin{equation}
h_{ n}(x)= \frac{ \left(-1\right)^{ n}}{ \sqrt{ n!} (2 \pi)^{ \frac{ 1}{ 4}}} e^{ \frac{ x^{ 2}}{ 2}} \frac{ {\rm d}^{ n}}{ {\rm d}x^{ n}} \left\lbrace e^{ - \frac{ x^{ 2}}{ 2}}\right\rbrace.
\end{equation}
The family $( \psi_{ l, \theta})_{l\in \bbN^d}$ is an orthonormal basis of $L^2_\theta$. Defining
\begin{equation}\label{def:polynome}
\psi_{l,r,\theta}=\left(1+\lambda_l\right)^{-\frac{r}{2}}\psi_{l,\theta} =\left(1-\cL^*_\theta\right)^{-\frac{r}{2}}\psi_{l,\theta},
\end{equation}
 we get an orthonormal basis of $H^{r}_\theta$, which means that $\left\{\left(\psi_{l,r,\theta},0\right)\right\}_{l\in\bbN^d}\cup \{(0,e_k)\}_{k=1,\ldots,d}$ is an orthonormal basis of $\mathbf{H}^r_\theta$, and, relying on our ``pivot" space structure, $\left\{\left(u_{l,r,\theta},0\right)\right\}_{l\in\bbN^d}\cup \{(0,e_k)\}_{k=1,\ldots,d}$ is its dual basis in $\mathbf{H}^{-r}_\theta$, where
\begin{equation}
u_{l,r,\theta} = \left(1+\lambda_l\right)^{\frac{r}{2}}w_\theta \,\psi_{l,\theta}.
\end{equation}
We will denote
\begin{equation}
\partial_{u_{l,r,\theta}}\Theta_\mu = D\Theta_\mu \left(u_{l,r,\theta},0\right),\quad \partial_{e_k}\Theta_\mu = D\Theta_\mu \left(0,e_k\right),
\end{equation}
\begin{equation}
\partial^2_{u_{l,r,\theta}\,u_{l',r,\theta}}\Theta_\mu = D^2\Theta_\mu \left(\left(u_{l,r,\theta},0\right),\left(u_{l',r,\theta},0\right)\right),\quad \partial^2_{e_k\,e_{k'}}\Theta_\mu = D^2\Theta_\mu\left(\left(0,e_k\right),\left(0,e_{k'}\right)\right),
\end{equation}
and
\begin{equation}
\partial^2_{u_{l,r,\theta}\,e_k}\Theta_\mu = D^2\Theta_\mu \left(\left(u_{l,r,\theta},0\right),\left(0,e_k\right)\right).
\end{equation}
Remark that we these notations we have
\begin{equation}
D\Theta_{\mu} v=\sum_{l\in \bbN^d}\partial_{u_{l,r,\theta}}\Theta_\mu\,  \llangle[\big] v;(\psi_{l,r,\theta},0)\rrangle[\big]+\sum_{k=1}^d  \partial_{e_k}\Theta_\mu \,\llangle[\big] v;(0,e_k)\rrangle[\big],
\end{equation}
and
\begin{align}\label{eq: decomp D2Theta}
D^2\Theta_{\mu} (v,v')=&\sum_{l,l'\in \bbN^d}\partial^2_{u_{l,r,\theta}\,u_{l',r,\theta}}\Theta_\mu\,  \llangle[\big] v;(\psi_{l,r,\theta},0)\rrangle[\big]\, \llangle[\big] v';(\psi_{l',r,\theta},0)\rrangle[\big] \\
&+\sum_{k=1}^d\sum_{l\in \bbN^d}\partial^2_{u_{l,r,\theta}\,e_k}\Theta_\mu\,  \llangle[\big] v;(\psi_{l,\theta,r},0)\rrangle[\big]\, \llangle[\big] v';(0,e_k)\rrangle[\big] \nonumber\\
&+\sum_{k=1}^d\sum_{l\in \bbN^d}\partial^2_{e_k\, u_{l,r,\theta}}\Theta_\mu\,   \llangle[\big] v;(0,e_k)\rrangle[\big] \, \llangle[\big] v';(\psi_{l,r,\theta},0)\rrangle[\big]\nonumber\\
&+\sum_{k,k'=1}^d  \partial^2_{e_k\,e_{k'}}\Theta_\mu \,\llangle[\big] v;(0,e_k)\rrangle[\big]\, \llangle[\big] v';(0,e_{k'})\rrangle[\big]\nonumber ;
\end{align}
Finally, let us denote
\begin{equation}\label{eq:def psitheta}
\psi^\Theta_{t,r}=\sum_{l\in \bbN^d} \partial_{u_{l,r,\theta}}\Theta_{\Gamma_{t}} \, \psi_{l,r,\theta},
\end{equation}
so that
\begin{equation}
D\Theta_{\Gamma_t}(v,0) = \left\langle v, \psi^\Theta_{t,r}\right\rangle .
\end{equation}

Our aim, in the following proof Theorem~\ref{th:main}, is to show that $\mu_{N,Nt}$ stays close to $\Gamma_{u_0+Nt+v_{N,t}}$, where $v_{N,t}$ converges weakly to a diffusion with constant drift $b$ and constant diffusion coefficient $a^2$ given by
\begin{align}
b=& \frac{1}{T_\gd}  \sum_{l\in \bbN^d}\int_0^{T_\gd}\partial_{u_{l,r,\theta}}\Theta_{\Gamma_{s}}\left\langle \Gamma^1_{s},\nabla  \left(\gs^2\cdot \nabla \psi_{l,r,\theta}\right)\right\rangle \dd s\nonumber\\
& +\frac{2}{T_\gd} \sum_{l,l'\in \bbN^d}\int_0^{T_\gd}\bigg\langle \Gamma^1_{s},\partial^2_{u_{l,r,\theta'}\,u_{l',r,\theta'}}\Theta_{\Gamma_s} \left(\nabla\psi_{l,r,\theta'}-\langle  \Gamma^1_{s},\nabla \psi_{l,r,\theta'}\rangle\right)\nonumber\\
&\qquad \qquad \qquad \qquad \qquad \qquad \qquad \qquad \qquad\qquad \cdot \gs^2 \left(\nabla \psi_{l',r,\theta'}-\langle  \Gamma^1_{s},\nabla \psi_{l',r,\theta'}\rangle\right)\bigg\rangle \dd s\nonumber\\
&\quad +\frac{2}{T_\gd} \sum_{k=1}^d \gs_k^2 \int_0^{T_\gd} \partial^2_{e_k\,e_k}\Theta_{\Gamma_s} \dd s ,\label{eq: formula b}
\end{align}  
and
\begin{equation}
a^2 = \frac{2}{T_\gd}\sum_{k=1}^d \gs_k^2 \int_{0}^{T_\gd}\left\langle\Gamma^1_{s}, \left( \partial_{x_k}\psi^\Theta_{s,r}-\left\langle \Gamma^1_{s},\partial_{x_k}\psi^\Theta_{s,r}\right\rangle\right)^2 \right\rangle \dd s + \frac{2}{T_\gd}\sum_{k=1}^d \gs_k^2 \int_{0}^{T_\gd}\left(\partial_{e_k}\Theta_{\Gamma_{s}}\right)^2 \dd s.\label{eq: formula a2} 
\end{equation}

\begin{proof}[Proof of Theorem~\ref{th:main}]
Let us define the stopping time $\tau_N=(n_\tau T+\tau)\wedge \tau_r$, where $n_\tau$ and $\tau$ are defined in \eqref{eq:def ntau tau}. If $\mu_t$ is a solution to \eqref{eq:syst PDE} we have $\Theta(\mu_t)=\Theta(\mu_0)+t$, and thus, taking the derivative in time, $D\Theta_{\mu_t}\cU_{\mu_t}=1$. So, applying Itô's Lemma in Hilbert spaces, we get
\begin{align}\label{eq:Ito Theta(mu)}
\Theta(\mu_{N,t\wedge \tau_N})=&\Theta(\mu_{N,0})+t\wedge \tau_N-\frac1N\int_0^{t\wedge \tau_N} D_1\Theta_{\mu_{N,s}} \nabla \cdot (\gs^2\nabla p_{N,s})\dd s\\
&+\frac12 \int_0^{t\wedge\tau_N} D^2\Theta_{\mu_{N,s}}\,\dd \llbracket M_{N,\cdot} \rrbracket_s + W_{N,t}\nonumber,
\end{align}
where
\begin{align}
&\int_0^{t\wedge\tau_N} D^2\Theta_{\mu_{N,s}}\,\dd \llbracket M_{N,\cdot} \rrbracket_s\nonumber\\
& = \frac{2}{N^2} \sum_{i=1}^N\sum_{k=1}^d\gs^2_k\int_0^{t\wedge\tau_N} D^2\Theta_{\mu_{N,s}} \left(\left(-\partial_{x_k}\left(\gd_{Y_{i,s}}-p_{N,s}\right),e_k\right), \left(-\partial_{x_k}\left(\gd_{Y_{i,s}}-p_{N,s}\right),e_k\right)\right)\dd s \nonumber\\
& = \frac{2}{N^2} \sum_{i=1}^N\sum_{k=1}^d\gs^2_k\int_0^{t\wedge\tau_N} D^2\Theta_{\mu_{N,s}} \left(\left(\partial_{x_k}\gd_{Y_{i,s}},0\right), \left(\partial_{x_k}\gd_{Y_{i,s}},0\right)\right)\dd s \nonumber\\
&\quad + \frac{2}{N} \sum_{k=1}^d\gs^2_k\int_0^{t\wedge\tau_N} D^2\Theta_{\mu_{N,s}} \left(\left(\partial_{x_k}p_{N,s},0\right), \left(\partial_{x_k} p_{N,s},0\right)\right)\dd s \nonumber\\
&\quad + \frac{2}{N}\sum_{k=1}^d\gs^2_k\int_0^{t\wedge\tau_N} \partial^2_{e_k\, e_k}\Theta_{\mu_{N,s}} \dd s ,
\end{align}
\begin{equation}
W_{N,t\wedge \tau_N} = \frac{\sqrt{2}}{N}\sum_{i=1}^N\sum_{k=1}^d \gs_k\int_0^{t\wedge \tau^N} D\Theta_{\mu_{N,s}}\left(-\partial_{x_k}\left(\gd_{Y_{i,s}}- p_{N,s}\right),e_k\right)  \dd B_{i,k,s },
\end{equation}
and
\begin{equation}
\left[W_{n,\cdot} \right]_{t\wedge \tau^N} =\frac{2}{N^2}\sum_{i=1}^N\sum_{k=1}^d \gs_k^2\int_0^{t\wedge \tau^N} \left| D\Theta_{\mu_{N,s}}\left(-\partial_{x_k}\left(\gd_{Y_{i,s}}- p_{N,s}\right),e_k\right)\right|^2 \dd s.
\end{equation}
\medskip

Our aim is to prove the convergence in probability of the drift and quadratic variation of the process $v_{N,t}=\Theta\left(\mu_{N,(Nt)\wedge \tau^N}\right)-\Theta\left(\mu_{N,0}\right)$, which then implies the desired result via classical arguments (see for example \cite{Billingsley}). Let us place ourselves on the event $\{\tau^N\geq Nt_f\}$. We have, recalling Lemma~\ref{lem:control_delta_x},
\begin{align}
\Big|&\left[W_{n,\cdot} \right]_{nT+t}-\left[W_{n,\cdot} \right]_{nT} \Big|\nonumber\\
&\leq \frac{2}{N^2}\sum_{i=1}^N\sum_{k=1}^d \gs_k^2\int_{nT}^{(n+1)T}\left\Vert D\Theta_{\mu_{N,s}}\right\Vert^2_{\cB\left(\mathbf{H}^{-r}_{\frac{3\theta}{4}}\right)}\left\Vert \left(-\partial_{x_k}\left(\gd_{Y_{i,s}}- p_{N,s}\right),e_k\right)\right\Vert^2_{\mathbf{H}^{-r}_{\frac{3\theta}{4}}}\dd s \nonumber\\
&\leq \frac{C_1}{N^2}\sum_{i=1}^N\sum_{k=1}^d \gs_k^2\int_{nT}^{(n+1)T}\left(1+w_{-\frac{\theta}{2}}^2\left(Y_{i,s}\right)+\left\Vert p_{N,s}\right\Vert^2_{H^{-(r+1)}_{\frac{3\theta}{4}}}\right)\dd s\nonumber\\
&\leq \frac{C_2}{N}\int_{nT}^{(n+1)T}\left(1+\left\langle p_{N,s},w_{-\theta}\right\rangle+\left\Vert p_{N,s}\right\Vert^2_{H^{-(r+1)}_{\theta}}\right) \dd s.
\end{align}
where we used Lemma~\ref{lem:comparison H theta theta'}. So, recalling Proposition~\ref{prop: bound mu1t} and remarking that $w_{-\theta}\in H^{r}_\theta$, we have with high probability
\begin{equation}
\sup_{n\in\{0,\ldots, n_f\}} \sup_{t\in [0,T]}\Big|\left[W_{n,\cdot} \right]_{nT+t}-\left[W_{n,\cdot} \right]_{nT} \Big|\leq \frac{C_3}{N}.
\end{equation}
Similarly, we get
\begin{align}
\Bigg|\int_{nT}^{nT+t}& D^2\Theta_{\mu_{N,s}}\,\dd \llbracket M_{N,\cdot} \rrbracket_s\Bigg|\nonumber\\
&\leq \frac{2}{N^2} \sum_{i=1}^N\sum_{k=1}^d\gs^2_k\int_{nT}^{nT+t} \left\Vert D^2\Theta_{\mu_{N,s}} \right\Vert_{\cB\cL\left(\mathbf{H}^{-r}_{\frac{3\theta}{4}}\right)}\left\Vert \left(-\partial_{x_k}\left(\gd_{Y_{i,s}}-p_{N,s}\right),e_k\right)\right\Vert_{\mathbf{H}^{-r}_{\frac{3\theta}{4}}}^2\dd s\nonumber\\
&\leq \frac{C_4}{N}\int_{nT}^{(n+1)T}\left(1+\left\langle p_{N,s},w_{-\theta}\right\rangle+\left\Vert p_{N,s}\right\Vert^2_{H^{-(r+1)}_{\theta}}\right) \dd s,
\end{align}
and
\begin{equation}
\Bigg|\frac1N\int_{nT}^{nT+t} D_1\Theta_{\mu_{N,s}} \nabla \cdot (\gs^2\nabla p_{N,s})\dd s\Bigg|\leq \frac{C_5}{N}\int_{nT}^{(n+1)T}\left\Vert p_{N,s}\right\Vert^2_{H^{-(r+2)}_{\theta}} \dd s,
\end{equation}
so that, with high probability,
\begin{equation}
\sup_{n\in\{0,\ldots, n_f\}} \sup_{t\in [0,T]} \max \Bigg|\int_{nT}^{nT+t} D^2\Theta_{\mu_{N,s}}\,\dd \llbracket M_{N,\cdot} \rrbracket_s\Bigg| \leq \frac{C_6}{N},
\end{equation}
\begin{equation}
\sup_{n\in\{0,\ldots, n_f\}} \sup_{t\in [0,T]} \max \Bigg|\frac1N\int_{nT}^{(n+1)T} D_1\Theta_{\mu_{N,s}} \nabla \cdot (\gs^2\nabla p_{N,s})\dd s\Bigg| \leq \frac{C_7}{N}.
\end{equation}
From these estimates we deduce that the $n_0$ first time steps are negligible at the time scale $Nt$, and that it is sufficient to study the limit of the drift and quadratic variation at the time steps $nT$.

Now remark that Proposition~\ref{prop:closeness to M} imply that with high probability (recall that $T=k_T T_\gd$)
\begin{align}
\frac1N&\int_{nT}^{(n+1)T} D_1\Theta_{\mu_{N,s}} \nabla \cdot (\gs^2\nabla p_{N,s})\dd s \nonumber\\
&= \frac1N\int_{0}^{T} D_1\Theta_{\Gamma_{\Theta\left(\mu_{N,nT}\right)+s}} \nabla \cdot \left(\gs^2\nabla \Gamma^1_{\Theta\left(\mu_{N,nT}\right)+s}\right)\dd s+O\left(N^{-\frac32+\xi}\right)\nonumber\\
&= \frac{k_T}{N} \int_{0}^{T_\gd} D_1\Theta_{\Gamma_{s}}\nabla \cdot \left(\gs^2\nabla \Gamma^1_{s}\right)\dd s+O\left(N^{-\frac32+\xi}\right),
\end{align}
so the last term of the first line of \eqref{eq:Ito Theta(mu)} gives a constant drift $b_1$ in the limit at the time scale $Nt$, given by
\begin{align}
b_1 :&=\frac{1}{T_\gd} \int_{0}^{T_\gd} D_1\Theta_{\Gamma_{s}}\nabla \cdot \left(\gs^2\nabla \Gamma^1_{s}\right)\dd s=\frac{1}{T_\gd}  \sum_{l\in \bbN^d}\int_0^{T_\gd}\partial_{u_{l,r,\theta}}\Theta_{\Gamma_{s}}\left\langle \Gamma^1_{s},\nabla  \left(\gs^2\cdot \nabla \psi_{l,r,\theta}\right)\right\rangle \dd s.
\end{align}
Moreover we get, relying again on Proposition~\ref{prop:closeness to M}, with high probability (recall \eqref{eq:def psitheta})
\begin{align}
&\left[W_{n,\cdot} \right]_{(n+1)T}-\left[W_{n,\cdot} \right]_{nT} \nonumber\\
&=\frac{2}{N^2}\sum_{i=1}^N\sum_{k=1}^d \gs_k^2\int_{0}^{T} \left| D\Theta_{\Gamma_{\Theta\left(\mu_{N,nT}\right)+s}}\left(-\partial_{x_k}\left(\gd_{Y_{i,nT+s}}- \Gamma^1_{\Theta\left(\mu_{N,nT}\right)+s}\right),e_k\right)\right|^2 \dd s\nonumber\\
&\quad  +O\left(N^{-\frac32+\xi}\right)\nonumber\\
& =\frac{2}{N^2}\sum_{i=1}^N\sum_{k=1}^d \gs_k^2 \int_{0}^{T}\left( \partial_{x_k}\psi^\Theta_{\Theta\left(\mu_{N,nT}\right)+s,r}\left(Y_{i,nT+s}\right)-\left\langle \Gamma^1_{\Theta\left(\mu_{N,nT}\right)+s},\partial_{x_k}\psi^\Theta_{\Theta\left(\mu_{N,nT}\right)+s,r}\right\rangle\right)^2 \dd s\nonumber\\
&\quad + \frac{2}{N}\sum_{k=1}^d \gs_k^2 \int_{0}^{T}\left(\partial_{e_k}\Theta_{\Gamma_{\Theta\left(\mu_{N,nT}\right)+s}}\right)^2 \dd s+O\left(N^{-\frac32+\xi}\right)\nonumber\\
& =\frac{2k_T}{N}\sum_{k=1}^d \gs_k^2 \int_{0}^{T_\gd}\left\langle\Gamma^1_{s}, \left( \partial_{x_k}\psi^\Theta_{s,r}-\left\langle \Gamma^1_{s},\partial_{x_k}\psi^\Theta_{s,r}\right\rangle\right)^2 \right\rangle \dd s + \frac{2k_T}{N}\sum_{k=1}^d \gs_k^2 \int_{0}^{T_\gd}\left(\partial_{e_k}\Theta_{\Gamma_{s}}\right)^2 \dd s\nonumber\\
&\quad +O\left(N^{-\frac32+\xi}\right).
\end{align}
So the quadratic variation is given in the limit on the timescale $Nt$ by a constant diffusion with coefficient $a^2$ defined by
\begin{equation}
a^2=\frac{2}{T_\gd}\sum_{k=1}^d \gs_k^2 \int_{0}^{T_\gd}\left\langle\Gamma^1_{s}, \left( \partial_{x_k}\psi^\Theta_{s,r}-\left\langle \Gamma^1_{s},\partial_{x_k}\psi^\Theta_{s,r}\right\rangle\right)^2 \right\rangle \dd s + \frac{2}{T_\gd}\sum_{k=1}^d \gs_k^2 \int_{0}^{T_\gd}\left(\partial_{e_k}\Theta_{\Gamma_{s}}\right)^2 \dd s.
\end{equation}
For the last term, since from Theorem~\ref{th:Theta} we have, for $s\in [0,T]$,
\begin{equation}
\left\Vert D^2\Theta_{\mu_{N,nT+s}} - D^2\Theta_{\Gamma^\gd_{\Theta(\mu_{N,nT+s})}}\right\Vert_{\cB\cL (\mathbf{H}^{-r}_\theta)}\leq C_{\Theta,\gd} \left\Vert \mu_{N,nT+s} - \Gamma^\gd_{\Theta(\mu_{N,nT+s})}\right\Vert_{\mathbf{H}^{-r}_\theta},
\end{equation}
and since, from the calculation above, $\Theta\left(\mu_{N,nT+s}\right)-\Theta\left(\mu_{N,nT}\right)-s=O\left(N^{-1}\right)$ with high probability, we deduce
\begin{equation}
\left\Vert D^2\Theta_{\mu_{N,nT+s}} - D^2\Theta_{\Gamma^\gd_{\Theta(\mu_{N,nT})+s}}\right\Vert_{\cB\cL (\mathbf{H}^{-r}_\theta)}= O\left(N^{-\frac32+\xi}\right).
\end{equation}
So, with similar calculations as above we obtain, with high probability,
\begin{align}
&\frac12 \int_{nT}^{(n+1)T} D^2\Theta_{\mu_{N,s}}\,\dd \llbracket M_{N,\cdot} \rrbracket_s \nonumber\\
&= \frac{2}{N^2} \sum_{i=1}^N\sum_{k=1}^d\gs^2_k\int_{0}^{T} D^2\Theta_{\Gamma_{\Theta\left(\mu_{N,nT}\right)+s}} \left(\left(\partial_{x_k}\gd_{Y_{i,nT+s}},0\right), \left(\partial_{x_k}\gd_{Y_{i,nT+s}},0\right)\right)\dd s \nonumber\\
& \quad + \frac{2}{N} \sum_{k=1}^d\gs^2_k\int_{0}^{T}D^2\Theta_{\Gamma_{\Theta\left(\mu_{N,nT}\right)+s}} \left(\left(\partial_{x_k}\Gamma^1_{\Theta\left(\mu_{N,nT}\right)+s},0\right), \left(\partial_{x_k}\Gamma^1_{\Theta\left(\mu_{N,nT}\right)+s},0\right)\right)\dd s\nonumber\\
&\quad + \frac{2}{N}\sum_{k=1}^d\gs^2_k\int_{0}^{T} \partial^2_{e_k\, e_k}\Theta_{\Gamma_{\Theta\left(\mu_{N,nT}\right)+s}}\dd s  +O\left(N^{-\frac32+\xi}\right).
\end{align}
Now, for all $l,l'\in \bbN^d$, on an event of high probability, we have, relying on Proposition~\ref{prop:closeness to M},
\begin{align}
\Bigg|\frac{1}{N}\sum_{i=1}^N \llangle[\big] \left(\partial_{x_k}\gd_{Y_{i,nT+s}},0\right)&;\left(\psi_{l,r',\theta'},0\right)\rrangle[\big] \llangle[\big] \left(\partial_{x_k}\gd_{Y_{i,nT+s}},0\right);\left(\psi_{l',r',\theta'},0\right)\rrangle[\big] \nonumber\\
&\qquad\qquad\qquad\qquad \qquad - \left\langle \Gamma^1_{\Theta\left(\mu_{N,nT}\right)+s}, \partial_{x_k} \psi_{l,r',\theta'}\, \partial_{x_k} \psi_{l',r',\theta'} \right\rangle \Bigg|\nonumber\\
& =\left|\left\langle p_{N,nT+s}- \Gamma^1_{\Theta\left(\mu_{N,nT}\right)+s}, \partial_{x_k} \psi_{l,r',\theta'}\, \partial_{x_k} \psi_{l',r',\theta'} \right\rangle\right|\nonumber \\
&\leq N^{-\frac12 +\xi} \left\Vert  \partial_{x_k} \psi_{l,r',\theta'}\, \partial_{x_k} \psi_{l',r',\theta'}\right\Vert_{H^{r}_{\theta}},
\end{align}
and from Lemma~\ref{lem:estimate psil psil'} and Lemma~\ref{lem:comparison H theta theta'} we get, taking $\theta'\leq \frac{\theta}{2}$ and $r'> 2d+\lfloor r\rfloor +2$,
\begin{equation}
\sum_{l,l'\in \bbN^d}\left\Vert  \partial_{x_k} \psi_{l,r',\theta'}\, \partial_{x_k} \psi_{l',r',\theta'}\right\Vert_{H^{r}_{\theta}}\leq C_6\sum_{l,l'\in \bbN^d}\left(1+|l|\right)^{\frac{\lfloor r\rfloor -r'+2}{2}} \left(1+|l'|\right)^{\frac{\lfloor r\rfloor -r'+2}{2}}\leq C_7.
\end{equation}
We deduce:
\begin{align}
&\Bigg|\frac{2}{N^2} \sum_{i=1}^N\int_{0}^{T} D^2\Theta_{\Gamma_{\Theta\left(\mu_{N,nT}\right)+s}} \left(\left(\partial_{x_k}\gd_{Y_{i,nT+s}},0\right), \left(\partial_{x_k}\gd_{Y_{i,nT+s}},0\right)\right)\dd s\nonumber\\
&-\frac{2}{N}\sum_{l,l'\in \bbN^d}  \int_{0}^{T}\partial^2_{u_{l,r',\theta'}\, u_{l',r',\theta'}}\Theta_{\Gamma_{\Theta\left(\mu_{N,nT}\right)+s}} \left\langle \Gamma^1_{\Theta\left(\mu_{N,nT}\right)+s}, \partial_{x_k} \psi_{l,r',\theta'}\, \partial_{x_k} \psi_{l',r',\theta'} \right\rangle \dd s\Bigg|\nonumber\\
&\quad \leq C_7 \sup_{s\in [0,T_c]} \left\Vert  D^2\Theta_{\Gamma_s}\right\Vert_{\cB\cL\left(\mathbf{H}^{-r'}_{\theta'}\right)} N^{-\frac32+\xi}.
\end{align}
So, remarking that
\begin{equation}
\partial^2_{u_{l,r,\theta'}\,u_{l',r,\theta'}}\Theta_\mu =\left(1+\lambda_l\right)^\frac{r-r'}{2}\left(1+\lambda_{l'}\right)^\frac{r-r'}{2}\partial^2_{u_{l,r',\theta'}\,u_{l',r',\theta'}}\Theta_\mu
\end{equation}
one can replace $r'$ with $r$, and
the first term of the second line of \eqref{eq:Ito Theta(mu)} gives a constant drift $b_2$ in the limit at the time scale $Nt$, given by
\begin{align}
b_2:& = \frac{2}{T_\gd} \sum_{l,l'\in \bbN^d}\int_0^{T_\gd}\bigg\langle \Gamma^1_{s},\partial^2_{u_{l,r,\theta'}\,u_{l',r,\theta'}}\Theta_{\Gamma_s} \left(\nabla\psi_{l,r,\theta'}-\langle  \Gamma^1_{s},\nabla \psi_{l,r,\theta'}\rangle\right)\nonumber\\
&\qquad \qquad \qquad \qquad \qquad \qquad \qquad \qquad \qquad\qquad \cdot \gs^2 \left(\nabla \psi_{l',r,\theta'}-\langle  \Gamma^1_{s},\nabla \psi_{l',r,\theta'}\rangle\right)\bigg\rangle \dd s\nonumber\\
&\quad +\frac{2}{T_\gd} \sum_{k=1}^d \gs_k^2 \int_0^{T_\gd} \partial^2_{e_k\,e_k}\Theta_{\Gamma_s} \dd s .
\end{align}
\end{proof}

\appendix

\section{Weighted Sobolev norms and Hermite polynomials}\label{app:weighted sobolev}

We begin this appendix with the proof of Lemma~\ref{lem:control_delta_x}. 
\begin{proof}[Proof of Lemma~\ref{lem:control_delta_x}]
Let $h\in H_{ \theta}^{ r}$. Then, for all $ \eta>0$, since $ \frac{ 2}{ 4- \eta}> \frac{ 1}{ 2}$, the function $f_{ \eta}(x):= h(x) \exp \left( - \frac{ \theta \left\vert x \right\vert_{ K \sigma^{ -2}}^{ 2}}{ 4- \eta}\right)$ belongs to $H^{ r}$ (without weight, that is $w \equiv 1$) and there exists a constant $ C_{1, \eta}>0$ (independent of $h$, but such that $C_{1, \eta} \xrightarrow[ \eta\to 0]{}+\infty$) such that $ \left\Vert f_{ \eta} \right\Vert_{ H^{ r}}\leq C_{ 1,\eta} \left\Vert h \right\Vert_{ H^{ r}_{ \theta}}$. Using now that when $r>d/2$, $H^{ r}$ continuously injects into $L^{ \infty}$ (that is, in particular $ \left\Vert f_{ \eta} \right\Vert_{ \infty} \leq C_{ 2, r} \left\Vert f_{ \eta} \right\Vert_{ H^{ r}}$ for some constant $C_{ 2, r}$ independent of $ \eta$ (see \cite{Adams2003}, Th. 4.2), one obtains: 
\begin{align*}
\left\vert \left\langle \delta_{ x}\, ,\, h\right\rangle \right\vert= \left\vert h(x) \right\vert&= \left\vert f_{ \eta}(x) \right\vert  \exp \left(\frac{ \theta \left\vert x \right\vert_{ K \sigma^{ -2}}^{ 2}}{ 4- \eta}\right) \leq \left\Vert f_{ \eta} \right\Vert_{ \infty}  \exp \left(\frac{ \theta \left\vert x \right\vert_{ K \sigma^{ -2}}^{ 2}}{ 4- \eta}\right),\\
& \leq C_{ 2, r} C_{ 1,\eta} \left\Vert h \right\Vert_{ H^{ r}_{ \theta}}\exp \left(\frac{ \theta \left\vert x \right\vert_{ K \sigma^{ -2}}^{ 2}}{ 4- \eta}\right),
\end{align*} 
which gives the result.
\end{proof}

We now give two Lemma that are useful in the proof of Theorem~\ref{th:main}.
\begin{lemma}\label{lem:comparison H theta theta'}
There exists $r_0>0$ such that for all $r\geq r_0$ and $0<\theta'\leq \theta\leq 1$ there exists a constant $c_{r,\theta,\theta'}$ such that
\begin{equation}
\left\Vert u\right\Vert_{H^{-r}_{\theta'}}\leq c_{r,\theta,\theta'}\left\Vert u\right\Vert_{H^{-r}_{\theta}}.
\end{equation}
\end{lemma}

\begin{proof}
For all $\psi\in H^{r}_{\theta'}$ we have
\begin{equation}
\left\Vert \psi\right\Vert^2_{H^{r}_{\theta}} = \int \left(\left(1-\cL_\theta\right)^{\frac{r}{2}}\psi\right)^2w_\theta \leq \left\Vert \left(1-\cL_\theta\right)^{\frac{r}{2}}\psi\right\Vert_{L^2_{\theta'}}.
\end{equation}
Moreover it was proved in \cite{LP2021a}, proof of Lemma A.2., that the norm $\left\Vert\left(1-\cL_\theta\right)^{\frac{r}{2}}\cdot \right\Vert_{L^2_{\theta'}}$ is equivalent to the norm $\left\Vert \cdot\right\Vert_{H^{r}_{\theta'}}$, so there exists a constant $c_{r,\theta,\theta'}$ such that 
\begin{equation}
\left\Vert \psi\right\Vert^2_{H^{r}_{\theta}}\leq c_{r,\theta,\theta'}\left\Vert \psi\right\Vert^2_{H^{r}_{\theta'}}.
\end{equation}
So we deduce
\begin{equation}
\left|\langle u,\psi\rangle \right|\leq \left\Vert u\right\Vert_{H^{-r}_{\theta}}\left\Vert \psi\right\Vert^2_{H^{r}_{\theta}} \leq c_{r,\theta,\theta'}\left\Vert u\right\Vert_{H^{-r}_{\theta}}\left\Vert \psi\right\Vert^2_{H^{r}_{\theta'}},
\end{equation}
which implies the result.
\end{proof}

\begin{lemma}\label{lem:estimate psil psil'}
For all $r'\geq 0$, $\theta,\theta'\in (0,1]$ such that $2\theta'\leq \theta$ and all integer $j\geq 0$ there exists a constant $C_{j,r,\theta,\theta'}$ such that (recall \eqref{def:polynome})
\begin{equation}
\left\Vert  \partial_{x_k} \psi_{l,r',\theta'}\, \partial_{x_k} \psi_{l',r',\theta'}\right\Vert_{H^{j}_{\theta}}\leq C_{j,r,\theta,\theta'}\frac{|l|^\frac{j}{2} +|l'|^\frac{j}{2}}{\left(1+|l|\right)^{\frac{r'-1}{2}}\left(1+|l'|\right)^{\frac{r'-1}{2}}}.
\end{equation}
\end{lemma}

\begin{proof}
Recalling \eqref{eq:decomp Ltheta}, the estimates on Hermite polynomials of \cite{Bonan:1990} imply that there is a constant $C_{\theta'}$ such that for all $l\in \bbN^d$,
\begin{equation}
\psi_{l,\theta}^2 w_{\theta'}\leq C_{\theta'}.
\end{equation} 
Moreover, since $h'_n(x) = \sqrt{n}h_{n-1}(x)$ and $x h_{n-1}(x) =\sqrt{n} h_n(x)+\sqrt{n-1}h_{n-1}(x)$, for all integer $j\geq 0$ there exists a constant $C_{j,\theta,\theta'}$ such that 
\begin{equation}
\left|\left(1-\cL_\theta\right)^j \left( \partial_{x_k} \psi_{l,\theta'}\, \partial_{x_k} \psi_{l',\theta'} \right)\right| w_{\theta'} \leq C_{j,\theta,\theta'} |l|^\frac12\, |l'|^\frac12 \left( |l|^j +|l'|^j\right).
\end{equation}
Then we have
\begin{align}
\left\Vert  \partial_{x_k} \psi_{l,\theta'}\, \partial_{x_k} \psi_{l',\theta'}\right\Vert_{H^{j}_{\theta}}^2 &\leq \int \left| \partial_{x_k} \psi_{l,\theta'}\, \partial_{x_k} \psi_{l',\theta'}\right|\left|\left(1-\cL_\theta\right)^j  \left(\partial_{x_k} \psi_{l,\theta'}\, \partial_{x_k} \psi_{l',\theta'} \right) \right|w_{\theta}\nonumber\\
&\leq C_{j,\theta,\theta'} |l|^\frac12\, |l'|^\frac12 \left( |l|^j +|l'|^j\right)
\int  \left| \partial_{x_k} \psi_{l,\theta'}\, \partial_{x_k} \psi_{l',\theta'}\right|w_{\theta-\theta'}\nonumber\\
&\leq C_{j,\theta,\theta'} |l|\, |l'|\left( |l|^j +|l'|^j\right)
  \left\Vert \psi_{l,\theta'}\right\Vert_{L^2_{\theta'}} \left\Vert \psi_{l',\theta'}\right\Vert_{L^2_{\theta'}}\nonumber\\
&\leq C_{j,\theta,\theta'} |l|\, |l'|\left( |l|^j +|l'|^j\right),
\end{align}
and thus
\begin{equation}
\left\Vert  \partial_{x_k} \psi_{l,r',\theta'}\, \partial_{x_k} \psi_{l',r',\theta'}\right\Vert_{H^{j}_{\theta}}^2 \leq C_{j,\theta,\theta'}\frac{|l|\, |l'|\left( |l|^j +|l'|^j\right)}{\left(1+\lambda_l\right)^{r'}\left(1+\lambda_{l'}\right)^{r'}},
\end{equation}
which implies the result.
\end{proof}

\section*{Acknowledgements}

Both authors benefited from the support of the ANR–19–CE40–0023 (PERISTOCH),
C. Poquet from the ANR-17-CE40-0030 (Entropy, Flows, Inequalities), E. Lu\c con from the ANR-19-CE40-002 (ChaMaNe).

\end{document}